\documentclass[final,onefignum,onetabnum,a4paper]{siamltex} %% pdflatex
\usepackage{amsmath,amssymb,amsbsy,color,url,epsfig,graphicx,pifont,setspace,placeins}

  \usepackage{paralist}

\usepackage{appendix}

\usepackage{comment}

\usepackage{xspace}

  \usepackage[pdfstartview=,colorlinks=true,pdfmenubar=false,bookmarks=false]{hyperref}
  \hypersetup{urlcolor=blue, citecolor=red}
\renewcommand{\epsilon}{\varepsilon}
\renewcommand{\leq}{\leqslant}
\renewcommand{\geq}{\geqslant}

\newcommand{\eq}[1]{\eqref{#1}}

\newcommand{\fref}[1]{Fig.~\ref{#1}}
\newcommand{\frefs}[1]{Figs.~\ref{#1}}
\newcommand{\Fref}[1]{Figure~\ref{#1}}

\newcommand{\sref}[1]{Sec.~\ref{#1}}
\newcommand{\srefs}[1]{Secs.~\ref{#1}}

\newcommand{\ph}{\phantom}
\newcommand{\tr}[1]{\ignorespaces}

\newcommand{\be}{\begin{equation}}
\newcommand{\ee}{\end{equation}}

\newcommand\rdot{\dot{r}}
\newcommand\phidot{\dot{\phi}}
\newcommand\xdot{\dot{x}}
\newcommand\ydot{\dot{y}}

\newcommand\vdot{\dot{v}}
\newcommand\cA{\mathcal{A}}

\newcommand\cF{\mathcal{F}}
\newcommand\cL{\mathcal{L}}
\newcommand\cM{\mathcal{M}}

\newcommand\cO{\mathcal{O}}
\newcommand\cP{\mathcal{P}}
\newcommand\cR{\mathcal{R}}
\newcommand\C{\mathbb{C}}
\newcommand\N{\mathbb{N}}
\newcommand\R{\mathbb{R}}

\newcommand{\HH}{\textit{HH}}
\newcommand{\SL}{\textit{SL}\xspace}
\newcommand{\PD}{\textit{PD}\xspace}

\newcommand{\NormalForm}{\texttt{HHnfDDE}\xspace}

\newtheorem{conjecture}[theorem]{Conjecture}
\newtheorem{propose}[theorem]{Proposition}

\renewcommand\Re{\mathrm{Re}\,}
\renewcommand\Im{\mathrm{Im}\,}
\newcommand\gtb{\overline{\widetilde{g}}\mbox{}}

\newcommand{\email}[1]{\protect\href{mailto:#1}{#1}}

\title{Resonance phenomena in a scalar delay differential equation with two state-dependent delays}

\author{R.C. Calleja\thanks{Depto.\ Matem\'aticas y Mec\'anica, IIMAS, Universidad Nacional Aut\'onoma
         de M\'exico, 01000 M\'exico. (\email{calleja@mym.iimas.unam.mx})}
        \and A.R. Humphries\thanks{Departments of Mathematics \&
        Statistics, and, Physiology, McGill University, Montreal, Quebec H3A 0B9, Canada
        (\email{Tony.Humphries@mcgill.ca})}
        \and B. Krauskopf\thanks{Department of Mathematics, University of Auckland, Auckland 1142, New Zealand
        (\email{b.krauskopf@auckland.ac.nz})}}

\begin{document}
\maketitle

\begin{center}
\today
\end{center}

\newcommand{\slugmaster}{%
\slugger{CIA's}{xxxx}{xx}{x}{x--x}}%slugger should be set to juq, siads, sifin, or siims

%%%%%%%%%%%%%%%%%%%%%%%%%%%%%%%%%%%%%%%%%%%%%%%%%%%
\begin{abstract}
We study a scalar DDE with two delayed feedback terms that depend
linearly on the state. The associated constant-delay DDE, obtained by
freezing the state dependence, is linear and without recurrent
dynamics.  With state dependent delay terms, on the other hand, the
DDE shows very complicated dynamics. To investigate this, we perform a
bifurcation analysis of the system and present its bifurcation diagram
in the plane of the two feedback strengths. It is organized by
Hopf-Hopf bifurcation points that give rise to curves of torus
bifurcation and associated two-frequency dynamics in the form of
invariant tori and resonance tongues. We
numerically determine the type of the
Hopf-Hopf bifurcation points by computing the normal form on the
center manifold; this requires the expansion of the functional
defining the state-dependent DDE in a power series whose terms up to
order three only contain constant delays.  We implemented this
expansion and the computation of the normal form coefficients in
Matlab using symbolic differentiation, and the resulting code
\NormalForm is supplied as a supplement to this article.  Numerical
continuation of the torus bifurcation curves confirms the correctness
of our normal form calculations. Moreover, it enables us to compute
the curves of torus bifurcations more globally, and to find associated
curves of saddle-node bifurcations of periodic orbits that bound the
resonance tongues. The tori themselves are computed and visualized in
a three-dimensional projection, as well as the planar trace of a
suitable Poincar{\'e} section. In particular, we compute periodic
orbits on locked tori and their associated unstable manifolds (when
there is a single unstable Floquet multiplier). This allows us to
study transitions through resonance tongues and the breakup of a
$1\!:\!4$ locked torus. The work presented here demonstrates that state dependence alone is
capable of generating a wealth of dynamical phenomena.
\end{abstract}

%%%%%%%%%%%%%%%%%%%%%%%%%%%%%%%%%%%%%%%%%%%%%%%%%%%

\begin{keywords} State-dependent delay differential equations,
  bifurcation analysis, invariant tori, resonance tongues, Hopf-Hopf
  bifurcation, normal form computation
\end{keywords}

\begin{AMS}
                %%  34Kxx Functional-differential and
                %%  differential-difference equations
34K60,      %% Qualitative investigation and simulation of models
34K18,      %% Bifurcation theory
                %% 37Gxx	Local and nonlocal bifurcation theory
37G05,      %% Normal forms
                %% 37Mxx	Approximation methods and numerical
                %% treatment of dynamical systems
37M20      %% Computational methods for bifurcation problems
\end{AMS}

%%%%%%%%%%%%%%%%%%%%%%%%%%%%%%%%%%%%%%%%%%%%%%%%%%%%%%%%%%%%%%%%%%%%%%%%%%%%%%%%%%%%%%%%%%%%%%%%%%%%%%%%%%%

%%%%%%%%%%%%%%%%%%%%%%%%%%%%%%%%%%%%%%%%%%%%%%%%%%%
\section{Introduction}
\label{sec:intro}

Time delays arise naturally in numerous areas of application as an
unavoidable phenomenon, for example, in balancing and control
\cite{bps,fied,ims,just1,mtk,post1,ppk2,pyragas}, machining
\cite{IST07}, laser physics \cite{kane,kralen,luedge}, agent dynamics
\cite{kbs,leblanc,scholl,san}, neuroscience and biology
\cite{AFW92,fcwz,flmm,KCP14,Wal-Gui-Hum-13}, and climate modelling
\cite{dijkstra,kaper,ks}. Important sources of delays are
communication times between components of a system, maturation and
reaction times, and the processing time of information received. When
they are sufficiently large compared to the relevant internal time
scales of the system under consideration, then the delays must be
incorporated into its mathematical description. This leads to
mathematical models in the form of delay differential equations
(DDEs). In many situations the relevant delays can be considered to be
fixed; examples are the travel time of light between components of a
laser system and machining with rotating tools.

There is a well established theory of DDEs with a finite number of
constant delays as infinite dimensional dynamical systems; see, for
example, \cite{Bel-Coo-63,Hale77,Hal-Lun-93,DGVLW95,Smith11,stepan}.
Usually the phase space of the dynamical system is taken to be
$C=C\bigl([-\tau,0],\R^d\bigr)$, the Banach space of continuous
functions mapping $[-\tau,0]$ to $\R^d$, where $d$ is the number of
variables and $\tau$ is the largest of the delays. The DDE can then be
written as a retarded functional differential equation
\be \label{eq:rfde} u'(t)=F(u_t), %\dot{u}(t)=F(u_t),
\ee where $F:C\to\R^d$ and $u_t\in C$ for each $t\geq0$ is the
function \be \label{eq:ut} u_t(\theta)=u(t+\theta), \quad
\theta\in[-\tau,0].  \ee In other words, an initial condition consists
of a function over the time interval from the (maximal) delay $\tau$
ago up to time $0$, which (under appropriate mild assumptions)
determines the solution for all time $t>0$. In fact, solutions of
constant-delay DDEs depend smoothly on their initial conditions, and
linearizations at equilibria and periodic solutions have at most
finitely many unstable eigen-directions. As a consequence, bifurcation
theory for this class of DDEs is analogous to that for ordinary
differential equations (ODEs), and one finds the same types of
bifurcations. In particular, center manifold and normal form methods
allow for the local reduction of the DDE to an ODE describing the
dynamics near a bifurcation point of interest.  Moreover, advanced
numerical tools for simulation and bifurcation analysis of DDEs with
constant delays have become available in recent years
%\cite{Bel-Zen-03,BZMG09,Eng-Luz-Roo-02,Kra-Gre-03,NewDDEBiftool,sr}
\cite{Bel-Zen-03,BZMG09,Bre-Die-Gyl-Sca-Ver-16,Eng-Luz-Roo-02,Kra-Gre-03,NewDDEBiftool,sr}.
These
theoretical and numerical tools have been applied very successfully in
many application areas, including those mentioned above.

It is very important to realise that treating the delays that arise as
constant is a modelling assumption that must be justified. This can be
argued successfully, for example, in machining when the tool has
nearly infinite stiffness perpendicular to the cutting direction
\cite{stepan}, or in laser dynamics where light travels over a fixed
distance \cite{kane}. On the other hand, in many contexts, including
in biological systems and in control problems
%\cite{DGHP10,IST07,Jes-Cam-10,pyragas2,yb}
\cite{Cra-Hum-Mac-16,DGHP10,Die-Gyl-etal-10,Get-Wau-16,IST07,Jes-Cam-10,pyragas2,yb},
the delays one encounters
are not actually constant. In particular, they may depend on the state
in a significant way, that is, change dynamically during the
time-evolution of the system.

DDEs with state-dependent delays have been an active area of research
in recent years.  Many parts of the general theory of DDEs with
constant delays have been extended to also cover state-dependent DDEs,
where $\tau$ is now a global bound on the maximal possible delay; see
\cite{Har-Kri-Wal-Wu-06} and the discussion in
\cite{Hum-Dem-Mag-Uph-12}. However, the mathematical theory is
considerably more complicated and as yet incomplete.  Solutions of
state-dependent DDEs do not depend smoothly on initial conditions or
parameters unless extra assumptions are made on the initial conditions
\cite{Har-16}, and this dramatically complicates arguments around key
concepts, requiring new theory and proofs for asymptotics, the initial
value problem, bifurcations, and invariant manifolds. Indeed, these
important elements of the theory have been addressed only recently
\cite{Har-Kri-Wal-Wu-06,Hu-Wu-10,Krisztin03,
  JMPRN11b,JMPRN11,Sieber12,Walther04,Walther14}.
Similarly, the numerical bifurcation analysis of state-dependent DDEs
is more involved. Recent developments include approaches for the
continuation of solutions
and bifurcations for state dependent delay equations
\cite{Hum-Dem-Mag-Uph-12,NewDDEBiftool}.
The paper \cite{Kra-Gre-03} has methods for finding invariant
manifolds for DDEs with constant delays.
%\cite{Hum-Dem-Mag-Uph-12,Kra-Gre-03,ks,ppk2,sk4,NewDDEBiftool}.
Issues
that remain outstanding include smoothness of center manifolds and,
therefore, also normal form reductions.

In light of the considerable additional difficulty, state-dependent
delays are quite often replaced by constant delays --- by considering
some sort of average or nominal delays --- even in modelling
situations when this cannot be readily justified. The obvious question
is whether and when a state-dependent DDE displays dynamics that is
considerably different from that of the associated constant-delay DDE.

In this paper we address this practical question by studying a
prototypical DDE with state-dependent delays, rather than an
equation arising from a specific
application. This example DDE has the important property that it
exhibits very complicated dynamics with state dependence, while it reduces to a
linear DDE with only trivial dynamics if the delays are made
constant. Specifically, we consider here the scalar DDE
\be \label{eq:twostatedep} %\dot{u}(t)
u'(t)=-\gamma u(t) - \kappa_1 u(\alpha_1(t,u(t))) - \kappa_2 u(\alpha_2(t,u(t))),
\enspace\textrm{where}\enspace \alpha_i(t,u(t))=t-a_i-c_iu(t).
\ee
The two delay terms, with feedback strengths $\kappa_1,\kappa_2
\geq0$, are given by the linear functions $\alpha_i(t,u(t))$, where
$a_i$ and $c_i$ are strictly positive.  In the absence of the delay
terms, that is, for $\kappa_1 = \kappa_2 = 0$, \eqref{eq:twostatedep}
is a linear scalar equation whose solutions decay exponentially to the
origin with rate $\gamma >0$. For $\kappa_1, \kappa_2 \neq 0$, on the
other hand, the delay terms are present and constitute a feedback. When $c_1 = c_2 =
0$ the DDE \eqref{eq:twostatedep} is linear with two fixed delays
$a_1$ and $a_2$, while for $c_1, c_2 \neq 0$ the delay terms are linearly
state dependent.

A singularly perturbed version of \eq{eq:twostatedep} is studied in
\cite{HBCHM:1,KE:1,JMPRN16}. In \cite{KE:1} solutions are considered
near the singular Hopf bifurcations, while \cite{HBCHM:1} constructs
large amplitude singular solutions and studies the singular limit of
the fold bifurcations. Equation \eq{eq:twostatedep} is a
generalisation of the corresponding single delay DDE which can be obtained
from \eq{eq:twostatedep} by setting $\kappa_2=0$. The single delay DDE was
first introduced in a singularly perturbed form as an example problem
by Mallet-Paret and Nussbaum in \cite{JMPRN11} and considered
extensively in \cite{JMPRNP94} as part of a series of papers
\cite{JMPRNI,JMPRNII,JMPRNIII,JMPRN11,JMPRN16,JMPRNP94} studying
singularly perturbed solutions of state-dependent DDEs.

We consider \eq{eq:twostatedep} with all parameters non-negative and
without loss of generality assume that $a_2>a_1$. We also assume
\be \label{eq:gk2}
\gamma>\kappa_2.
\ee
It is shown in
\cite{Hum-Dem-Mag-Uph-12} that
if \eq{eq:gk2} holds
and
\be \label{eq:phibd}
\phi(t)\in\Bigl( -\frac{a_1}{c},\frac{a_1}{\gamma c}(\kappa_1+\kappa_2) \Bigr), \quad
\forall t\in\Bigl[ -a_2-\frac{a_1}{\gamma}(\kappa_1+\kappa_2),0\Bigr]
\ee
then equation \eq{eq:twostatedep} is well posed and all
solutions of the initial value problem composed of solving
\eq{eq:twostatedep} for $t\geq 0$ with the initial function
\be \label{eq:ih}
u(t)=\phi(t), \quad t\leq 0
\ee
satisfy
\be \label{eq:ubd}
u(t)\in\Bigl( -\frac{a_1}{c}, \frac{a_1}{\gamma c}(\kappa_1+\kappa_2)\Bigr), \quad \forall t>0.
\ee
This bound on the
solution also implies a bound on the delays with \eq{eq:twostatedep}
and \eq{eq:ubd} implying that
\be \label{eq:alphabd}
\alpha_i(t,u(t))\in\Bigl( t - a_i - \frac{a_1}{\gamma}(\kappa_1+\kappa_2) , t \Bigr)
\subset \Bigl( t - a_2 - \frac{a_1}{\gamma}(\kappa_1+\kappa_2) , t \Bigr), \quad \forall t\geq0
\ee
and, in particular, the state-dependent delays can never
become advanced when $\gamma>\kappa_2$.  It is also shown in
\cite{Hum-Dem-Mag-Uph-12} that there exists
$\xi\in[0,a_2+\frac{a_1}{\gamma}(\kappa_1+\kappa_2)]$ such that
$\alpha_i(t,u(t))$ is a strictly monotonic increasing function of $t$
for $t>\xi$.

Notice that the DDE \eq{eq:twostatedep} is of the form \eq{eq:rfde}
with $d=1$ if we let \be \label{eq:F} F(\phi)=-\gamma
\phi(0)-\kappa_1\phi(-a_1-c\phi(0)) - \kappa_2\phi(-a_2-c\phi(0)).
\ee We take $\tau=a_2+\frac{a_1}{\gamma}(\kappa_1+\kappa_2)$, which by
\eq{eq:alphabd} ensures that $\alpha_i(t,u(t))\in[t-\tau,t]$ for
$t\geq0$ and the function $u_t$ includes all the information necessary
to evaluate $u'(t)$.  Moreover, provided the initial function $\phi$
is Lipschitz it follows from standard DDE theory \cite{Driver63} that
the initial value problem has a unique solution satisfying
\eq{eq:ubd}.

For $c_1=c_2=0$ general theory \cite{Bel-Coo-63,Hale77,Hal-Lun-93}
states that, depending on the values of $\gamma$, $\kappa_1$ and
$\kappa_2$, all trajectories of \eqref{eq:twostatedep} decay to the
origin or grow exponentially in time.
In other words, the dynamics of the
system without state dependence in the delay terms is indeed
trivial. On the other hand, it was shown in \cite{Hum-Dem-Mag-Uph-12}
that state dependence of the delay terms changes the dynamics
completely, since the function $F$ in \eqref{eq:F} is nonlinear.
Therefore, the state dependency of the
delays for $c_1,c_2 \neq 0$ is responsible for nonlinearity in the
system. The two delay terms introduce two oscillatory degrees of
freedom into the system, which may then interact nonlinearly. As a
result, the dynamics of the DDE \eqref{eq:twostatedep} is no longer
linear; rather it is, colloquially speaking, potentially at least as
complicated as that of two coupled nonlinear oscillators with
dissipation. Indeed, the interest in \eqref{eq:twostatedep} arises
from the fact that it is effectively the simplest example one can
consider of a DDE with several state-dependent delays. In particular,
any non-trivial dynamics that one finds must be due to the state
dependence.

%%%%%%%%%%%%%%%%%%%%%%%%%%%%%%%%%%%%%%%%%%%%%%%%%%%
\begin{figure}[t!]
\begin{center}
\includegraphics{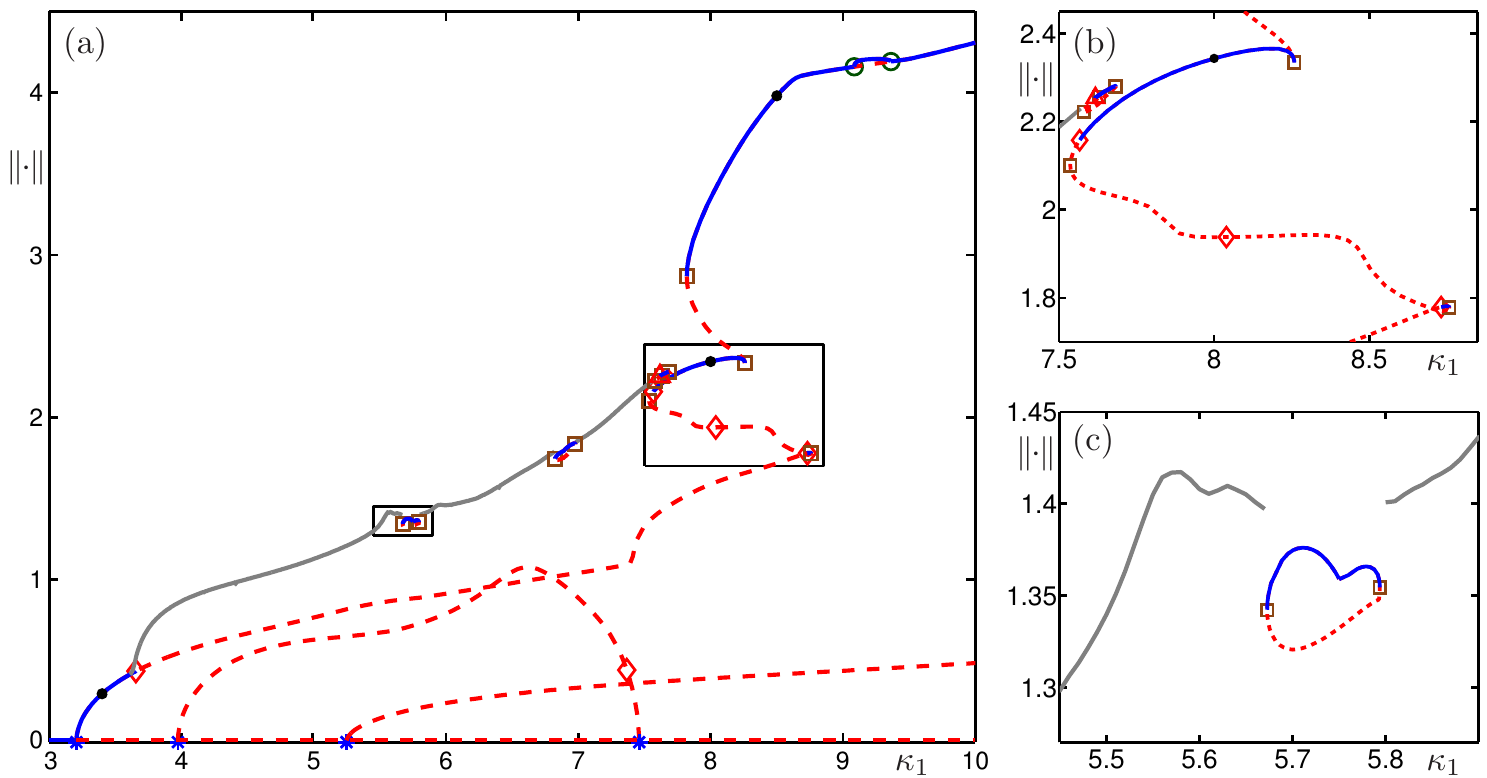}
\caption{One-parameter bifurcation diagram in $\kappa_1$ of
\eqref{eq:twostatedep}, showing the norm
%$\parallel \! u(t)\! \parallel
$\| u(t)\|
= \max u(t)-\min u(t)$ of periodic orbits bifurcating
from Hopf bifurcations of the trivial solution (a).  Stable orbits are
shown as solid blue curves and unstable ones as dashed red curves;
indicated are points of Hopf bifurcation (stars), saddle-node of limit
cycle bifurcation (squares), period-doubling bifurcations (circles),
and torus bifurcation (diamonds). Also shown is a grey curve of tori
that bifurcate from the principal branch of periodic orbits at
$\kappa_1\approx3.6557$. Panels (b) and (c) are two enlargements near
the stable part of the principal branch and near an isola of periodic
orbits associated with $1\!:\!4$ phase locking.  The black dots
correspond to the stable periodic orbits shown in
\fref{fig:perorbits}.  Here $\kappa_2=3.0$ and, throughout,
$\gamma=4.75$, $a_1=1.3$, $a_2=6.0$ and $c_1 = c_2 =1.0$.
Reproduced with permission from \cite{Hum-Dem-Mag-Uph-12}. Copyright 2012
American Institute of Mathematical Sciences.
}
\label{fig:1Dbif}
\end{center}
\end{figure}
%%%%%%%%%%%%%%%%%%%%%%%%%%%%%%%%%%%%%%%%%%%%%%%%%%%

Throughout this paper we will take
\be \label{eq:oriannaparams}
\gamma = 4.75,\quad a_1 = 1.3, \quad a_2 = 6, \quad c_1 = c_2 = 1,
\ee
and vary the values of $(\kappa_1,\kappa_2)$ with $\kappa_2\in(0,4.75)$ to
satisfy \eq{eq:gk2}.  The parameter set \eq{eq:oriannaparams} was
first identified as producing interesting dynamics for
\eq{eq:twostatedep} in \cite{Hum-Dem-Mag-Uph-12}.
There, one-parameter
bifurcation diagrams for \eq{eq:twostatedep} were produced for
this parameter set with fixed values of $\kappa_2$.
In \cite{Hum-Dem-Mag-Uph-12}, it was also noticed
that the bifurcation diagram is topologically very different for
other choices of parameters.

\fref{fig:1Dbif} illustrates the results obtained in
\cite{Hum-Dem-Mag-Uph-12} with $\kappa_2=3$ and the other parameters
given by \eq{eq:oriannaparams}, where the dynamics of
\eq{eq:twostatedep} was explored by means of finding the Hopf
bifurcations of the zero solution and continuing the branches of
bifurcating periodic orbits. As panel (a) shows, the zero solution
loses stability in a first Hopf bifurcation at $\kappa_1 \approx
3.2061$ where a branch of stable periodic solutions emerges. These
lose stability in a torus (or Neimark-Sacker) bifurcation at $\kappa_1
\approx 3.6557$. The branch of (unstable) saddle periodic solutions
regains stability in the interval $\kappa_1 \in [7.5665 , 8.2585]$
after two saddle-node (or fold) bifurcations and several further torus
bifurcations; see the enlargement in \fref{fig:1Dbif}(b). A further
two saddle-node bifurcations lead to a hysteresis loop of the branch
and the periodic solution is stable again for $\kappa_1 > 7.82$,
except for $\kappa_1\in[9.0857,9.3624]$ where a pair of
period-doubling bifurcations lead to a short interval of stable
period-doubled solutions.  Also shown in \fref{fig:1Dbif}(a) are
branches of bifurcating stable tori, which are represented by the
maximum of the norm along a numerically computed trajectory of
sufficient length. As is expected from general theory one finds locked
dynamics on the torus when $\kappa_1$ passes through resonance
tongues. The associated periodic orbits on the torus can be continued
and \fref{fig:1Dbif}(c) shows the isola of periodic solutions
corresponding to $1\!:\!4$ phase locking. Notice that there are
further Hopf bifurcation points and bifurcating branches of periodic
solutions in \fref{fig:1Dbif}(a), but none of them are stable.

%%%%%%%%%%%%%%%%%%%%%%%%%%%%%%%%%%%%%%%%%%%%%%%%%%%
\begin{figure}[t!]
\begin{center}
\includegraphics{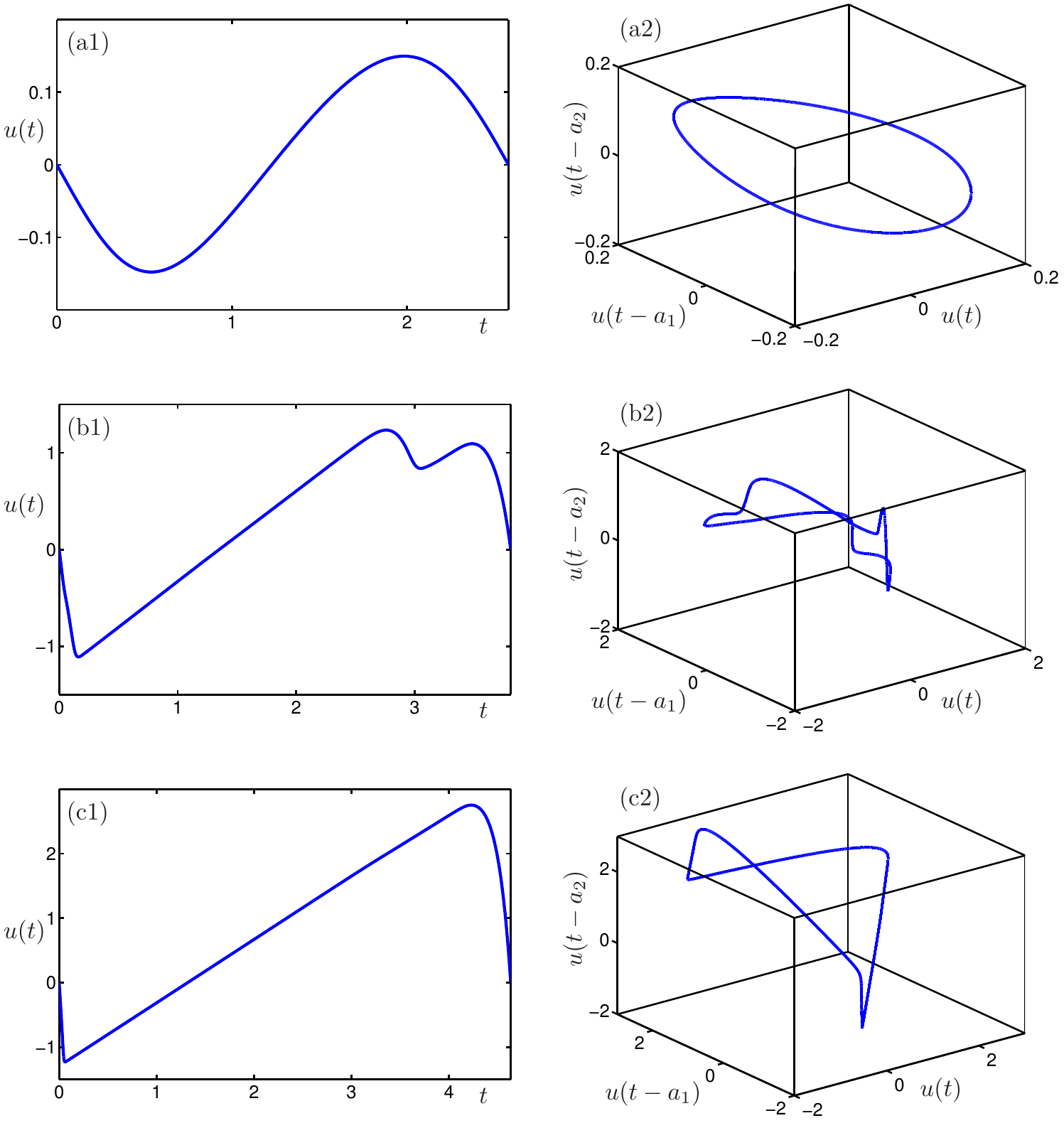}
\caption{Three stable periodic orbits from the principal branch in
\fref{fig:1Dbif}, shown as a time series over one period (left column)
and in projection into $(u(t),u(t-a_1),u(t-a_2))$-space (right
column); here $\kappa_1=3.4$ in row (a), $\kappa_1=8.0$ in row (b),
and $\kappa_1=8.5$ in row (c). }
\label{fig:perorbits}
\end{center}
\end{figure}
%%%%%%%%%%%%%%%%%%%%%%%%%%%%%%%%%%%%%%%%%%%%%%%%%%%

\Fref{fig:perorbits} shows examples of stable periodic solutions from
the three main ranges of stability discussed above, for values of
$\kappa_1$ as indicated by the black dots in \fref{fig:1Dbif}(a).
Shown in \fref{fig:perorbits} are the time series of $u(t)$ over one
period and the orbit in projection onto
$(u(t),u(t-a_1)u(t-a_2))$-space of the respective periodic solution.
The periodic solution in row (a) of \fref{fig:perorbits} is almost
perfectly sinusoidal, as is expected immediately after a Hopf
bifurcation.  The periodic solution in row (b), on the other hand,
features two local maxima and is close to a saw-tooth
shape. Similarly, the periodic solution in \fref{fig:perorbits}(c) is
very close to a simple saw-tooth, with a single linear rise and then a
sharp drop in $u(t)$. Sawtooth periodic solutions and some of their
bifurcations are considered in \cite{HBCHM:1}, where a singularly
perturbed version of \eq{eq:twostatedep} is studied.

The results from \cite{Hum-Dem-Mag-Uph-12}, summarized in
\frefs{fig:1Dbif} and \ref{fig:perorbits}, clearly show that
\eqref{eq:twostatedep} features highly nontrivial dynamics due to the
state dependence. On the other hand, a more detailed bifurcation
analysis of the system has not been performed. The only two-parameter
continuation performed in \cite{Hum-Dem-Mag-Uph-12} is limited to that
of the curves of Hopf bifurcations in the
$(\kappa_1,\kappa_2)$-plane. It identified Hopf-Hopf (or double Hopf)
bifurcations, but neither they nor the curves of torus bifurcations
emerging from them were investigated in that work. Moreover, the
bifurcating tori were not studied in detail in
\cite{Hum-Dem-Mag-Uph-12}; in particular, stable tori themselves were
not computed when phase locked.

%%%%%%%%%%%%%%%%%%%%%%%%%%%%%%%%%%%%%%%%%%%%%%%%%%%
\begin{figure}[t!]
\begin{center}
\includegraphics{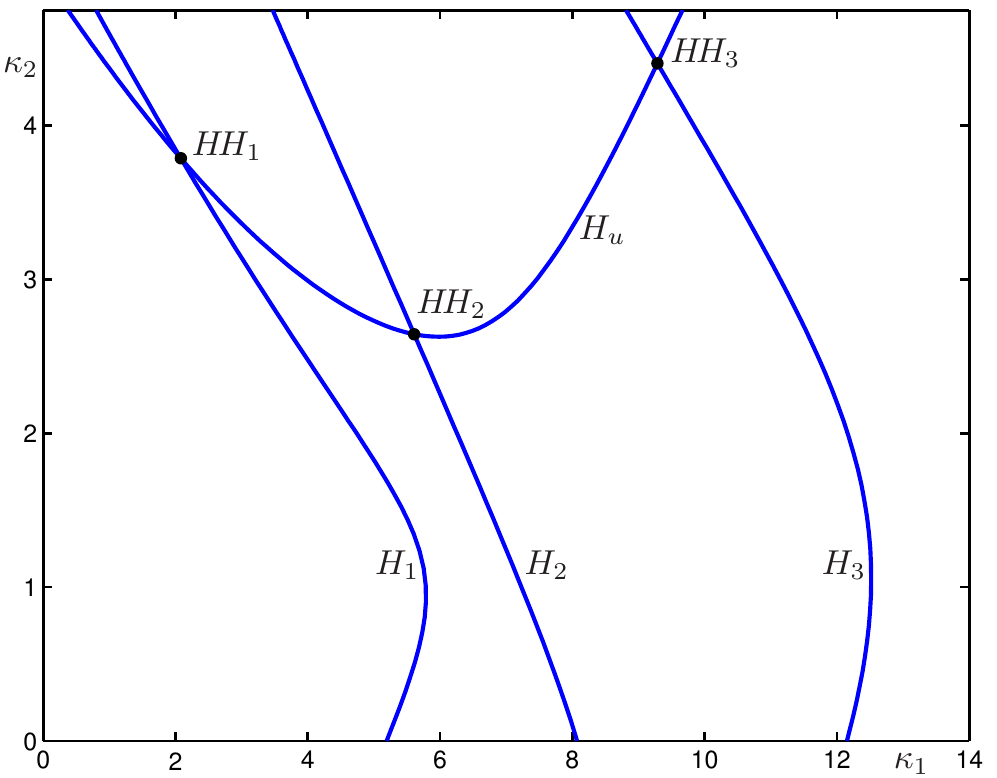}
\caption{Curves of Hopf bifurcation in the $(\kappa_1,\kappa_2)$-plane
of \eqref{eq:twostatedep}; the upper Hopf bifurcation curve $H_u$
intersects the Hopf bifurcation curves $H_j$ for $j=1,2,3$ at
Hopf-Hopf bifurcation points $\HH_j$.}
\label{fig:hopfcurves}
\end{center}
\end{figure}
%%%%%%%%%%%%%%%%%%%%%%%%%%%%%%%%%%%%%%%%%%%%%%%%%%%

To highlight the full extent of the dynamics generated by the state
dependence, in this work we present a bifurcation study of
\eqref{eq:twostatedep} that goes well beyond that in
\cite{Hum-Dem-Mag-Uph-12}.  Our focus is on two-frequency dynamics and
associated resonance phenomena; our main objects of study are the
bifurcation diagram in the $(\kappa_1,\kappa_2)$-plane and the
associated dynamics in phase space. The starting point of our
investigation is the arrangement of the Hopf bifurcation curves of
\eqref{eq:twostatedep} shown in \fref{fig:hopfcurves}.

A Hopf bifurcation occurs when a complex conjugate pair of
characteristic values crosses the imaginary axis in the linearized
system. State-dependent DDEs are linearized around equilibria by first
freezing the state-dependent delays at their steady-state values. This
technique has long been applied heuristically, but more recently has
been established rigorously by Gy\"ori and Hartung
\cite{Gyo-Har-00,Gyo-Har-07} for a class of problems including
\eq{eq:twostatedep}. Hence, we obtain
\begin{equation}\label{eq:lineareq} u'(t) = - \gamma u(t)- \kappa_1
u(t - a_1)- \kappa_2 u(t - a_2)
\end{equation}
as the linearization of \eq{eq:twostatedep} about the trivial steady
state $u\equiv 0$. The characteristic equation for \eq{eq:lineareq} is
given by
\begin{equation}\label{chareq} 0 = \lambda + \gamma +\kappa_1
e^{-a_1\lambda_1}+\kappa_2e^{-a_2\lambda_2},
\end{equation}
and so at a Hopf bifurcation we have $\lambda=\pm i\omega$ with
\begin{equation}\label{hopflocation} 0 = i\omega + \gamma +\kappa_1
e^{-i a_1\omega}+\kappa_2e^{-i a_2 \omega}.
\end{equation}

The three curves $H_1$, $H_2$ and $H_3$ in \fref{fig:hopfcurves}
emerge from $\kappa_2 = 0$ and are functions of $\kappa_2$. These
three Hopf bifurcation curves are intersected by the curve $H_u$,
which exists only above $\kappa_2 \approx 2.627$ and is a function of
$\kappa_1$. The three intersection points $\HH_1$, $\HH_2$ and $\HH_3$
are codimension-two points of Hopf-Hopf bifurcation.  From
\eqref{hopflocation} it follows that there are in fact infinitely many
Hopf bifurcation curves of \eqref{eq:twostatedep} as
$\kappa_1\to\infty$ and, consequently, other Hopf-Hopf points;
however, these are not shown in \fref{fig:hopfcurves} because we
concentrate here on the $\kappa_1$-range of $[0,14]$.  Note that we
only show the $(\kappa_1,\kappa_2)$-plane for $\kappa_2 \leq
\gamma=4.75$, because this is the $\kappa_2$-range for which we know
that the state-dependent DDE is well posed.

The numerical computation of Hopf bifurcations in state-dependent DDEs
has been implemented in the DDE-BIFTOOL software package
\cite{Eng-Luz-Roo-02,NewDDEBiftool}, and this capability actually
predates their rigorous proof.  Eichmann~\cite{Eichmann06} was the
first to establish a rigorous Hopf bifurcation theorem for
state-dependent DDEs, but results have only appeared in the published
literature much more recently \cite{Hu-Wu-10,Sieber12}. We perform
here a calculation of the four-dimensional normal form ODE on the
center manifold of the Hopf-Hopf points $\HH_1$, $\HH_2$ and
$\HH_3$. As far as we are aware, this is the first such calculation to
determine the type of Hopf-Hopf bifurcations in a state-dependent
DDE. The Hopf-Hopf normal form ODE with the multitude of cases that
can arise in the unfolding is presented in detail in
\cite{Kuz-04-Book}. In constant-delay equations it has already been
studied, see for instance \cite{Bel-Cam-94}; the normal form procedure
is also elaborated in \cite{Wu-Guo-13} and has been implemented
recently \cite{Wage14} as part of DDE-BIFTOOL \cite{NewDDEBiftool} for
constant delays only.  Our approach is to derive a constant-delay DDE
from the state-dependent DDE \eqref{eq:twostatedep} by expanding the
state dependence to sufficient order in (many) constant delays. The
Hopf-Hopf normal form ODE can then be computed from this
constant-delay DDE with established methods, and specifically we
implemented the approach from \cite{Wu-Guo-13}. In this way, we are
able to determine the type of the Hopf-Hopf bifurcation and show that
a pair of torus bifurcation curves emerges from each of the points
$\HH_1$, $\HH_2$ and $\HH_3$. The reduction to the constant-delay DDE
and the corresponding resulting normal form coefficients are presented
in \sref{sec:HH}, where we also compare our results with those
obtained from the DDE-BIFTOOL implementation. Further details of the
normal form calculations can be found in
Appendix~\ref{app:NormalForm}. Our Matlab code \NormalForm, which
implements the constant-delay expansion and computes the normal form
coefficients for the Hopf-Hopf bifurcation, is available as a
supplement to this paper.

The dynamics on the bifurcating tori may be quasi-periodic or locked,
and this is organised by resonance tongues that are bounded by curves
of saddle-node (or fold) bifurcations of periodic orbits.  We proceed
in \sref{sec:tori} by computing and presenting bifurcating stable
quasiperiodic and phase-locked tori. The Matlab \cite{Matlab}
state-dependent DDE solver \texttt{ddesd} is used to find trajectories
on stable invariant tori. In this way, we find quasiperiodic (or
high-period) tori. To obtain locked tori, we find and continue the
locked periodic solutions with the software package DDE-BIFTOOL
\cite{Eng-Luz-Roo-02,NewDDEBiftool}. The unstable manifolds of the
saddle periodic orbits on the torus are then represented as
two-dimensional surfaces obtained by numerical integration of trajectories in these
%unstable manifolds.
manifolds.

Since \eqref{eq:twostatedep} is a scalar DDE, but its phase-space is
infinite dimensional, we consider finite-dimensional projections of
the infinite-dimensional phase space. Moreover, we also show the tori
in suitable projections of the Poincar\'e map defined by $u(t)$
passing through $0$. This allows us to reveal the inherently
low-dimensional character of these invariant tori and associated
bifurcations.

We then perform in \sref{sec:resonance} a bifurcation study of the
emergence of tori and associated resonance phenomena. Specifically, we
compute and illustrate in the $(\kappa_1,\kappa_2)$-plane the curves
of torus bifurcation emerging from the Hopf-Hopf bifurcation point
$\HH_1$ and the associated structure of resonance tongues. We also
consider in detail the properties and bifurcations of the invariant
tori inside and near the regions of strong $1\!:\!3$ and $1\!:\!4$
resonances.  More specifically, in \sref{sec:breakup} we show how the
$1\!:\!4$ locked torus loses normal hyperbolicity and then breaks up
in a complicated sequence of bifurcations as $\kappa_1$ is changed.
Finally, in \sref{sec:overall} we present the overall
bifurcation diagram in the $(\kappa_1,\kappa_2)$-plane, provide
some conclusions and point out directions for future
research.

%%%%%%%%%%%%%%%%%%%%%%%%%%%%%%%%%%%%%%%%%%%%%%%%%%%
\section{Normal form at Hopf-Hopf bifurcation}
\label{sec:HH}

Here we derive the normal form of the Hopf-Hopf bifurcations of
\eqref{eq:twostatedep}.  For constant-delay DDEs a center manifold
reduction \cite{Bel-Cam-94,Wu-Guo-13} transforms the DDE into an ODE
on the center manifold, and the normal form of the Hopf-Hopf
bifurcation for ODEs is well known and can be found in
\cite{Kuz-04-Book}.  For state-dependent DDEs, the existence of a
$C^1$ center-unstable manifold has been proved by several authors (for instance, see
\cite{Qes-Wal-09, Kri-06, Stu-12}), with verifiable regularity conditions that
equation \eqref{eq:twostatedep} satisfies, when the spectrum of
\eqref{eq:lineareq} has eigenvalues $\lambda$ satisfying that
$\Re(\lambda) \geq 0$.  However,
the existence of a $C^3$ regular center-unstable manifold, as required for the Hopf-Hopf bifurcation analysis,
has not been rigorously established in the state-dependent case. Nor has
the normal form of the Hopf-Hopf
bifurcation for a state-dependent DDE previously been
elaborated.

Noting that linearization of \eqref{eq:twostatedep} reduces it to the
constant-delay DDE \eqref{eq:lineareq}, our approach is instead to obtain a
series expansion of the right-hand side of
\eqref{eq:twostatedep} in which the low-order terms only involve
constant delays. In particular, the state dependency will only appear in the
higher-order remainder term.
The derivation of the terms up to order three of the normal form
DDE with constant delays near the Hopf-Hopf bifurcation is exact.
We then, as is usual in the field, disregard the remainder
term and consider only this truncated expansion. We conjecture that
the truncated constant-delay DDE fully describes all of the dynamics
near the Hopf-Hopf bifurcation in the state-dependent DDE. We then proceed
by applying the established center manifold
reduction of \cite{Bel-Cam-94,Wu-Guo-13}
to obtain an ODE on the center manifold.  The flow
restricted to the center manifold satisfies an ODE in four-dimensional
space, which can be reduced to a normal form to determine the type of
Hopf-Hopf bifurcation that occurs.  The virtue of this method is that
we study a four-dimensional ODE as opposed to an infinite dimensional
semi-flow. Of course, this construction only works close to the point
of the Hopf-Hopf bifurcation in parameter space, where the center
manifold persists since the rest of the eigenvalues are at a positive
distance from the imaginary axis; the center manifold should be a
normally hyperbolic invariant manifold in the infinite-dimensional
phase space.

Since the state dependency of the delays is the only source of
nonlinearity in the DDE \eqref{eq:twostatedep}, the correct treatment
of these state-dependent delays is essential to our results.
Specifically, our strategy is as follows.  We Taylor expand the
state-dependent terms $u(t-a_i-cu(t))$ in time about their
constant-delay reductions $u(t-a_i)$. This removes the state
dependency from
the equations, but at the cost of introducing derivatives of
$u(t-a_i)$ in higher-order terms.  Not wanting to deal with neutral
DDEs, we remove the derivatives $\tfrac{d^k}{dt^k}u(t-a_i)$ by
differentiating \eqref{eq:twostatedep} $k-1$ times and evaluating them
at $t-a_i$.  This introduces additional delays into the DDE, and also
reintroduces the state dependency of the delays, but only in the
quadratic and higher-order terms. The quadratic state-dependent delays
are removed by the same process of Taylor expansion and
substitution. We can repeat this process as many times as desired to
obtain a DDE with only constant delays in the terms up to $k$-th order
for any $k$. Normal form theory for Hopf-Hopf bifurcation requires the
expansion up to order three, which is why we stop at this order.  By
using the integral form of the remainder in Taylor's theorem, it is
possible to obtain an explicit expression for the higher-order
terms.
In the current work, we conjecture, but do not prove, that the remainder term can indeed be disregarded.
This allows us to apply the
techniques of \cite{Bel-Cam-94,Wu-Guo-13} to the lower-order
constant-delay part of our expanded DDE to determine the normal form
equations, as well as the Hopf-Hopf unfolding bifurcation types.

There is a long and often inglorious history of Taylor expanding in
DDEs to alter or eliminate the delay terms. It is obviously invalid to
expand $u(t-a)$ about $u(t)$ when $|u(t-a)-u(t)|$ is large, which will
be the typical case when $a$ is not small. But related to the
phenomenon of delay induced instability, even when $u$ is close to
steady-state so that $|u(t-a)-u(t)|\ll1$, expanding $u(t-a)$ about
$u(t)$ can change the stability of the steady state; see
\cite{Driver77} for examples.  In the current work, we expand terms of
the form $u(t-a-c u(t))$ about $u(t-a)$ close to steady state. Hence,
not only is the difference in the $u$-values small, that is $|u(t-a-c
u(t))-u(t-a)|\ll1$, but crucially the difference in the time values is also
small, that is, $|(t-a-c u(t))-(t-a)|=|c u(t)|\ll1$.

Having found the normal form of the Hopf-Hopf bifurcation of
\eqref{eq:dde3} we compare the resulting bifurcations predicted by the
normal form calculation with the numerically determined bifurcation
curves for the full state-dependent DDE \eqref{eq:twostatedep}. Close
to the Hopf-Hopf points we find very good agreement, which gives us
confidence in the results obtained by both approaches.
In particular, these results constitute strong numerical evidence that the
resulting normal form for the expanded
constant-delay DDE \eqref{eq:dde3} is indeed that for the
state-dependent DDE \eqref{eq:twostatedep}. While proving this
conjecture is beyond the scope of this
paper, we remark that such a proof, and indeed the expansions that we
perform, require at least $C^3$ regularity of (the solutions in) the
manifold.  To our knowledge, the best regularity result for the center
manifolds in state-dependent DDEs establishes just $C^1$ regularity
\cite{Kri-06}, and $C^r$ regularity with $r>1$ has not yet been
established for center manifolds of state-dependent DDEs.
Nevertheless, the expansions we perform here do not seem to present
any obstruction to obtaining the formal expressions for small
amplitudes of the function $u$.
In fact, one notices that knowing
the $C^1$-smoothness of the local center-unstable manifold
justifies that the solutions can be continued for negative times.
Since in our case we are close to the steady state $u(t) = 0$,
the delays are bounded and the solutions must be $C^k$ smooth
in time.
Indeed, having $C^k$-regular solutions
could lead to obtaining $C^k$ smooth time-$1$ maps, and these are
perhaps the basis to construct a $C^k$-smooth center manifold. This
possible route to $C^k$ regularity is already proposed in
\cite{Har-Kri-Wal-Wu-06}.  We also mention that results for invariant
tori of state dependent DDEs have been derived recently in spaces of
smooth and analytic functions; see \cite{He-Lla-15a, He-Lla-15b}.

We elaborate our steps as follows.  In \sref{sec:expansion}, we
present the details of the expansion of the state-dependent DDE to
obtain a DDE with only constant delays up to order three.  In
\sref{sec:centreman} we describe aspects of the projection onto the
center manifold for this constant-delay DDE, and present the
derivation of the normal form coefficients.  The algebraic details of
these calculations are contained in Appendix~\ref{app:NormalForm}.  In
\sref{sec:NormalForm} we use the normal form obtained to determine the
type of the Hopf-Hopf bifurcation for the three Hopf-Hopf bifurcations
seen in \fref{fig:hopfcurves}.

%%%%%%%%%%%%%%%%%%%%%%%%%%%%%%%%%%%%%%%%%%%%%%%%%%%
\subsection{Expansion of the nonlinearity}
\label{sec:expansion}

In this section, we perform the expansion  of the state dependent delay equation \eqref{eq:twostatedep}  and obtain a constant-delay equation with many delays and a remainder term which is small for solutions in the center or unstable manifolds.

To describe the expansion of the nonlinearity in
\eqref{eq:twostatedep} it is convenient to define the difference
operator $L$ that generates the linear terms on the right hand side of
equation \eqref{eq:lineareq} as
\begin{equation}\label{eq:Llinearop}
L u(t) \equiv -\gamma u(t)-
\kappa_1 u(t-a_1) - \kappa_2 u(t - a_2).
\end{equation}
The difference operator $L$ can be applied recursively,
and it will be useful below to note that
\begin{align} \label{eq:L2}
L^2 u(t-a_i) & = -\gamma Lu(t-a_i) -
\sum_{j=1}^2\kappa_j Lu(t-a_i-a_j)\\
& = \gamma^2 u(t-a_i)+2\gamma\sum_{j=1}^2\kappa_j
Lu(t-a_i-a_j)+\sum\limits_{j,m=1}^2\!\kappa_j\kappa_m
u(t-a_i-a_j-a_m). \notag
\end{align}

\begin{theorem} \label{prop:expansion}
For functions $u$ in the center or unstable manifold of the steady state $u(t) = 0$, the state dependent delay equation \eqref{eq:twostatedep} can be written as a constant-delay equation up to fourth order as
\begin{align} \label{eq:expfull}
u'(t) =  Lu(t)&+\sum_{i=1}^2 \kappa_i cu(t)Lu(t-a_i)+\sum_{i,j=1}^2\kappa_i\kappa_j c^2u(t)u(t-a_i)Lu(t-a_i-a_j) \\
& -\frac12(cu(t))^2\sum_{i=1}^2\kappa_i L^2u(t-a_i)+\cR(t), \notag
\end{align}
with $\cR(t)=\cO(\|u\|_5^4)$ where $\|u\|_5=\sup_{\theta\in[-5a_2,0]}|u(\theta)|$.
\end{theorem}

\begin{proof}
Recall from \eqref{eq:alphabd} that delays are globally bounded by $\tau=a_2+a_1(\kappa_1+\kappa_2)/\gamma$
for the state-dependent DDE \eq{eq:twostatedep}. Since $a_2>a_1$ for $|u|<\delta$ we obtain the stronger bound that
$t-\alpha_j(t,u(t))\leq a_2+c\delta$. Now consider $u$ in the center or unstable manifold so that solutions can be extended in the past.
Using \eq{eq:Llinearop} we can rewrite equation
\eq{eq:lineareq} as $u'(t)=Lu(t)$ and equation \eq{eq:twostatedep} as
\be \label{eq:expansion1} u'(t)= Lu(t) - \sum_{i=1}^2 \kappa_i \bigl
[u(t-a_i - cu(t)) - u(t-a_i)\bigr].
\ee
As already noted, the only
nonlinearities in \eqref{eq:twostatedep} arise from the
state dependency of the delays, and we must handle these terms
carefully to obtain a correct expansion for the normal form.  Close to steady state and close
to Hopf bifurcation, the state-dependent part of the delay term,
$-cu(t)$, will be close to zero.  Therefore, close to the bifurcation
the term $t - a_i - c u(t)$ represents a small displacement from the
constant delay $t-a_i$. Since we assume $a_i>0$ the perturbation will
not be singular.

We write Taylor's theorem as
\begin{align} \notag
u^{(p)}(w- \tau - c u(w))& = u^{(p)}(w-\tau) +
\int_0^1u^{(p+1)}(w - \tau -c u(w)\,s_1) ds_1 (-cu(w))\\ \notag
& = u^{(p)}(w-\tau) + u^{(p+1)}(w - \tau) (-cu(w))\\ \notag
& \qquad + \int_0^1 \int_0^{s_1} u^{(p+2)}(w - \tau -c u(w)\,s_1\,s_2) ds_2
(-cu(w) s_1)ds_1 (-cu(w)) \\  \label{eq:taylork}
& = \sum_{j=0}^k \frac{1}{j!}u^{(p+j)}(w- \tau)(-c u(w))^j \\  \notag
& \qquad + \bigg( \int_0^1 \int_0^{s_1} \cdots \int_0^{s_{k-1}}
u^{(p+j+1)}(w - \tau -c u(w)\,s_1\,s_2\cdots s_k)\\
& \qquad\qquad\qquad \cdot [s_1(s_1\,s_2)\cdots(s_1\cdots
s_k)]ds_k\cdots ds_1 \bigg)\cdot (-cu(w))^{j+1},  \notag
\end{align}
where we note that on the unstable and center manifolds solutions are $C^p$, because they can be extended backwards in time, the delays are bounded, and solutions become more regular as we integrate \eq{eq:twostatedep} forwards in time.
Equation~\eq{eq:taylork} gives an estimate of the residue of
Taylor's theorem in terms of $(-c u(w))^{j+1}$ and $u^{(p+j+1)}$.
Now, we use \eqref{eq:taylork} with $w=t$, $\tau=a_i$, $p=0$ and $k=2$
to obtain
\begin{align} \label{eq:exp1}
u'(t) & = Lu(t) -\sum_{i=1}^2 \kappa_i \sum_{j=1}^2
\frac{1}{j!}u^{(j)}(t-a_i)(-cu(t))^j \\
& \quad + \left[\sum_{i=1}^2 \kappa_i \int_0^1
\int_0^{s_1}\int_0^{s_2}u^{(3)}(t-a_i -
cu(t)s_1s_2s_3)s_1^3s_2^2s_3\,ds_3ds_2ds_1\right] (-cu(t))^3. \notag
\end{align}
Note that we choose $k=2$ so that the integral remainder
term is quartic; more precisely it is $\cO([u(t)]^3u^{(3)}(t))$. But with bounded delays it follows from
differentiating \eq{eq:twostatedep} that for $\delta>0$ sufficiently small
\begin{align} \label{eq:u3bd}
|u^{(3)}(t-a_i-c\delta)|&\leq C_2\sup_{\theta\in[-a_i-a_2-2c\delta,0]}|u''(\theta)|
\leq C_3\sup_{\theta\in[-3a_2-3c\delta,0]}|u'(\theta)| \\ \notag
&\leq C_4\sup_{\theta\in[-4a_2-4c\delta,0]}|u(\theta)|\leq C_4\|u\|_5.
\end{align}
One problem with the expansion \eqref{eq:exp1} is
that the nonlinear terms include delayed derivative terms in $u'$,
$u''$ and $u^{(3)}$.  We want to eliminate terms of this form to avoid
the possibility of neutrality in our equations.  To this end, we
consider first the terms of the form $u'(t-a_i)$ appearing in
\eqref{eq:exp1}.  Applying \eqref{eq:twostatedep} gives
$$u'(t-a_i)=- \gamma u(t-a_i)- \sum_{j=1}^2\kappa_j  u(t - a_i - a_j - cu(t-a_i)).$$
To remove the state dependency from the right-hand side, we apply
\eqref{eq:taylork} with $w=t-a_i$, $\tau=a_j$, $p=0$ and $k=1$ to obtain
\begin{align} \label{eq:utai}
u'(t-a_i)& =- \gamma u(t-a_i)-
\sum_{j=1}^2\kappa_j u(t - a_i - a_j ) +\sum_{j=1}^2\kappa_j cu'(t-a_i-a_j)u(t-a_i)\\ \notag
& \quad +\left[\sum_{j=1}^2\kappa_j \int_0^1\int_0^{s_1} u''(t-a_i-a_j-c u(t -
a_i)s_1)s_1\,ds_2ds_1\right] (-cu(t-a_i))^2.
\end{align}
But using \eqref{eq:twostatedep} again and
\eqref{eq:taylork} with $w=t-a_i-a_j$, $\tau=a_m$ and $p=k=0$ we have
\begin{align} \notag
u'(t-a_i&-a_j) =- \gamma u(t-a_i-a_j)-
\sum_{m=1}^2\kappa_m u(t - a_i - a_j - a_m - cu(t-a_i-a_j))\\  \label{eq:utaiaj}
& =- \gamma u(t-a_i-a_j)- \sum_{m=1}^2\kappa_m u(t - a_i - a_j - a_m)\\ \notag
&\quad +\left[\sum_{m=1}^2\kappa_m\int_{0}^{1}u'(t-a_i - a_j -a_m- c
u(t-a_i-a_j)s_1)ds_1\right ](-cu(t-a_i - a_j)).
\end{align}
Hence, we can rewrite \eqref{eq:exp1} as
\be \label{eq:exp2} u'(t)=
Lu(t)+N_2u(t)+N_{23}u(t)-\frac12\sum_{i=1}^2 \kappa_i
u''(t-a_i)(cu(t))^2+\cR_{24}(t),
\ee
where $N_2u(t)$ contains the
quadratic terms in the expansion of nonlinearity, and $N_{23}u(t)$
contains the cubic terms arising from the substitution of
\eqref{eq:utaiaj} and \eqref{eq:utai} into \eqref{eq:exp1}, with
\begin{align}
N_2u(t) & =\sum_{i=1}^2 \kappa_i cu(t)\Bigl[- \gamma
u(t-a_i)- \sum_{j=1}^2\kappa_j u(t - a_i - a_j )\Bigr] =\sum_{i=1}^2
\kappa_i cu(t)Lu(t-a_i), \label{eq:N2op} \\ \label{eq:N23op}
N_{23}u(t) & = \sum_{i,j=1}^2\kappa_i\kappa_j c^2u(t)u(t-a_i)\Bigl[-\gamma
u(t-a_i-a_j)-\sum_{m=1}^2\kappa_m u(t - a_i - a_j -a_m)\Bigr] \\
& =\sum_{i,j=1}^2\kappa_i\kappa_j c^2u(t)u(t-a_i)Lu(t-a_i-a_j). \notag
\end{align}
The expression $\cR_{24}(t)$ contains the fourth-order
integral remainder term of the Taylor series stated in
\eqref{eq:exp1}, as well as the additional fourth order integral terms arising
from the substitution of \eqref{eq:utai} and \eqref{eq:utaiaj} into
\eqref{eq:exp1}.

It remains to expand the terms $u''(t-a_i)$ in
\eqref{eq:exp2}. Differentiating \eqref{eq:twostatedep} and then
applying \eqref{eq:taylork} with $p=1$ and $k=0$, gives
\begin{align} \notag
u''(t-a_i)& =-\gamma u'(t-a_i) -( 1-cu'(t-a_i)
)\sum_{j=1}^2\kappa_j u'(t-a_i-a_j-cu(t))\\ \label{eq:uddtai}
&=-\gamma u'(t-a_i)-( 1-cu'(t-a_i) )\sum_{j=1}^2\kappa_j \Bigl[u'(t-a_i-a_j)\\ \notag
& \qquad +\int_{0}^{1}u'(t-a_i - a_j - c u(t-a_i)s_1)ds_1(-cu(t-a_i))\Bigr].
\end{align}
Similar to \eqref{eq:utai} and \eqref{eq:utaiaj}, but this
time applying \eqref{eq:taylork} with $p=k=0$, we can remove the
$u'(t-a_i)$ and $u'(t-a_i-a_j)$ terms from \eqref{eq:uddtai}. Just
considering the linear terms in \eqref{eq:uddtai} and using \eq{eq:L2}
we find that
\begin{align*}
-\gamma & u'(t- a_i) - \sum_{j=1}^2\kappa_j u'(t-a_i-a_j)\\
& = -\gamma\bigl[-\gamma u(t-a_i) - \sum_{j=1}^2\kappa_j u(t-a_i-a_j-cu(t-a_i))\bigr] \\
& \quad - \sum_{j=1}^2\kappa_j \Bigl[-\gamma u(t-a_i-a_j) - \sum_{m=1}^2\kappa_m u(t-a_i-a_j-a_m-cu(t-a_i-a_j))\Bigr]\\
& = -\gamma\biggl[-\gamma u(t-a_i) - \!\sum_{j=1}^2\kappa_j
\Big[ u(t-a_i-a_j)\\
&\qquad\qquad+\!\int_{0}^{1}\!\!u'(t-a_i - a_j - c u(t-a_i)s_1)ds_1(-cu(t-a_i))\Big]\biggr]\\
& \quad - \sum_{j=1}^2\kappa_j \biggl[-\gamma u(t-a_i-a_j) - \sum_{m=1}^2\kappa_m\Big[
u(t-a_i-a_j-a_m) \\
&\qquad\qquad +\int_{0}^{1}u'(t-a_i - a_j -a_m- c u(t-a_i-a_j)s_1)ds_1(-cu(t-a_i - a_j))\Big]\biggr]\\
& = L^2u(t-a_i)
 + \sum_{j=1}^2\gamma\kappa_j\int_{0}^{1}u'(t-a_i - a_j - c u(t-a_i)s_1)ds_1(-cu(t-a_i))\\
& \qquad  + \sum_{j,m=1}^2\kappa_j\kappa_m \int_{0}^{1}u'(t-a_i - a_j -a_m- c u(t-a_i-a_j)s_1)ds_1(-cu(t-a_i - a_j)).
\end{align*}
Hence, from \eqref{eq:exp2} we obtain \eqref{eq:expfull},
%\begin{align} \label{eq:expfull}
%u'(t) =  Lu(t)&+\sum_{i=1}^2 \kappa_i cu(t)Lu(t-a_i)+\sum_{i,j=1}^2\kappa_i\kappa_j c^2u(t)u(t-a_i)Lu(t-a_i-a_j) \\
%& -\frac12(cu(t))^2\sum_{i=1}^2\kappa_i L^2u(t-a_i)+\cR(t), \notag
%\end{align}
where the remainder term $\cR(t)$ contains all the integral terms
derived above.
Equation~\eq{eq:u3bd} can be used to show that the remainder term in \eq{eq:exp1} is $\cO(\|u\|_5^4)$, and all the remaining integral remainder terms are seen to be $\cO(\|u\|_5^4)$ similarly.
\end{proof}

Overall, we have transformed the state-dependent DDE
\eqref{eq:twostatedep} into DDE \eqref{eq:expfull} whose terms up to
order three contain only constant delays. The price for doing this is
the introduction of additional delay terms.  While
\eqref{eq:twostatedep} contains two state-dependent delays, and its
linearization contains two constant delays, in equation
\eqref{eq:expfull} the second-order terms features five
and the third-order terms nine constant delays. Indeed, it is
easy to see that, if we continued the expansion in \eqref{eq:exp1} to
higher order, then the term $-(-cu(t)^j)\sum_{i=1}^2\kappa_i
u^{(j)}(t-a_i)$ leads to a $j^{\rm th}$-order term of the form
$-(-cu(t)^j)\sum_{i=1}^2\kappa_i L^ju(t-a_i)$. Thus, when $a_1$ and
$a_2$ are not rationally related, we will obtain $j(j+3)/2$ delays at
$j^{\rm th}$-order, namely all the terms of the form $u(t-m a_1-n
a_2)$ where $m,n$ are nonnegative integers and $1\leq m+n\leq
j$. Recalling that $a_2>a_1$ the largest delay appearing at $j^{\rm
th}$-order is then $u(t-ja_2)$.

If desired the derivatives of $u$ that appear in
$\cR(t)$ can all be removed by using \eqref{eq:twostatedep} and
derivatives of that equation, just as we removed such derivatives from
the lower-order terms.  This would result in state-dependent delays
appearing in the $\cR(t)$. Alternatively the state dependency or
distributed delay terms could be moved to higher-order terms by
truncating the expansions above at higher order. Importantly, the remainder
terms are beyond the orders that we will need for subsequent normal
form consideration, and we have the following.

%\begin{conjecture} \label{prop:remainder}
%The local dynamics near a Hopf-Hopf bifurcation of the
%state dependent delay equation
%\eqref{eq:twostatedep} is determined solely by the third-order terms
%of the constant-delay expansion \eqref{eq:expfull}. In other words,
%the remainder term $\cR(t)$ can be disregarded and standard normal
%form calculation for constant-delay DDEs can be applied to determine
%the nature of the Hopf-Hopf bifurcation.
%\end{conjecture}

%\begin{conjecture}
%The local dynamics near the steady state $u(t) = 0$ of the state dependent delay
%equation \eqref{eq:twostatedep} are determined solely by the constant-delay expansion
%up to the given order. In other words, the remainder term $\cR(t)$ can be
%disregarded and standard normal form calculation for constant-delay DDEs can
%be applied to study steady-state bifurcations of \eqref{eq:twostatedep}.
%\end{conjecture}

\begin{conjecture} \label{conject}
The local dynamics near the steady state $u(t) = 0$ of the state dependent delay
equation \eqref{eq:twostatedep} are determined solely by the constant-delay expansion
up to the given order.
In other words, to study steady-state bifurcations of \eqref{eq:twostatedep}
standard normal form calculations for constant-delay DDEs can be applied to
the constant-delay expansion truncated to suitable order.
\end{conjecture}

Specifically for the Hopf-Hopf bifurcations of interest,
from now on we consider only the constant-delay DDE we derived to third order in \eqref{eq:expfull}.
Not using the difference operator $L$, it takes the form
\begin{align}  \label{eq:dde3}
& u'(t) = \\ \notag
& - \gamma u(t)- \kappa_1 u(t-a_1)-\kappa_2 u(t-a_2)
-\sum_{i=1}^2 \kappa_i cu(t)\Bigl[\gamma u(t-a_i)+\sum_{j=1}^2\kappa_j u(t - a_i\! -a_j)\Bigr]\\ \notag
& -\sum_{i,j=1}^2\kappa_i\kappa_j c^2u(t)u(t-a_i)\Bigl[\gamma u(t-a_i\!-a_j)+\sum_{m=1}^2\kappa_m u(t-a_i\!-a_j\!-a_m)\Bigr]\\
& \mbox{}\hspace{-0.5em}-\frac12(cu(t))^2\sum_{i=1}^2\!\kappa_i\Bigl[\gamma^2u(t-a_i)+2\gamma\!\sum_{j=1}^2\!\kappa_ju(t-a_i\!-a_j)
+\!\!\!\sum_{j,m=1}^2\!\!\!\kappa_j\kappa_mu(t-a_i\!-a_j\!-a_m)\Bigr]. \notag
\end{align}
We remark that this way of writing the constant-delay DDE is
convenient for the implementation of the DDE-BIFTOOL normal form
computations which require a DDE with constant delays, and in the
supplemental material as \texttt{sys\_cub\_rhs} we provide a
DDE-BIFTOOL system definition of \eqref{eq:dde3}.  However, our own
Hopf-Hopf normal form code \NormalForm works directly from the
state-dependent DDE \eqref{eq:twostatedep}, and computes
\eqref{eq:dde3} from \eqref{eq:twostatedep} using symbolic
differentiation as the first step for deriving the normal form
parameters.

%%%%%%%%%%%%%%%%%%%%%%%%%%%%%%%%%%%%%%%%%%%%%%%%%%%
\subsection{Center manifold reduction and resulting normal form}
\label{sec:centreman}

The next step is to derive the normal form for the constant-delay DDE
\eqref{eq:dde3}. For constant-delay DDEs there are well established
techniques for deriving normal forms through center manifold
reductions.  To the best of our knowledge,
the Hopf-Hopf bifurcation for a constant-delay DDE was
first elaborated in B\'elair and Campbell \cite{Bel-Cam-94}, but here we
follow the derivation of Wu and Guo \cite{Wu-Guo-13}.  The main idea
in this construction is to study the restriction of the semi-flow of
\eqref{eq:dde3} to the center manifold at the point of the Hopf-Hopf
bifurcation.  On the center manifold the flow satisfies an ODE in
four-dimensional space.  The reduction to normal form for Hopf-Hopf
bifurcations of ODEs is well known, and we follow Kuznetsov
\cite{Kuz-04-Book} to determine the type of Hopf-Hopf bifurcation that
occurs.

The algebraic steps to determine the normal form are detailed in
Appendix~\ref{app:NormalForm} in the supplementary materials, and we implemented our own Matlab code
\NormalForm which uses symbolic differentiation to compute the expansion of the state-dependent DDE
\eqref{eq:twostatedep} described in \sref{sec:expansion}, and then to
evaluate the normal form expressions for the resulting constant delay
DDE \eqref{eq:dde3}.
%to evaluate the resulting expressions.
To determine the
location of the codimension-two Hopf-Hopf points under consideration,
we start from an approximate location and solve for
$(\kappa_1,\kappa_2,\omega_1,\omega_2)$ so that the pair of
frequencies $\omega_1\ne\omega_2$ both solve \eq{hopflocation}
simultaneously for the same pair of parameter values
$(\kappa_1,\kappa_2)$. Our auxiliary routine \texttt{findHH} uses the Matlab
function \texttt{fminsearch} to minimise
$$f(\kappa_1,\omega_1,\kappa_2,\omega_2)=\sum_{j=1}^2\Bigl(\gamma+\kappa_1\cos(a_1\omega_j)+\kappa_2\cos(a_2\omega_j)\Bigr)^2
+\Bigl(\omega_j-\kappa_1\sin(a_1\omega_j)-\kappa_2\sin(a_2\omega_j)\Bigr)^2,$$
since this function contains the real and imaginary parts of two
copies of \eq{hopflocation}. In this way, we are able to find the
Hopf-Hopf point essentially to machine precision (we use tolerances of
$10^{-14}$).  At the Hopf-Hopf point we then evaluate the derivatives
and functions needed to obtain the center manifold coefficients
$g_{lsrk}^j$ in \sref{sec:code} of the supplemental materials, where we
employ symbolic
differentiation to avoid numerical errors.  Thus, we expect that our
normal form parameter calculations should be accurate essentially to
machine precision, and certainly to eight or more significant figures.

Recently, Wage \cite{Wage14} implemented an extension \texttt{ddebiftool\_nmfm} for DDE-BIFTOOL
to compute normal form coefficients at local bifurcations of steady
states in constant-delay DDEs. This applies a sun-star calculus based
normalisation technique to compute the normal form and center manifold
coefficients together, as elaborated for constant-delay DDEs by
Janssens~\cite{Janssens10}.
The DDE-BIFTOOL implementation only applies to constant-delay DDEs, and so cannot be applied directly to \eq{eq:twostatedep}. However, we can use DDE-BIFTOOL to compute the normal forms of the Hopf-Hopf points of the expanded constant-delay DDE \eq{eq:dde3}.
%\repl{The difference between this approach and
%our own implementation results in intermediate coefficients being
%scaled differently, but the final normal form coefficients computed by
%both methods should agree.}{
The difference between the DDE-BIFTOOL implementation (sun-star calculus approach to compute normal form and center manifold coefficients together) and our approach (center manifold reduction first, then compute normal form of resulting ODE system) results in intermediate coefficients being
scaled differently, but the final normal form coefficients computed by
both methods should agree. For the DDE-BIFTOOL computations it is
suggested to supply a user-defined routine to compute higher-order
derivatives. However, with nine delays in the constant-delay DDE
\eq{eq:dde3}, determining these derivatives would be a formidable
task, and so we use the default DDE-BIFTOOL finite-difference
derivative approximations. As an error control this computes the
normal form coefficients twice with finite difference approximations
of different order. However, in our experience this error estimate is
often misleading as the actual errors are usually much larger than the
estimate, as we will see in the next section.

%%%%%%%%%%%%%%%%%%%%%%%%%%%%%%%%%%%%%%%%%%%%%%%%%%%
\subsection{Hopf-Hopf normal forms}
\label{sec:NormalForm}

%%%%%%%%%%%%%%%%%%%%%%%%%%%%%%%%%%%%%%%%%%%%%%%%%%%
\begin{table}[t]
\begin{center}
\begin{tabular}{|c|c||c|c|c|c|} \hline
& Computed & \multicolumn{4}{|c|}{\rule{0pt}{2.3ex}DDE-BIFTOOL} \\ \cline{3-6}
& Normal Form & \rule{0pt}{2.3ex}$H_1$ High & $H_1$ Low & $H_u$ High & $H_u$ Low \\ \hline
$\kappa_1$ & 2.080920227069894 & \multicolumn{2}{|c|}{2.080905301795540} & \multicolumn{2}{|c|}{2.080662320398254} \\
$\kappa_2$ & 3.786800923405767 & \multicolumn{2}{|c|}{3.786811738802836} & \multicolumn{2}{|c|}{3.786929718494380} \\ \hline
$\omega_1$ & 2.487102830659818 & \multicolumn{2}{|c|}{2.487103286770640} & \multicolumn{2}{|c|}{1.582142631415513} \\
$\omega_2$ & 1.582152129599611 & \multicolumn{2}{|c|}{1.582151566193548} & \multicolumn{2}{|c|}{2.487110459273053} \\ \hline
$\vartheta$ & \phantom{-}5.291049995477200 & \!\phantom{-}5.2909997813\! & \!\phantom{-}5.2909980111\! & \!-0.0222756426\! &  \!-0.0222756534\! \\
$\delta$ & -0.022289571330147 & \!-0.0222816360\! & \!-0.0222817195 & \!\phantom{-}5.2909133110\! & \!\phantom{-}5.2909132195\!
\\ \hline
\end{tabular}
\end{center}
\caption{Values of $\kappa_i$ and $\omega_i$ at the Hopf-Hopf
bifurcation $\HH_1$, seen in \fref{fig:hopfcurves}, and the parameters
$\vartheta$ and $\delta$ that define the scaled truncated amplitude
equation \eqref{K8.112q}.  The values in the first column are computed
with our Matlab code \NormalForm  applied to \eq{eq:twostatedep}, which implements the procedure
described in Appendix~\ref{app:NormalForm}.  The other columns are
produced with the normal form extension of DDE-BIFTOOL, applied to the constant-delay DDE \eq{eq:dde3}
to obtain four different approximations, two on each of the two intersecting branches
of Hopf bifurcations, one from a low order approximation finite
difference approximation to the derivatives and one using a
higher-order approximation. The matlab code to generate all output is supplied in the Supplementary Materials.}
\label{tab:HH1}
\end{table}
%%%%%%%%%%%%%%%%%%%%%%%%%%%%%%%%%%%%%%%%%%%%%%%%%%%

%%%%%%%%%%%%%%%%%%%%%%%%%%%%%%%%%%%%%%%%%%%%%%%%%%%
\begin{table}[t]
\setlength{\tabcolsep}{3pt}
%* \hfill *\\
\begin{center}
\small
\begin{tabular}{|c|c|c|c|} \hline
& $\HH_1$ & $\HH_2$ & $\HH_3$ \\ \hline
$\kappa_1$ & 2.080920227069894 & $5.608860749294630$ & $9.284862308872761$ \\
$\kappa_2$ & 3.786800923405767 & $2.643352614515402$ & $4.403906490530705$ \\ \hline
$\omega_1$ & 2.487102830659818 & $6.608351858283422$ & $10.93073224661102\tr{9}$  \\
$\omega_2$ & 1.582152129599611 & $1.765757669232216$ & $1.952009077103193$ \\ \hline
\rule{0pt}{2.6ex}$\widetilde{g}_{2100}^1=\tfrac12g_{2100}^1$ &  $-0.81417665\tr{2897697} - 0.00407087\tr{5580656}i$ & $-8.59821703\tr{3379171} -10.3402562\tr{0151920}i$ & $\ph{-}8.25785960\tr{9402708} -81.8392092\tr{40317913}i$ \\
$\widetilde{g}_{1011}^1=g_{1011}^1$ & $-0.72563615\tr{0423584} + 0.26699379\tr{1853270}i$  & $-4.14512262\tr{4265354} - 0.48508142\tr{7768031}i$ & $-20.2850232\tr{60567193} +11.4745454\tr{26007372}i$ \\
$\widetilde{g}_{1110}^2=g_{1110}^2$ & $-0.45302394\tr{1909892} - 0.29997922\tr{8722909}i$ & $\phantom{-}1.74982076\tr{7058454} - 7.92866388\tr{4389642}i$ & $\ph{-}31.0314747\tr{69824001} -74.3567344\tr{87707911}i$ \\
\rule[-1.2ex]{0pt}{0pt}$\widetilde{g}_{0021}^2=\tfrac12g_{0021}^2$ & $-0.13405924\tr{9130964} - 0.29906145\tr{8302696}i$ & $-1.42981504\tr{4056368} - 0.22951923\tr{5629436}i$ & $-0.26054578\tr{7811855} - 0.38071817\tr{0546475}i$ \\ \hline
\rule{0pt}{2.6ex}$G_{2100}^1(0)$ & $-0.69871613\tr{3454478} - 0.28257330\tr{2492288}i$ & $-7.50609582\tr{7847942} - 4.15081310\tr{4539677}i$ & $-16.8534773\tr{87548306} -28.0243853\tr{52455454}i$ \\
$G_{1011}^1(0)$ & $-0.51573055\tr{8790600} - 0.23247968\tr{6193008}i$ & $-5.26325881\tr{5778850} + 0.05175630\tr{9540828}i$ & $-21.3834727\tr{31029173} +12.4878724\tr{82305509}i$ \\
$G_{1110}^2(0)$ & $\phantom{-}0.01557408\tr{3096157} - 0.46117993\tr{2732848}i$  & $\phantom{-}5.55956094\tr{1739460} - 2.01536072\tr{4708596}i$ & $\ph{-}50.3666025\tr{28818383} -66.5262024\tr{69092752}i$ \\
$G_{0021}^2(0)$ & $-0.09747225\tr{2054214} - 0.22785268\tr{9753873}i$  & $-0.65677277\tr{0545076} - 0.20185598\tr{4151354}i$ & $-0.20383503\tr{6817261} + 0.19032437\tr{4649657}i$ \\ \hline
$p_{11}$ & $-0.698716133454477$  & $-7.506095827847883$ & $-16.853477387548608$ \\
$p_{12}$ & $-0.515730558790600$  & $-5.263258815778782$ & $-21.383472731028913$  \\
$p_{21}$ & $\phantom{-}0.015574083096158$  & $\phantom{-}5.559560941739119$ & $\ph{-}50.366602528819492$ \\
$p_{22}$ & $-0.097472252054214$  & $-0.656772770545075$ & $-0.2038350368172633\tr{2}$ \\  \hline
$\vartheta$ & $\phantom{-}5.291049995477200$  & $\ph{-}8.013820078762780$ & $\ph{-}104.90577608695922$ \\
$\delta$ & $-0.022289571330147$  & $-0.740672790388973$ & $-2.9884991311069409\tr{7}$
\\ \hline
\end{tabular}
\caption{The locations and the main normal form and amplitude equation
parameters at the three Hopf-Hopf points $\HH_j$ shown in
\fref{fig:hopfcurves}, computed with our Matlab code
\NormalForm.}
\label{tab:HHj}
\end{center}
\end{table}
%%%%%%%%%%%%%%%%%%%%%%%%%%%%%%%%%%%%%%%%%%%%%%%%%%%

We perform the normal form analysis for the parameter values given
in \eq{eq:oriannaparams}, which are the same as used in
\fref{fig:hopfcurves} and throughout this paper.  For these parameter
values the locations of the Hopf-Hopf points and the resulting normal
form parameters can be found as described in the previous section.

In Table~\ref{tab:HH1} we state the results of five different
computations for the first Hopf-Hopf point $\HH_1$.  The normal form
parameters $\vartheta$ and $\delta$ define coefficients in the scaled
truncated amplitude equations
\begin{equation} \label{K8.112q}
\begin{split}
\xi_1' &= \xi_1(\mu_1 -\xi_1-\vartheta\xi_2),\\
\xi_2' &= \xi_2(\mu_2 -\xi_2-\delta\xi_1),
\end{split}
\end{equation}
for $\xi_j\geq0$,
which determine the dynamics and bifurcations seen as
$\mu_j=\Re(\lambda_j)$ are varied close to the Hopf-Hopf point where
$\mu_1=\mu_2=0$. The derivation of \eq{K8.112q} is given in
Appendix~\ref{app:NormalForm}, culminating in equation \eq{K8.112}.

The first column of Table~\ref{tab:HH1} gives the values computed with
our \NormalForm code described in
\srefs{sec:expansion}-\ref{sec:centreman}; for comparison the other
columns give values computed with DDE-BIFTOOL's normal form
extension. DDE-BIFTOOL finds Hopf-Hopf points by checking along a
branch of Hopf bifurcations for where a second pair of characteristic
values crosses the imaginary axis. Thus, with DDE-BIFTOOL, it is
possible to obtain two different approximations to the same Hopf-Hopf
point by searching along each of the two intersecting branches of Hopf
points; in Table~\ref{tab:HH1} we give the locations of $\HH_1$ found
on the Hopf curves $H_1$ and $H_u$ (see \fref{fig:hopfcurves}). As
noted in \sref{sec:centreman}, when computing derivatives via
finite-differences, DDE-BIFTOOL provides two different
finite-difference approximations to give an indication of the
error. The parameters $\vartheta$ and $\delta$ computed on $H_1$ with
the two different finite difference approximations agree to a relative
error of about $10^{-6}$, indicating that the finite-difference
approximations are both quite accurate, and similarly on the branch
$H_u$.  However, the agreement is not so good when we compare the
answers obtained on the two branches. Firstly, we see that the values
of $\vartheta$ and $\delta$ are swapped on the two branches, which is
correct and natural.  DDE-BIFTOOL takes as $\omega_1$ the value of
$\omega$ for the Hopf bifurcation occurring on the branch one is
searching along, and takes as $\omega_2$ the value of $\omega$ for the
second pair of characteristic values crossing the imaginary
axis. Hence, the values of $\omega_1$ and $\omega_2$ are swapped when
the search is switched from one branch to the other, and this results
in the values of $\vartheta$ and $\delta$ also being swapped. However,
even after swapping, we see that the values of $\vartheta$ and
$\delta$ calculated by DDE-BIFTOOL only agree to about four
significant figures between the two branches. This also indicates the
relative accuracy to which the values of $\kappa_1$, $\kappa_2$,
$\omega_1$ and $\omega_2$ for the Hopf-Hopf point agree on the two
branches.  So it seems that the accuracy of the DDE-BIFTOOL computed
normal forms is limited by the accuracy to which DDE-BIFTOOL computes
the location of the Hopf-Hopf points, and not by the accuracy to which
it computes the normal forms themselves.

We can also swap the $\omega_j$ in the computation of the normal forms
in our code \NormalForm.  Because of the symmetry between the
parameters, for the index $j=1$ or $2$ so that $3-j$ indicates the
other index, swapping the $\omega$ values $\omega_j \leftrightarrow
\omega_{3-j}$ exchanges $\vartheta$ and $\delta$ and the other normal
form coefficients (see
Appendix~\ref{app:NormalForm}) as follows:
$$g_{lsrk}^j \leftrightarrow g_{rkls}^{3-j}, \quad \widetilde{g}_{lsrk}^j
\leftrightarrow \widetilde{g}_{rkls}^{3-j}, \quad G_{lsrk}^j
\leftrightarrow G_{rkls}^{3-j}, \quad p_{ij} \leftrightarrow p_{3-i3-j}.$$
Because we find the
Hopf-Hopf point to machine precision and evaluate the derivatives
symbolically, when the $\omega_j$ are exchanged, we find that the
respective normal form coefficients are identical to machine
precision. In fact, the idea of swapping the $\omega_j$ and checking
the normal form coefficients and parameters turned out to be very
useful during the checking and debugging of our code.

Table~\ref{tab:HHj} gives the normal form parameters for the first
three Hopf-Hopf points $\HH_j$ seen in \fref{fig:hopfcurves}, and also
some of the more important intermediate coefficients described in
Appendix~\ref{app:NormalForm}. Here we report only one set of normal
form parameters for each Hopf-Hopf point $\HH_j$ computed with our
Matlab code \NormalForm.  We always take $\omega_1 > \omega_2$ and,
since the period of the periodic orbit bifurcating from the curve
$H_u$ is always the largest, this corresponds to taking $\omega_1$ as
the frequency of the Hopf bifurcation $H_j$ for $j=1,2$ or $3$, and
$\omega_2$ as the frequency of the Hopf bifurcation $H_u$.
Our normal form calculations give the following overall result.

\begin{propose} \label{prop:nf}
At each of the three Hopf-Hopf points $\HH_1$, $\HH_2$ and $\HH_2$
\begin{itemize}
\item[(i)]
$p_{11}<0$ and $p_{22}<0$, which means that normal form coefficients
$\vartheta$ and $\delta$ are sufficient to
determine the type of the Hopf-Hopf bifurcation that occurs
\cite{Kuz-04-Book};
\item[(ii)]
the non-degeneracy conditions \textrm(HH.0)-\textrm(HH.6) in
Appendix~\ref{app:NormalForm} hold; and
\item[(iii)]
$\vartheta>0>\delta$, which corresponds
to subcase III of the simple case as described in Sec.~8.6.2 of
\cite{Kuz-04-Book}; see also Appendix~\ref{sec:DHopfUnfold}.
\end{itemize}
\end{propose}

In the normal form parameters plane of
$(\mu_1,\mu_2)=(\Re(\lambda_1),\Re(\lambda_2))$, Hopf bifurcations
occur along the horizontal $\mu_1$-axis with the bifurcating periodic
orbit existing in the upper half plane, and along the vertical
$\mu_2$-axis with the bifurcating periodic orbit existing in the right
half plane.  Proposition~\ref{prop:nf} implies that there are
two curves of torus bifurcations emerging from
the origin, which is the codimension-two Hopf-Hopf point: one in the
first quadrant and one in the fourth quadrant, with the torus existing
in the convex cone between them. On the upper torus bifurcation curve
the torus bifurcates from the periodic orbit that exists in the upper
half plane, and on the lower torus bifurcation curve it bifurcates
from the periodic orbit which exists in the right half plane.  The
five regions of generic phase portraits are labelled in panel III of
Fig.~8.25 in \cite{Kuz-04-Book} (but notice a typo: 13 should be 12), and
the corresponding generic phase portraits are given in Fig.~8.26 of
\cite{Kuz-04-Book}.

%%%%%%%%%%%%%%%%%%%%%%%%%%%%%%%%%%%%%%%%%%%%%%%%%%%
\begin{figure}[t!]
\begin{center}
\includegraphics{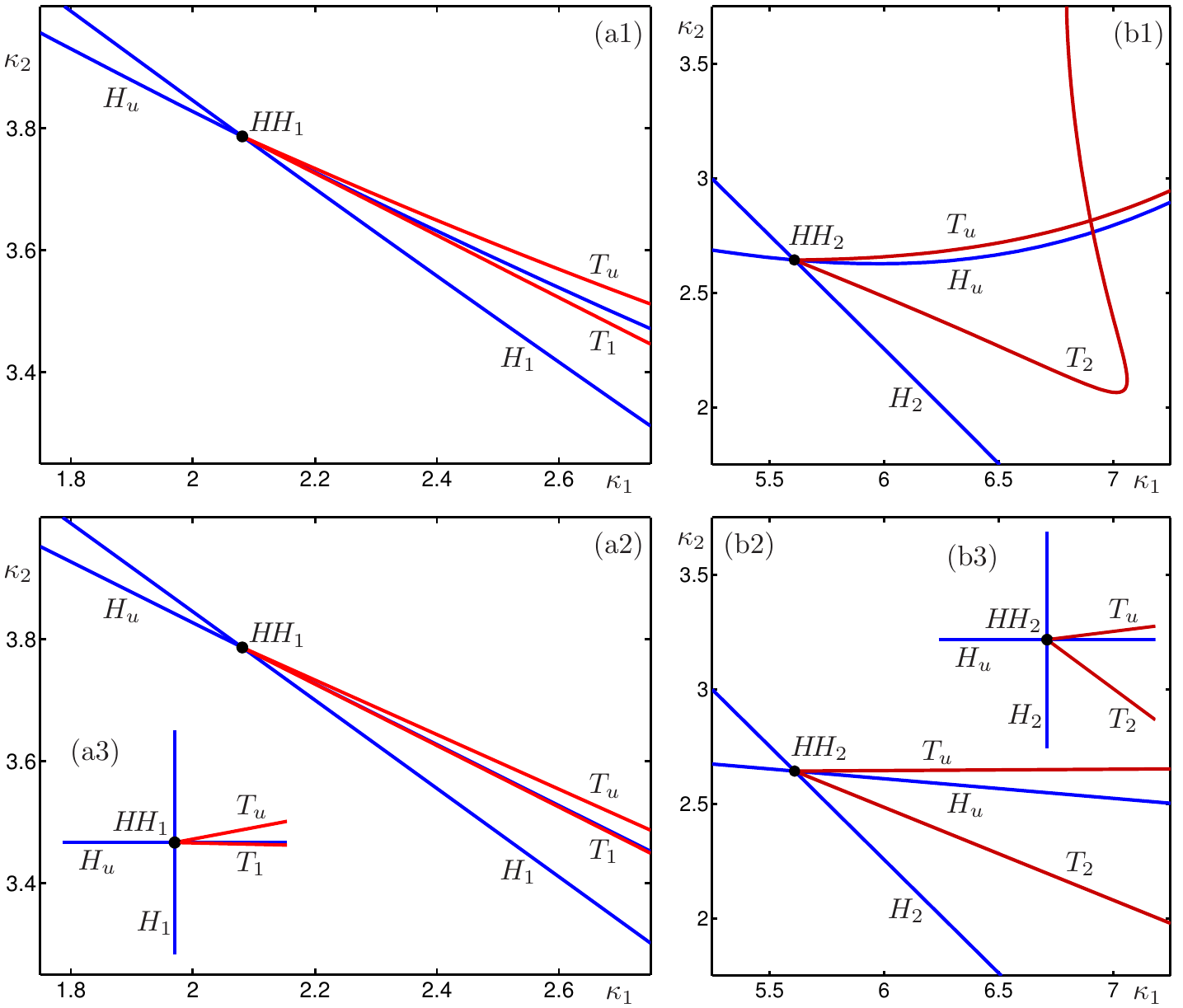}
\caption{Comparison in the $(\kappa_1,\kappa_2)$-plane near $\HH_1$
and $\HH_2$ between numerically computed torus bifurcation curves for
the state-dependent DDE \eqref{eq:twostatedep} in panels (a1) and
(b1), and their linear approximations in panels (a2) and (b2) obtained
by evaluating the normal form coefficients at the respective Hopf-Hopf
point and applying the coordinate transformation \eq{invtransf}. The
inset panels (a3) and (b3) show the $(\mu_1,\mu_2)$-plane of the
normal form \eq{K8.112q} before this transformation.}
\label{fig:doublehopf}
\end{center}
\end{figure}
%%%%%%%%%%%%%%%%%%%%%%%%%%%%%%%%%%%%%%%%%%%%%%%%%%%

\Fref{fig:doublehopf} shows how our normal form calculations manifest
themselves near $\HH_1$ and $\HH_2$. Panels (a1) and (b1) show the
local bifurcation diagrams of the original state-dependent DDE
\eq{eq:twostatedep} as computed with DDE-BIFTOOL \cite{NewDDEBiftool},
consisting of the Hopf bifurcation curve $H_u$ intersecting the Hopf
bifurcation curves $H_1$ and $H_2$ in $\HH_1$ and $\HH_2$ (as in
\fref{fig:hopfcurves}), as well as the associated torus bifurcation
curves $T_u$, $T_1$ and $T_2$. Panels (a2) and (a3) and panels (b2)
and (b3) of \fref{fig:doublehopf} show the results of our normal form
calculations at $\HH_1$ and $\HH_2$, respectively. Panels (a3) and
(b3) show the positions of the curves of torus bifurcation in the
$(\mu_1,\mu_2)$-plane of the normal form \eq{K8.112q}. As was
discussed, $T_u$ lies in the first quadrant and the curves $T_1$ and
$T_2$ each lie in the fourth quadrant. Moreover, the normal form
calculations also give the slope of the torus curves in the
$(\mu_1,\mu_2)$-plane via the actual values of $\vartheta$ and $\delta$
and \eq{eq:T1} and \eq{eq:T2}. In particular, $T_1$ lies very close to
$H_u$ near $\HH_1$ in panel (a3), while $T_2$ is well separated from
$H_u$ near $\HH_2$ in panel (b3). Since the Jacobian matrix defined in
nondegeneracy condition \textrm(HH.6) in
Appendix~\ref{sec:DHopfUnfold} is invertible at each point $\HH_j$, we
can use the coordinate transformation \eq{invtransf} to map the
$(\mu_1,\mu_2)$-plane back to the $(\kappa_1,\kappa_2)$-plane of
\eqref{eq:twostatedep}. The result is shown in panels (a2) and (b2) of
\fref{fig:doublehopf}, where all curves are actually straight lines
that represent the linear approximations, that is, the slopes, of the
respective Hopf and torus bifurcation curves near $\HH_1$ and
$\HH_2$. There is excellent correspondence between the nature, order
and slopes of the respective bifurcation curves illustrated in panels
(a1) and (a2) and in panels (b1) and (b2), respectively.
%This fact is
%clear evidence for our assertion that we indeed computed the
%correct Hopf-Hopf normal form of the full state-dependent DDE
%\eqref{eq:twostatedep}.
This fact is clear evidence, over and above the two independent
normal form calculations, that Proposition~\ref{prop:nf} is
correct and indeed represents the Hopf-Hopf normal form of
the full state-dependent DDE \eqref{eq:twostatedep}.

Clearly, the bifurcation curves in the local bifurcation diagrams in
\fref{fig:doublehopf}(a1) and (b1) are actually nonlinear, and this
explains the visible differences with panels (a2) and (b2) further
away from $\HH_1$ and $\HH_2$, respectively. The curvature of the
these bifurcation curves could be captured by computing higher-order
terms in the normal forms, but this is very cumbersome and rarely
done. Rather, we will continue these bifurcation curves numerically
with DDE-BIFTOOL more globally throughout the
$(\kappa_1,\kappa_2)$-plane.  As we will see in the next section, the
full bifurcation diagram is very complicated.

%%%%%%%%%%%%%%%%%%%%%%%%%%%%%%%%%%%%%%%%%%%%%%%%%%%
%%%%%%%%%%%%%%%%%%%%%%%%%%%%%%%%%%%%%%%%%%%%%%%%%%%
\section{Structure of bifurcating tori}
\label{sec:tori}

The existence of Hopf-Hopf bifurcation points that give rise to
torus bifurcation curves clearly indicates that \eqref{eq:twostatedep}
should feature multi-frequency dynamics and, in particular,
quasi-periodic and locked dynamics on invariant tori.

%%%%%%%%%%%%%%%%%%%%%%%%%%%%%%%%%%%%%%%%%%%%%%%%%%%
\begin{figure}[t!]
\begin{center}
\includegraphics{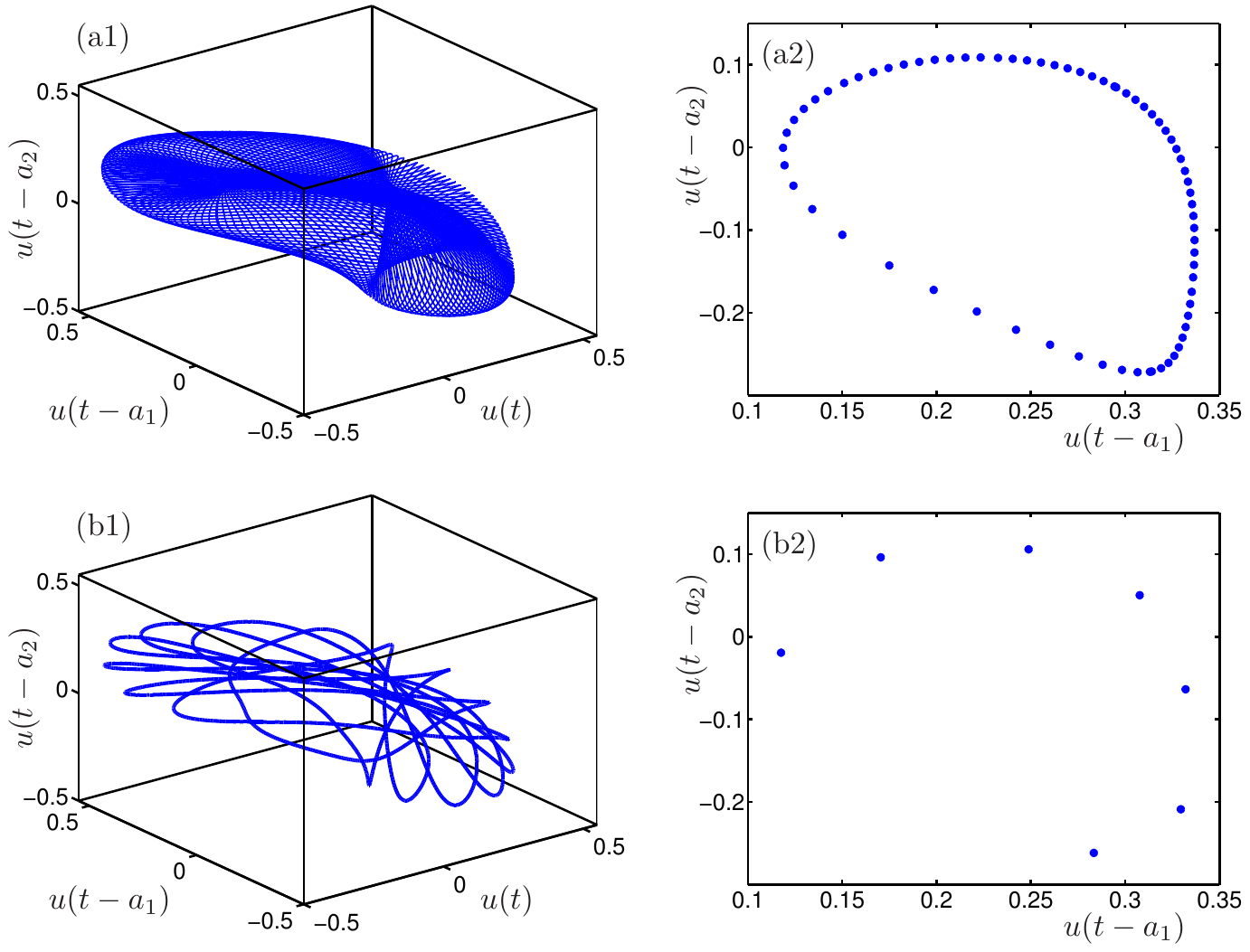}
\caption{Quasi-periodic torus for $\kappa_1=4.44$ in row (a) and
$3\!:\!7$ phase-locked periodic orbit for $\kappa_1=4.409556$ in row
(b), where $\kappa_2=3.0$. Panels (a1) and (b1) show projections onto
$(u(t),u(t-a_1),u(t-a_2))$-space, and panels (a2) and (b2) the
trace in the $(u(t-a_1),u(t-a_2))$-plane of the
Poincar{\'e} return map defined by $u(t)=0$.}
\label{fig:simtori}
\end{center}
\end{figure}
%%%%%%%%%%%%%%%%%%%%%%%%%%%%%%%%%%%%%%%%%%%%%%%%%%%

\Fref{fig:simtori} shows two examples of dynamics on an invariant
torus, which were obtained by numerical integration of
\eqref{eq:twostatedep} and after transients have been allowed to die
down. The respective dynamics on the torus are illustrated in the left
column in projection onto the $(u(t),u(t-a_1),u(t-a_2))$-space.  The
right column shows points in the $(u(t-a_1),u(t-a_2))$-plane whenever
$u(t)=0$. In other words, it shows a two-dimensional projection of the
function segments of the Poincar{\'e} return map defined by
$u(t)=0$. This representation in the $(u(t-a_1),u(t-a_2))$-plane has
been chosen to give a good impression of the low-dimensional character
of the tori we encounter, and we refer to it as the Poincar{\'e} trace
for short; see below for more details on how to construct a Poincar\'e
map of a DDE.  In \fref{fig:simtori}(a) the dynamics are
quasi-periodic (or of very high period) so that the shown single
trajectory covers the torus densely; in the Poincar\'e trace this
corresponds to an invariant closed curve, which is filled out denser
and denser as a longer trajectory is computed. An example of locked
dynamics on the torus is given in row (b) of \fref{fig:simtori}. More
specifically, shown is the attracting periodic orbit on the torus (not
shown) in projection onto $(u(t),u(t-a_1),u(t-a_2))$-space in panel
(b1), and the associated Poincar\'e trace in the
$(u(t-a_1),u(t-a_2))$-plane in panel (b2). They show that the locked
periodic orbit forms a $3\!:\!7$ torus knot.

Overall, \fref{fig:simtori} illustrates that two-dimensional invariant
tori of \eqref{eq:twostatedep} can be represented conveniently in
projection onto the three-dimensional $(u(t),u(t-a_1),u(t-a_2))$-space
and by their Poincar{\'e} trace in the $(u(t-a_1),u(t-a_2))$-plane. We
now discuss the choice of Poincar\'e map for the state-dependent
scalar DDE \eqref{eq:twostatedep} in somewhat more detail.  It is easy
to see that $u\equiv 0$ is the unique steady state of
\eqref{eq:twostatedep}.  Equation \eq{eq:alphabd} and the positivity
of the parameters implies that any orbit that does not cross $u=0$
will be eventually monotonic, and also that $u(t)$ and $u'(t)$ cannot
have the same sign on a time interval longer than $\tau$.  Hence,
since all periodic and quasi-periodic orbits cross $u=0$, it is
natural to use this condition for defining the Poincar\'e map.  More
specifically, we define the Poincar\'e section \be \label{eq:Sigma}
\Sigma=\{\phi\in C: \phi(0)=0 \}, \ee which is a codimension-one
subspace of the infinite-dimensional phase space $C$ of
\eqref{eq:twostatedep}. Hence, $\Sigma$ is infinite dimensional
itself, and the local Poincar\'e map $P_\Sigma$ on $\Sigma$ is defined
as the map that takes a downward transversal crossing of zero
($\phi(0)=0$ with $\phi'(0)<0$) to the next such crossing.  The
infinite dimensionality of $\Sigma$ obscures the structure of the
low-dimensional invariant sets (namely periodic orbits and tori) we
wish to visualize, which is why one considers projections of $C$ and,
hence, $\Sigma$.

We consider the projection $\cP: C\to\R^3$ via \be \label{eq:P} \cP
u_t=(u_t(0),u_t(-a_1),u_t(-a_2))=(u(t),u(t-a_1),u(t-a_2)) \in \R^3,
\ee with corresponding projection \be \label{eq:Sigma3}
\cP_\Sigma=\{(0,u(t-a_1),u(t-a_2))\} \cong \{u(t-a_1),u(t-a_2))\} =
\R^2.  \ee This generalises an idea of Mackey and Glass
\cite{Mac-Gla-77}, who were the first to project solutions of DDEs
into finite dimensions by plotting values of $u(t-\tau)$ against
$u(t)$ for a single delay DDE.

For simplicity, we refer to the projected Poincar\'e section also as
$\Sigma$ and, throughout, we consider the invariant objects of the
local Poincar\'e map $P_\Sigma$ defined for points with $u(t)=0$ and
$u'(t)<0$ (to ensure that there is a unique intersection set for
periodic orbits and tori). As was already mentioned, we refer to the
respective intersection set in the $(u(t-a_1),u(t-a_2))$-plane as the
Poincar\'e trace of the invariant object.

We remark that, when the DDE has a sufficient number $d$ of
independent variables (at least three), a convenient alternative
projection from $C$ to $\R^d$ is obtained by projecting the function
segment $u_t\in C$ onto its head-point $u_t(0)=u(t)\in\R^d$.  See
\cite{Gre-Kra-Eng-02, Kra-Gre-03} for an example of this construction
for a laser system with $d=3$. However, this approach is not useful
for visualising the dynamics of \eqref{eq:twostatedep} because $u_t$
is scalar.

%%%%%%%%%%%%%%%%%%%%%%%%%%%%%%%%%%%%%%%%%%%%%%%%%%%
\begin{figure}[t!]
\begin{center}
\includegraphics{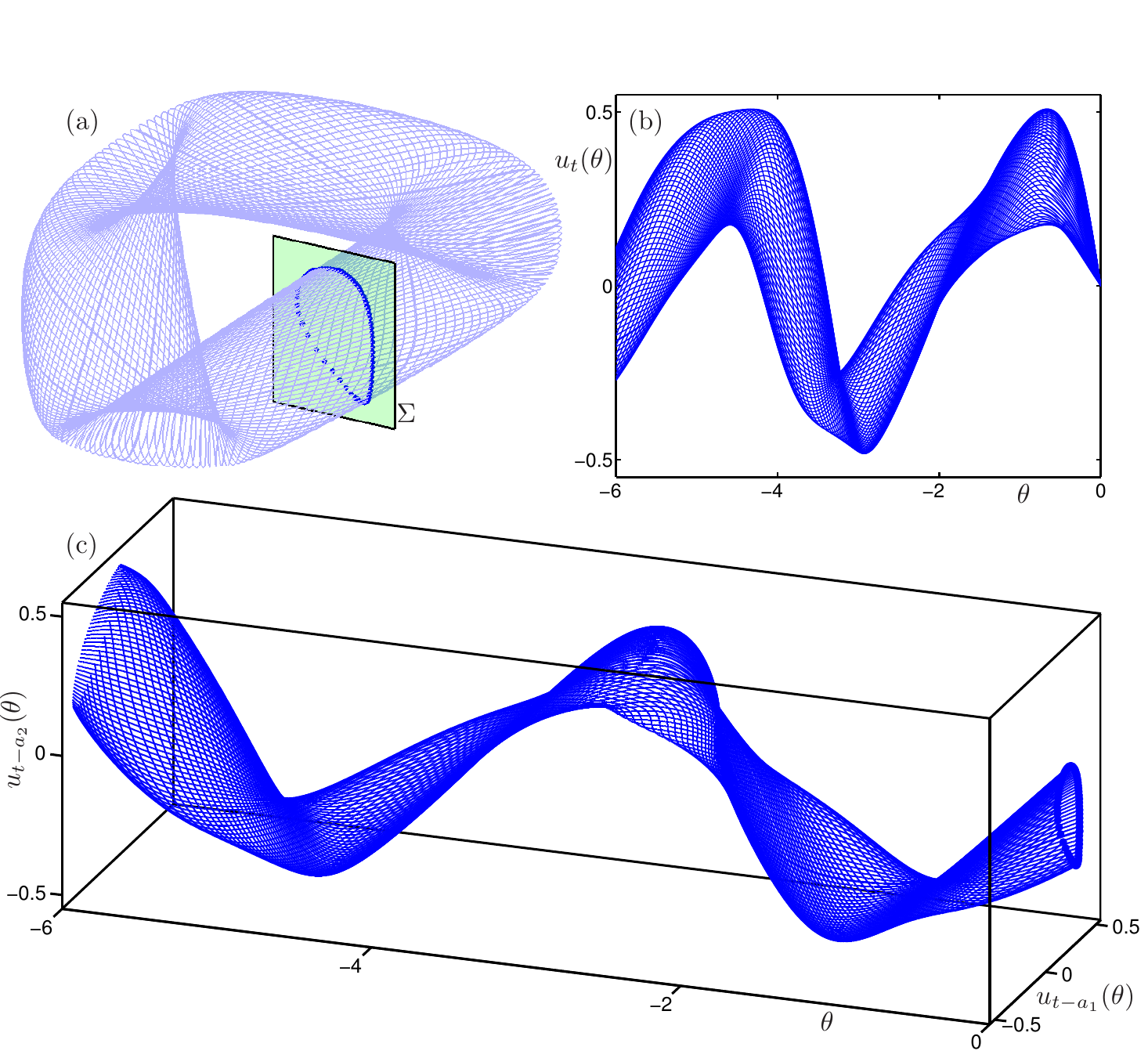}
\caption{Illustration of Poincar{\'e} section and trace for the
quasi-periodic torus for $\kappa_1=4.44$ and $\kappa_2=3.0$ from
\fref{fig:simtori}(a). Panel (a) shows the projection onto
$(u(t),u(t-a_1),u(t-a_2))$-space of a single solution of
\eqref{eq:twostatedep} on the torus (light blue), together with the
trace (blue dots) on the (projected) section $\Sigma$ (green); the
corresponding function segments are shown in panel (b) as functions
$u_t$, and in panel (c) as function segments
$(u_{t-a_1}(\theta),u_{t-a_2}(\theta))$, over the delay interval
$\theta\in[-6,0]$, respectively. In panel (c) the Poincar\'e trace is
seen in the plane for $\theta=0$, which corresponds to $\Sigma$.}
\label{fig:torusqp}
\end{center}
\end{figure}
%%%%%%%%%%%%%%%%%%%%%%%%%%%%%%%%%%%%%%%%%%%%%%%%%%%

\Fref{fig:torusqp} illustrates the different projections and
representations with the example of the quasi-periodic torus from
\fref{fig:simtori}(a).  \Fref{fig:torusqp}(a) shows a different view
of the torus in $(u(t),u(t-a_1),u(t-a_2))$-space together with the
Poincar\'e trace in the local section $\Sigma$. This image is very
similar to illustrations one finds in the literature of quasi-periodic
tori of three-dimensional vector fields; in particular, the torus
appears to be smooth and the intersection curve with $\Sigma$ is a
smooth simple closed curve. That we are in fact dealing with a scalar
state-dependent DDE with an infinite dimensional phase space is
illustrated in panels (b) and (c). \Fref{fig:torusqp}(b) shows the
function segments $u_t(\theta)$ corresponding to all the points of the
Poincar\'e trace on $\Sigma$ in the $u(t-a_1),u(t-a_2))$-plane in
panel (a).  That is, the functions segments for the points on the
torus with $u(t)=0$ (or equivalently $u_t(0)=0$) and $u'(t)<0$. Note
that, because the section $\Sigma$ is defined by the condition
$u(t)=0$, all these function segments are defined over the same fixed
time interval $[-a_2, 0] = [-6,0]$, and all end up at the same point
$u(0)=0$.  \Fref{fig:torusqp}(c) shows a different representation of
the function segments associated with the points of the Poincar\'e
trace, with the function segments
$(u_{t-a_1}(\theta),u_{t-a_2}(\theta))$ illustrating the `history
tails' over the time interval $[-6,0]$ associated with the trace in
(the two-dimensional projection of) $\Sigma$.  Notice that in this
representation the invariant torus appears as a cylinder that is swept
out by the function segments in the corresponding orbit under the
local Poincar\'e map $P_\Sigma$, with the Poincar\'e trace seen in the
plane for $\theta=0$ in \fref{fig:torusqp}(c).

%%%%%%%%%%%%%%%%%%%%%%%%%%%%%%%%%%%%%%%%%%%%%%%%%%%
\begin{figure}[t!]
\begin{center}
\includegraphics{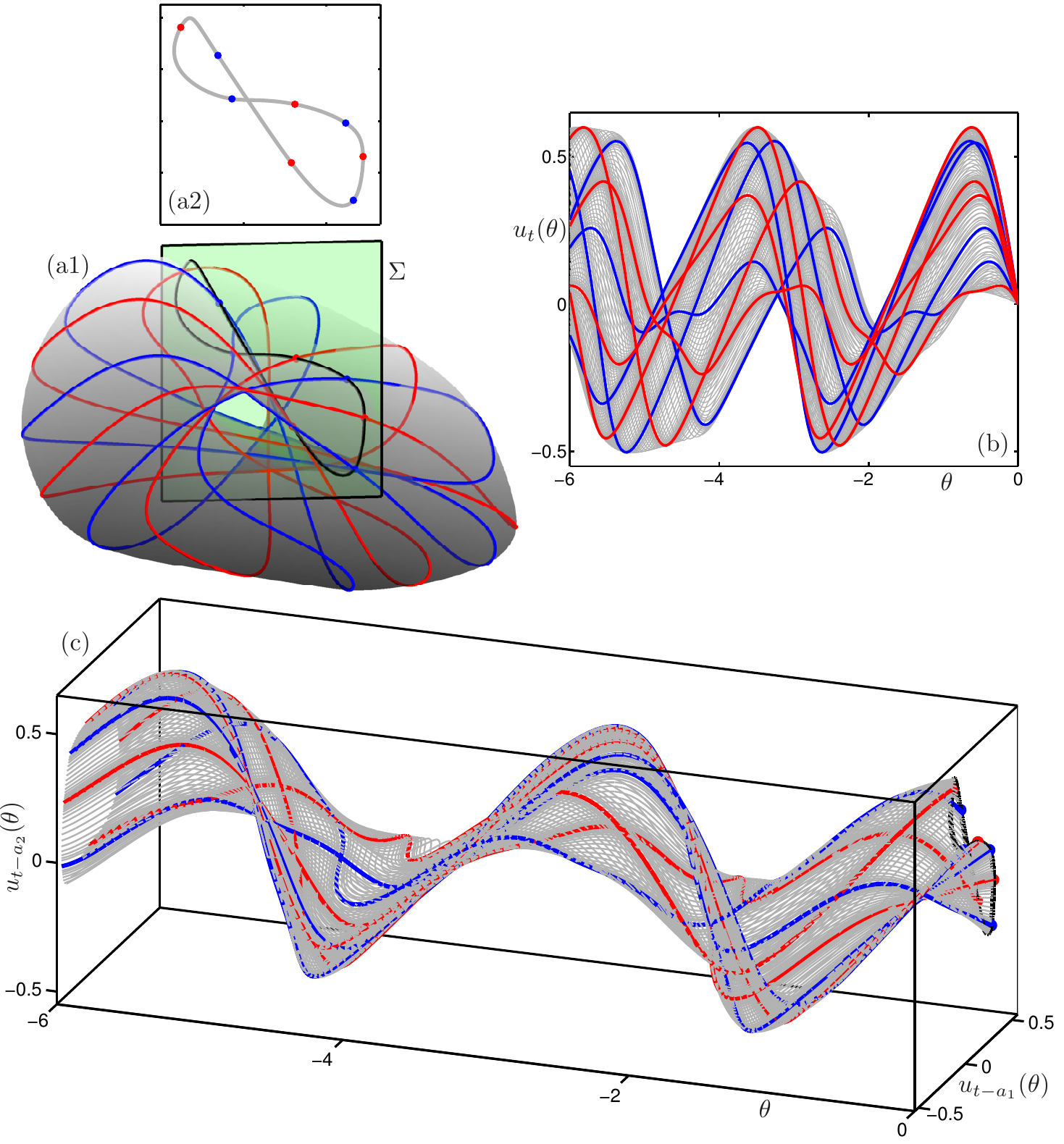}
\caption{Illustration of Poincar{\'e} section and trace for the
$1\!:\!4$ phase-locked torus for $\kappa_1=5.405$ and $\kappa_2=2.45$.
Panel (a1) shows the projection onto $(u(t),u(t-a_1),u(t-a_2))$-space
of the relevant invariant objects, namely, of the stable periodic
orbit (blue), the saddle periodic orbit (red), its unstable manifold
(grey curve), together with the trace on the (projected) section
$\Sigma$ (green).  Panel (a2) shows only the trace of these objects in
$\Sigma$.  The corresponding function segments are shown in panel (b)
as functions $u_t$, and in panel (c) as function segments
$(u_{t-a_1}(\theta),u_{t-a_2}(\theta))$, over the delay interval
$\theta\in[-6,0]$, respectively.
}
\label{fig:toruslocked}
\end{center}
\end{figure}
%%%%%%%%%%%%%%%%%%%%%%%%%%%%%%%%%%%%%%%%%%%%%%%%%%%

\Fref{fig:toruslocked} shows an example of a smooth invariant torus
with $1\!:\!4$ phase-locked dynamics on it.  In panels (a)--(c) the
torus is represented in the same manner as the quasi-periodic torus in
\fref{fig:torusqp}.  However, in contrast to \fref{fig:simtori}(b)
that only shows the locked stable periodic orbit on the torus,
\fref{fig:toruslocked} also shows the unstable locked periodic orbit
and its two-dimensional unstable manifold, which together form the
locked invariant torus itself. \fref{fig:toruslocked}(a1) shows the
torus rendered as a surface in $(u(t),u(t-a_1),u(t-a_2))$-space with
the stable and unstable locked periodic orbits lying on it. Also shown
is the section $\Sigma$ and the Poincar\'e trace; for clarity, the
trace is shown on its own in the $(u(t-a_1),u(t-a_2))$-plane in panel
(a2). Associated segments $u_t$ are shown as functions of $\theta$ in
\fref{fig:toruslocked}(b), and in terms of $(u_{t-a_1}(\theta),
u_{t-a_2}(\theta))$ in \fref{fig:toruslocked}(c).

The torus in \fref{fig:toruslocked} gives rise to a single smooth
curve as the trace in the $(u(t-a_1),u(t-a_2))$-plane, on which lie
four points of a stable period-four orbit and four points of an
unstable period-four orbit; see \fref{fig:toruslocked}(a2). The stable
periodic orbit was found by numerical simulation. It was then used to
start a continuation of the periodic orbit in the parameter $\kappa_1$
which yielded, after a fold or saddle-node bifurcation of periodic
orbits, the unstable periodic orbit. This calculation also confirmed
that, as theory predicts, the unstable periodic orbit has exactly one
unstable Floquet multiplier. We extracted the unstable eigenfunction
associated with the unstable periodic orbit on the torus and used it
to define two initial functions in the local unstable manifold of the
periodic orbit (one on each side of the orbit). Then numerical
integration near the periodic point and along the unstable
eigenfunction was used to compute trajectories that lie on the
unstable manifold; associated orbit segments are shown in
\fref{fig:toruslocked}(b) and (c). Careful selection and ordering of
orbit segments on the unstable manifolds (between intersections with
the Poincar\'e section) allowed us to render the torus as a surface in
$(u(t),u(t-a_1),u(t-a_2))$-space in \fref{fig:toruslocked}(a1), and to
draw the corresponding one-dimensional curve in the
$(u(t-a_1),u(t-a_2))$-plane in \fref{fig:toruslocked}(a2).

Again, the representation of locked dynamics on the torus in
\fref{fig:toruslocked} is very reminiscent of what one would expect to
find in a three-dimensional vector field. Notice, however, that --- in
contrast to the quasi-periodic torus in \fref{fig:torusqp} --- the
invariant curve in the $(u(t-a_1),u(t-a_2))$-plane has a point of
self-intersection. The torus in $(u(t),u(t-a_1),u(t-a_2))$-space also
has a curve of self-intersection; see \fref{fig:toruslocked}(c).  This
is due to projection from the infinite-dimensional phase space and a
reminder that we are dealing with a DDE and not a low-dimensional
dynamical system. While self-intersections may occur, we believe that
the chosen Poincar\'e section $\Sigma$ defined by $u(t)=0$ is the most
convenient and natural choice for the study of multi-frequency
dynamics in \eqref{eq:twostatedep}.

%%%%%%%%%%%%%%%%%%%%%%%%%%%%%%%%%%%%%%%%%%%%%%%%%%%
\subsection{Resonance tongues and locked tori}
\label{sec:resonance}

%%%%%%%%%%%%%%%%%%%%%%%%%%%%%%%%%%%%%%%%%%%%%%%%%%%
\begin{figure}[t!]
\begin{center}
\includegraphics{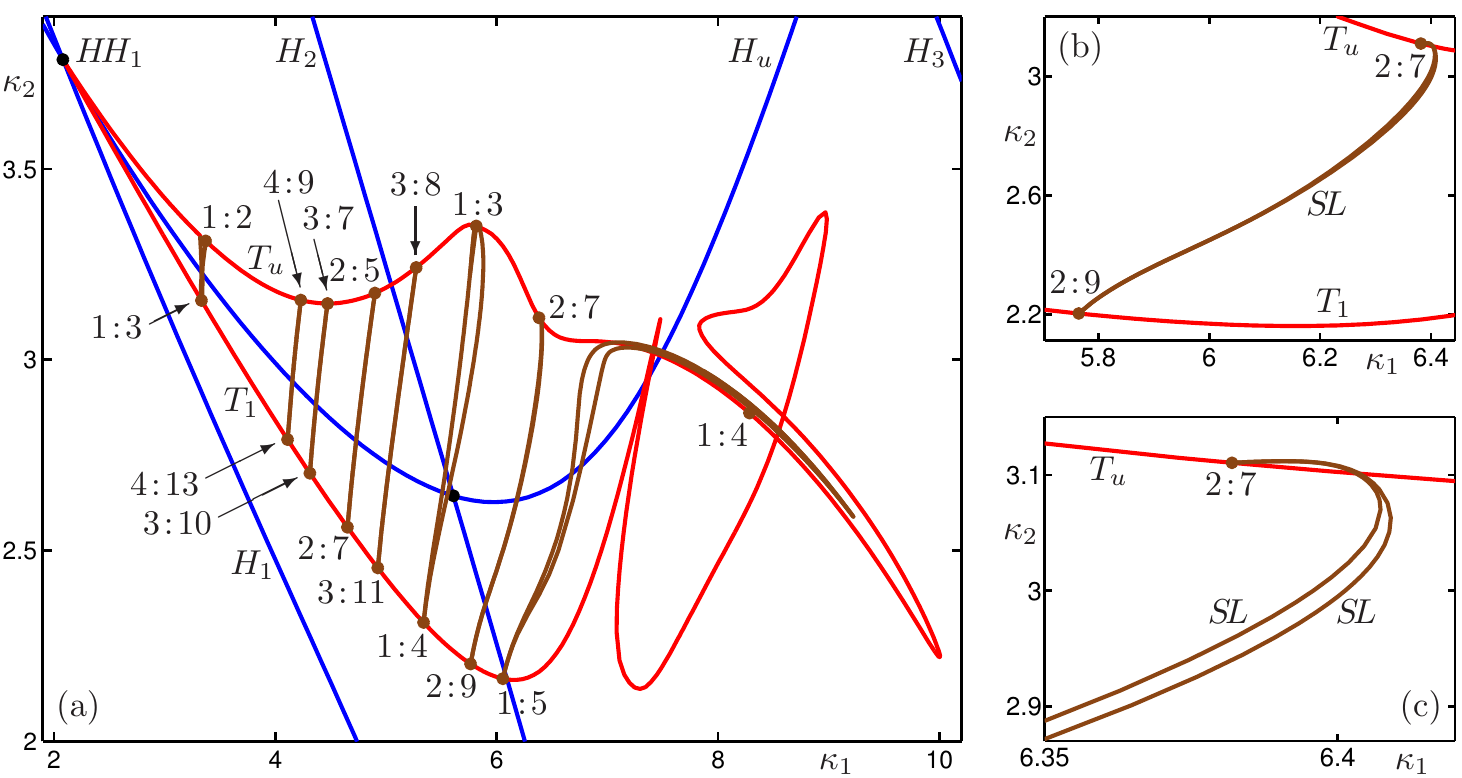}
\caption{The torus bifurcation curves $T_u$ and $T_1$ emerging from
the Hopf-Hopf bifurcation point $\HH_1$ and associated resonance
tongues in the $(\kappa_1,\kappa_2)$-plane (a). Panels (b) and (c) are
successive enlargements of the resonance tongue that connects a
$2\!:\!7$ resonance on $T_u$ with a $2\!:\!9$ resonance on $T_1$.}
\label{fig:biftori}
\end{center}
\end{figure}
%%%%%%%%%%%%%%%%%%%%%%%%%%%%%%%%%%%%%%%%%%%%%%%%%%%

Continuation of the two torus bifurcation curves that are known to
emerge from the Hopf-Hopf point $\HH_1$ in the
$(\kappa_1,\kappa_2)$-plane shows that the two local curves $T_u$ and
$T_1$ are actually part of a single curve; it is shown in
\fref{fig:biftori}.  Along the two local branches one finds many
points of $p\!:\!q$ resonance where the Floquet multiplier is a
rational multiple of $2\pi$. They can be detected during the
continuation of the torus bifurcation curve, and \fref{fig:biftori}(a)
shows such resonances for $q \leq 13$. Emerging from each point of
$p\!:\!q$ resonance are two curves of fold or saddle-node of periodic
orbit bifurcations, which bound a resonance tongue or region where the
dynamics on the torus is $p\!:\!q$ locked.
In \fref{fig:biftori}(a) we find
that the pair of saddle-node of periodic
orbit bifurcation curves emerging from each $p\!:\!q$
resonance point on the upper branch $T_u$ can be continued
to a $p\!:\!(p+q)$ resonance point on the lower branch $T_1$.
The enlargement in panel (b) shows this for the specific example of the
$2\!:\!7$ resonance on $T_u$ and the $2\!:\!9$ resonance on $T_1$; the
further enlargement in \fref{fig:biftori}(c) shows the narrow tip of
the resonance tongue near the $2\!:\!7$ resonance point.

Such `connected resonance tongues' near a
Hopf-Hopf bifurcation point are a curious phenomenon that has not been
reported elsewhere to the best of our knowledge. Note that general theory (for ODEs
and DDEs with fixed delays) states that the existence of smooth
(normally hyperbolic) invariant tori --- with locked dynamics in resonance
tongues and quasi-periodic dynamics along curves in the
$(\kappa_1,\kappa_2)$-plane --- is guaranteed only
locally near the curves $T_u$ and $T_1$. Since, a $p\!:\!q$ torus
knot is topologically different from a $p\!:\!(p+q)$ torus knot,
the respective locked solutions near $T_u$ and $T_1$
cannot lie on one and the same smooth invariant torus. Nevertheless, a
locked solution on a torus is simply a periodic orbit, and it may
continue to exist even when
the underlying torus disappears. When no longer constrained
to lie on an invariant torus, a $p\!:\!q$ periodic orbit can be
transformed smoothly into a $p\!:\!(p+q)$ periodic orbit, which
explains why the saddle-node of periodic
orbit bifurcation curves may connect the respective points on $T_u$
and $T_1$. It is important to realise, however, that the regions that
the pair of curves bound cannot contain smooth invariant tori
throughout;
some examples of non-smooth tori will be presented in
\sref{sec:breakup}. The questions of how the smooth tori near $T_u$
and $T_1$ break up and how
the overall phenomenon is organised by the Hopf-Hopf bifurcation
certainly merit further study, ideally in the setting of a
four-dimensional ODE.

%%%%%%%%%%%%%%%%%%%%%%%%%%%%%%%%%%%%%%%%%%%%%%%%%%%
\begin{figure}[t!]
\begin{center}
\includegraphics{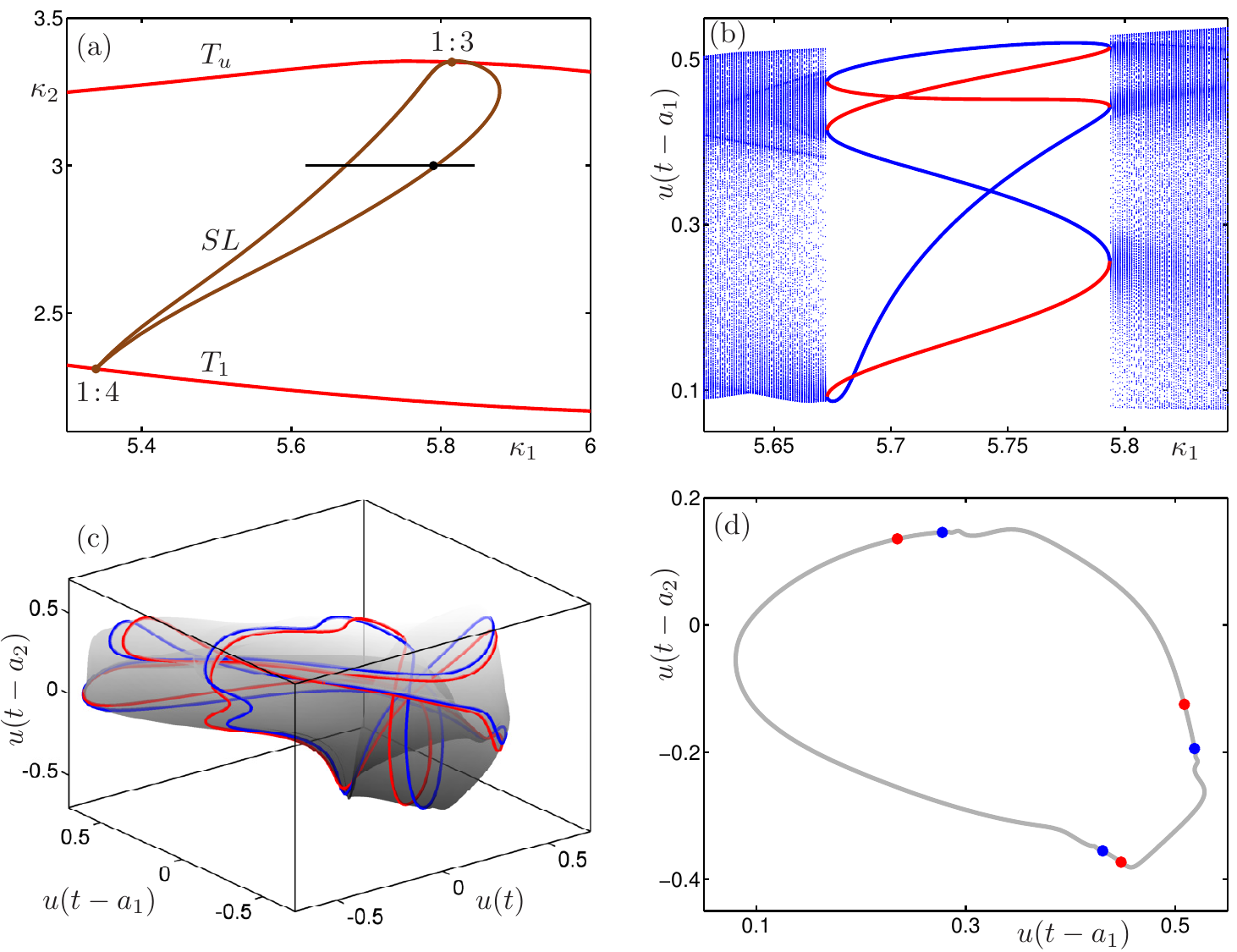}
\caption{The resonance tongue that connects a $1\! :\!3$ resonance on
  $T_u$ with a $1\!:\!4$ resonance on $T_1$. Panel (a) shows this
  resonance tongue in the $(\kappa_1,\kappa_2)$-plane. Panel (b) is a
  one-parameter bifurcation diagram in $\kappa_1$ for fixed
  $\kappa_2=3.0$, showing the values of $u(t-a_1)$ of the Poincar{\'e}
  trace of the stable periodic orbit (blue) and of the saddle periodic
  orbit (red) inside the resonance tongue, and of other solutions on
  tori outside the resonance tongue.  Panel (c) shows the $1\!:\!3$
  phase-locked torus (grey) for $\kappa_1=5.79$ with the stable and
  saddle periodic orbits in projection onto
  $(u(t),u(t-a_1),u(t-a_2))$-space, and panel (d) is its Poincar\'e
  trace in the $(u(t-a_1),u(t-a_2))$-plane.  The accompanying
  animation \texttt{chk\_anim9.avi} shows the corresponding evolution
  of the Poincar\'e trace over the $\kappa_1$-range in panel (b).}
\label{fig:1to3}
\end{center}
\end{figure}
%%%%%%%%%%%%%%%%%%%%%%%%%%%%%%%%%%%%%%%%%%%%%%%%%%%

Near the points of resonances on $T_u$ and $T_1$ the respective locked
dynamics must be expected to take place on a smooth invariant torus;
indeed \fref{fig:toruslocked} is an example of such a smooth torus
with locked dynamics. \fref{fig:1to3}(a) shows an enlargement of the
resonance tongue that connects a $1\! :\!3$ resonance on $T_u$ with a
$1\!:\!4$ resonance on $T_1$, and panel (b) shows the continuation of
the corresponding locked periodic orbits for $\kappa_2=3$. There are
three branches of stable and three branches of unstable periodic
solution in \fref{fig:1to3}(b), which meet at saddle-node bifurcations
marking the left and right boundaries of this region of locking. Tori
beyond the resonance region in panel (b) feature dynamics that is
quasi-periodic or of very high period; they were found by parameter
sweeping with numerical integration.  \fref{fig:1to3}(c) shows the
invariant torus for $\kappa_1=5.79$ (near the right boundary of the
locking region) as a surface in $(u(t),u(t-a_1),u(t-a_2))$-space, and
panel (d) is its trace for the Poincar\'e map defined by $u(t)=0$. The
torus was again found by computing the one-dimensional unstable
manifolds of the saddle periodic orbits. As \fref{fig:1to3}(c) and (d)
indicate clearly, this invariant torus is $1\! :\!3$ locked and
smooth. The animation \texttt{chk\_anim9.avi} in the supplemental materials
shows the evolution of the Poincar\'e trace as the parameter
$\kappa_1$ is swept over the range shown in \fref{fig:1to3}(b).

%%%%%%%%%%%%%%%%%%%%%%%%%%%%%%%%%%%%%%%%%%%%%%%%%%%
\begin{figure}[t!]
\begin{center}
\includegraphics{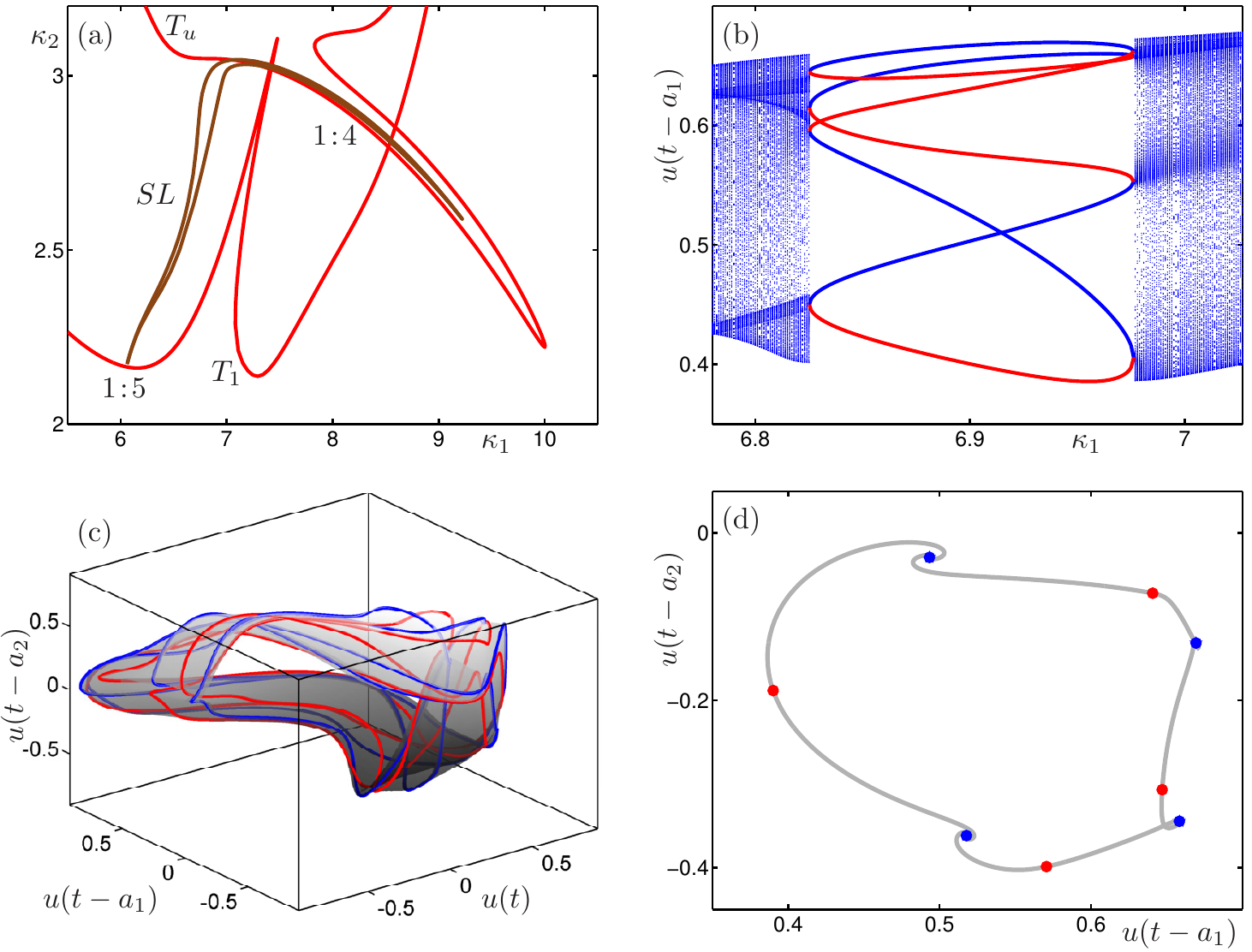}
\caption{The resonance tongue that connects a $1\! :\!4$ resonance on
  $T_u$ with a $1\!:\!5$ resonance on $T_1$. Panel (a) shows this
  resonance tongue in the $(\kappa_1,\kappa_2)$-plane. Panel (b) is a
  one-parameter bifurcation diagram in $\kappa_1$ for fixed
  $\kappa_2=3.0$, showing the values of $u(t-a_1)$ of the Poincar{\'e}
  trace of the stable periodic orbit (blue) and of the saddle periodic
  orbit (red) inside the resonance tongue, and of other solutions on
  tori outside the resonance tongue.  Panel (c) shows the $1\!:\!4$
  phase-locked torus-like object (grey) for $\kappa_1=6.93$ with the
  stable and saddle periodic orbits in projection onto
  $(u(t),u(t-a_1),u(t-a_2))$-space, and panel (d) is its Poincar\'e
  trace in the $(u(t-a_1),u(t-a_2))$-plane.  The accompanying
  animation \texttt{chk\_anim10.avi} shows the corresponding evolution
  of the Poincar\'e trace over the $\kappa_1$-range in panel (b).}
\label{fig:1to4}
\end{center}
\end{figure}
%%%%%%%%%%%%%%%%%%%%%%%%%%%%%%%%%%%%%%%%%%%%%%%%%%%

On the other hand, the saddle-node of periodic orbit bifurcation
curves in \fref{fig:biftori}(a) connect a $p\!:\!q$ resonance point on
$T_u$ to a $p\!:\!(p+q)$ resonance point on $T_1$. Hence, the torus
inside the respective resonance tongue cannot be smooth throughout,
because the knot type on a smooth invariant two-torus is an
invariant. While a $p\!:\!q$ periodic orbit can change smoothly into a
$p\!:\!(p+q)$ periodic orbit ---as \fref{fig:biftori} shows --- this
cannot happen on one and the same smooth two-torus.

\fref{fig:1to4}(a) shows an enlargement of the resonance tongue that
connects a $1\! :\!4$ resonance on $T_u$ with a $1\!:\!5$ resonance on
$T_1$. The one-parameter bifurcation diagram for $\kappa_2=3.0$ in
\fref{fig:1to4}(b) shows that one is dealing with $1\! :\!4$ locking:
there are four branches each of stable and unstable periodic orbits,
which meet in saddle-node bifurcations at the boundary of the
resonance tongue; the dynamics beyond the tongue is again
quasi-periodic or of very high period. The situation looks exactly as
that near the $1\!  :\!3$ resonance point in
\fref{fig:1to3}(b). However, as \fref{fig:1to4}(c) and (d) show, there
is no longer a smooth invariant torus. Rather, the one-dimensional
unstable manifold of the saddle periodic orbit spirals around the
stable periodic orbit; see panel (d). This means that the stable
periodic orbit has developed a pair of complex conjugate leading
Floquet multipliers, which is one mechanism for the loss of normal
hyperbolicity of an invariant torus that is known from ODE theory
\cite{mcgehee}. Note that the loss of normal hyperbolicity is found
numerically by two independent computations. The manifold seen to
spiral in panel (d) was computed by using the initial value problem
solver {\tt ddesd} and its event detection to compute a trajectory in
the unstable manifold of the periodic orbit and its intersections with
the Poincar\'e section, revealing the spiralling dynamics.  But we
also used DDE-BIFTOOL to directly compute the Floquet multipliers of
the unstable periodic orbit, confirming that the two dominant
multipliers are complex conjugate. The loss of normal hyperbolicity is
very clearly seen in the animation \texttt{chk\_anim10.avi} in the
supplemental materials, which shows the evolution of the Poincar\'e
trace in a one-parameter $\kappa_1$-sweep across the resonance
tongue. Namely, stable periodic points on the Poincar\'e trace are
denoted by stars in the animation when their dominant Floquet
multipliers are complex conjugate; this happens across much of this
traverse of the resonance tongue, and the unstable manifold of the
saddle periodic orbit is then seen to spiral into the stable periodic
points on the Poincar\'e trace.

%%%%%%%%%%%%%%%%%%%%%%%%%%%%%%%%%%%%%%%%%%%%%%%%%%%
\subsection{Break-up of $1\!:\!4$ locked torus}
\label{sec:breakup}

%%%%%%%%%%%%%%%%%%%%%%%%%%%%%%%%%%%%%%%%%%%%%%%%%%%
\begin{figure}[t!]
\begin{center}
\includegraphics{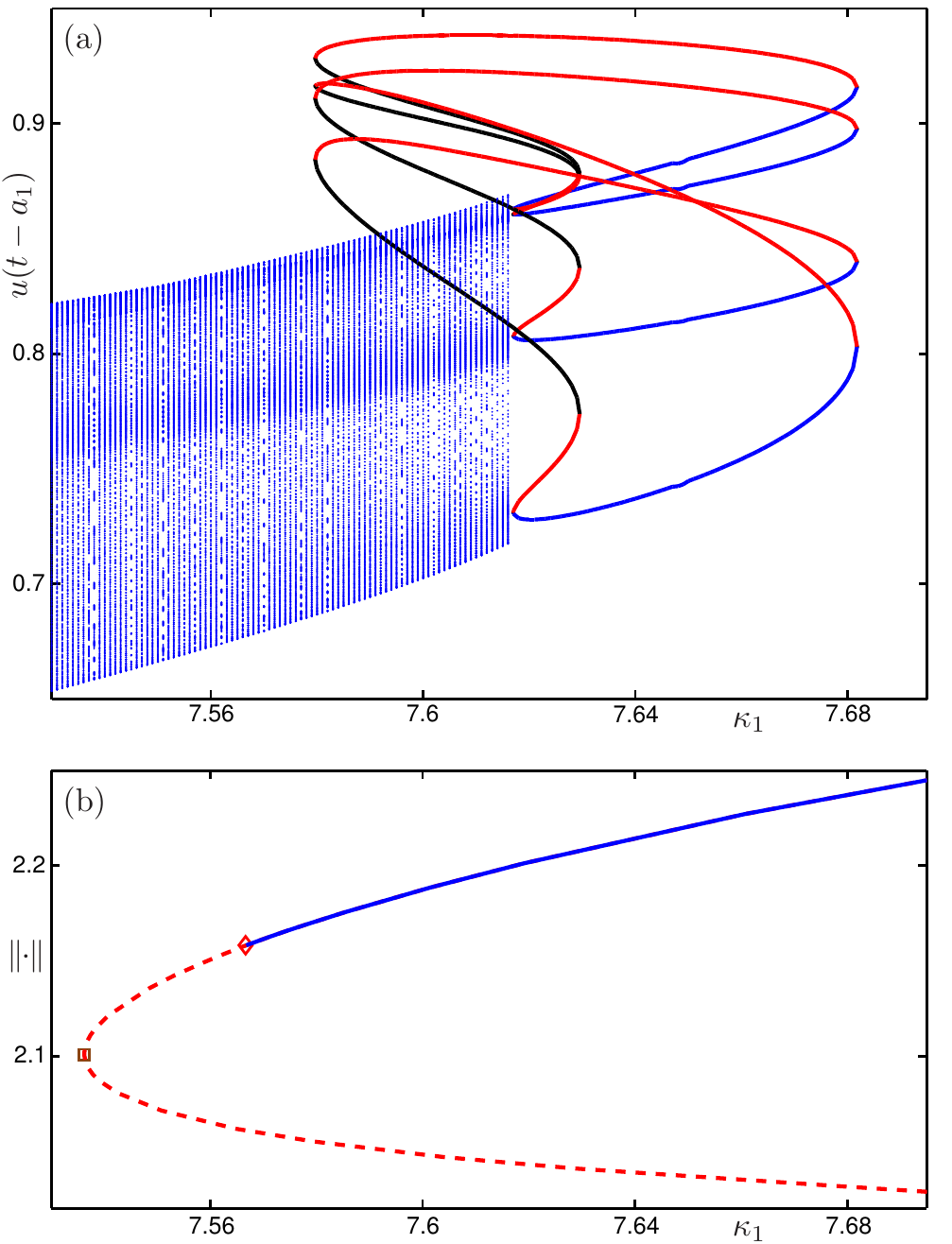}
\caption{One-parameter bifurcation diagrams relevant for the
transition throught the 1:4 resonance. Panel (a) shows the values of
$u(t-a_1)$ of the Poincar{\'e} trace of solutions on tori outside the
resonance tongue and of period-four periodic orbits that are stable
(blue), have one unstable Floquet multiplier (red), or have two
unstable Floquet multipliers (black).  Panel (b) shows the
simultaneously existing pair of principal periodic orbits that are
born in a saddle-node bifurcation, and one of which is stable (blue)
past the torus bifurcation (diamond).}
\label{fig:1to4breakup}
\end{center}
\end{figure}
%%%%%%%%%%%%%%%%%%%%%%%%%%%%%%%%%%%%%%%%%%%%%%%%%%%

%%%%%%%%%%%%%%%%%%%%%%%%%%%%%%%%%%%%%%%%%%%%%%%%%%%
\begin{figure}[t!]
\begin{center}
\includegraphics{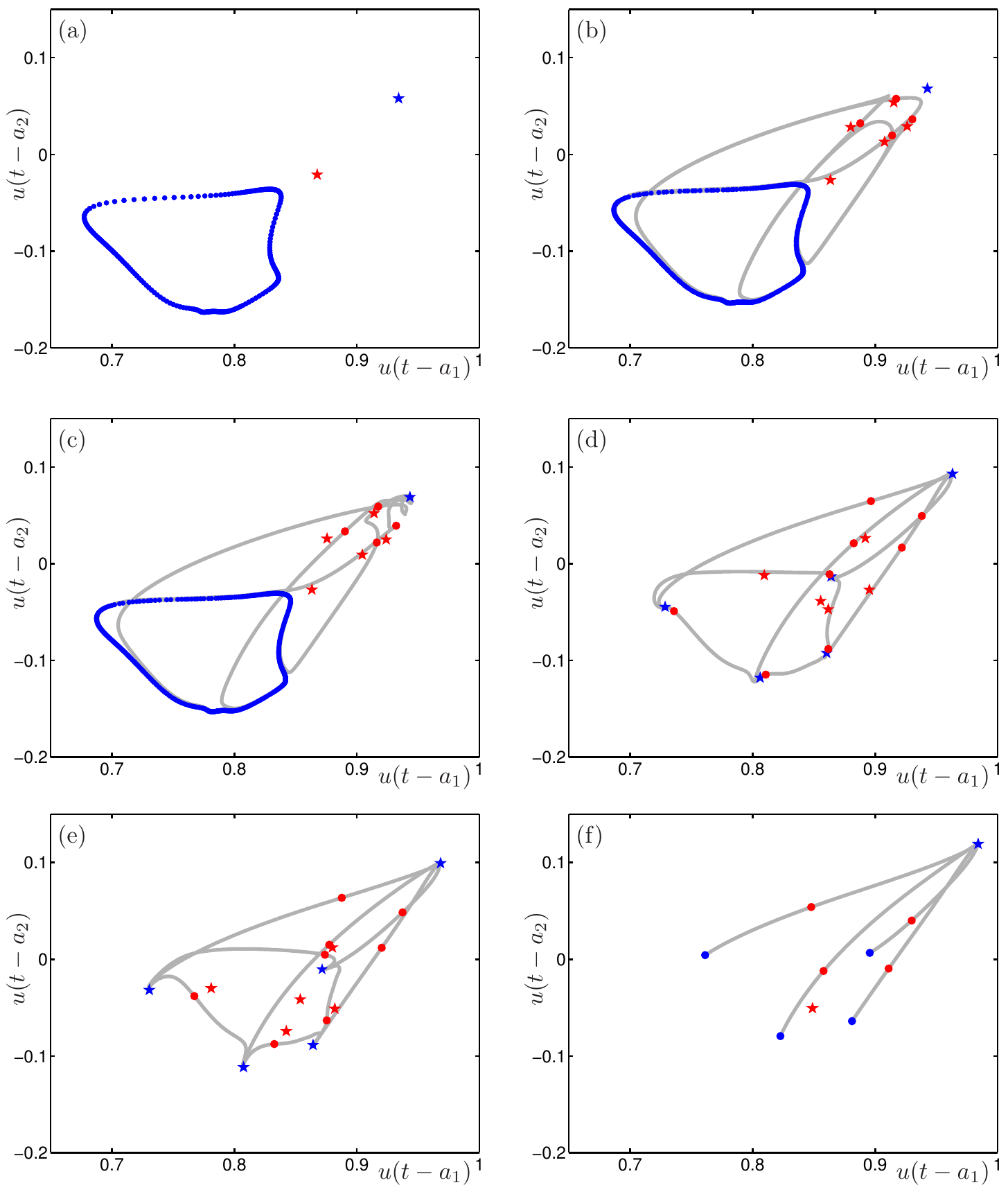}
\vspace{-1ex}
\caption{Sequence of Poincar\'e traces in the
$(u(t-a_1),u(t-a_2))$-plane showing the break-up of a torus with
$1\!:\!4$ phase locking. Shown are invariant curves (bold blue dots),
stable periodic points (blue stars) and saddle periodic points with
two unstable Floquet multipliers (red stars) and with a single
unstable Floquet multiplier (red dots); also shown are the traces of
the unstable manifolds (grey curves) of the latter saddle points. Here
$\kappa_2 = 3$ and in panels (a)--(f) $\kappa_1$ takes the values
$7.567$, $7.58$, $7.581$, $7.618$, $7.629$, and $7.666$,
respectively. See also the accompanying animation \texttt{chk\_anim12.avi}.
}
\label{fig:1to4breakupPS}
\end{center}
\end{figure}
%%%%%%%%%%%%%%%%%%%%%%%%%%%%%%%%%%%%%%%%%%%%%%%%%%%

In the previous section we discussed the local transition for fixed
$\kappa_2=3$ through a $1\!:\!4$ resonance as $\kappa_1$ changes near
$\kappa_1=6.93$. Notice in \fref{fig:1to4}(a) that the associated
resonance tongue in the in $(\kappa_1,\kappa_2)$-plane has the shape
of a horseshoe with maxima of the two bounding saddle-node curves at
$\kappa_1 \approx 7$. Both of the two maxima occur for
$\kappa_2>3$. Hence, for $\kappa_2=3$ there is a range of
$\kappa_1$-values outside this resonance tongue before it is entered
again at $\kappa_1\approx7.617$ when $\kappa_1$ is increased further
beyond the range shown in \fref{fig:1to4}(b). As we will show now, the
transition through this second part of the $1\!:\!4$ resonance tongue
results in the break-up and disappearance of the torus via a
complicated scenario of bifurcations that involves nearby periodic
orbits.

The sequence of bifurcations for fixed $\kappa_2=3$ and the associated
dynamics are illustrated by two companion
figures. \Fref{fig:1to4breakup} shows two one-parameter bifurcation
diagrams in $\kappa_1$, and \fref{fig:1to4breakupPS} shows the
associated sequence of Poincar{\'e} traces in the
$u(t-a_1),u(t-a_2))$-plane; see also the accompanying animation
\texttt{chk\_anim12.avi}, which animates the evolution of the
Poincar\'e traces for $\kappa_1\in[7.530,7.702]$.

Starting at $\kappa_1 = 7.5$, there is an invariant torus with
quasiperiodic or high-period solutions on it; see
\fref{fig:1to4breakup}(a).  As $\kappa_1$ is increased, the first
bifurcation of interest is the creation of two saddle periodic
orbits at a saddle-node bifurcation of periodic orbits at $\kappa_1
\approx 7.5363$. We refer to them as the principal periodic
orbits because their branch can actually be traced back to first Hopf
bifurcation $H_1$; see \fref{fig:1Dbif}. As is shown in
\fref{fig:1to4breakup}(b), at $\kappa_1 \approx 7.5664$ one of the
two saddle periodic orbits gains stability in a
torus bifurcation when the branch of periodic orbits crosses the torus
curve $T_u$. This torus bifurcation is close to $1\!:\!4$ resonance,
with numerically computed Floquet multipliers
$\rho\approx-0.019\pm1.000073i$ very close to $\pm i$.  There is then
an interval of $\kappa_1$-values for which the stable periodic orbit
on the principal branch and the stable quasi-periodic torus co-exist;
see \fref{fig:1to4breakup}(a). The associated invariant closed curve in the
$u(t-a_1),u(t-a_2))$-plane is shown in \fref{fig:1to4breakupPS}(a),
together with the two points that represent the stable and saddle principal
periodic orbits in the Poincar{\'e} trace.

At $\kappa_1\approx7.5796$ another saddle-node bifurcation of periodic
orbits creates a pair of period-four orbits, one of which has exactly
one and the other two unstable Floquet multipliers; see
\fref{fig:1to4breakup}(a).  In the Poincar{\'e} trace in
\fref{fig:1to4breakupPS}(b), for $\kappa_1=7.58$, these are represented
by two sets of period-four points. Also shown is the one-dimensional
trace of the unstable manifold of the saddle periodic orbit with one
unstable Floquet multiplier; note that both its branches (on either
side of the respective periodic point) converge to the attracting
invariant curve.  Almost immediately afterwards, for
$7.58<\kappa_1<7.581$, there is a bifurcation that changes the nature
of the unstable manifold of the saddle period-four orbit. As
\fref{fig:1to4breakupPS}(c) shows, one branch now goes to the
attracting principle periodic orbit (blue star), while the other
branch still goes to the attracting invariant curve. This means that,
on the level of the Poincar{\'e} trace, we are dealing with a global
bifurcation that is described in the approximating normal form of a
$1\!:\!4$ resonance as a saddle connection of square type
\cite{124nonlin,124expmath}.

At $\kappa_1 \approx 7.617$ the $1\!:\!4$ resonance tongue is
re-entered and we find two locked period-four orbits on the torus, one
of which is attracting and the other has a single unstable Floquet
multiplier. In the trace in \fref{fig:1to4breakupPS}(d) they are shown
as a further two sets of period-four points. Also shown is the trace
of the unstable manifold of the saddle four-periodic orbit on the
torus, both branches of which end up at neighboring period-four
attracting points to form a smooth invariant curve. Hence, the torus
is still normally hyperbolic (that is, smooth) as is expected near the
boundary of a resonance tongue. Notice that the respective branch of
the unstable manifold of each saddle period-four point off the invariant
curve now also goes to the attracting periodic orbit on the torus.

As $\kappa_1$ is increased further, the torus loses normal
hyperbolicity. More specifically, the branches of all unstable
manifolds shown in \fref{fig:1to4breakupPS}(e) approach the attracting
period-four orbit along the same side of its weak stable
eigen-direction, so that a cusp is formed along the attracting
period-four orbit. Moreover, the period-four orbit with two unstable
Floquet multipliers, created at $\kappa_1\approx7.5796$ and
not mentioned since, now enters the action.  As $\kappa_1$ increases,
this saddle periodic orbit approaches the saddle periodic orbit on the
torus, which has a single unstable Floquet multiplier. At $\kappa_1
\approx7.6295$, the two period-four orbits annihilate each other in a
further saddle-node bifurcation; see \fref{fig:1to4breakup}(a). The
periodic points and the associated unstable manifold disappears
at this value of $\kappa_1$. Hence, as \fref{fig:1to4breakupPS}(f)
illustrates, we are left with the two remaining period-four orbits: the
attracting one and other saddle periodic orbit. Notice that the
unstable manifold of the latter does not change in this process,
meaning that one branch of each period-four point in the trace still
ends up at the principal periodic orbit, and the other at the
respective attracting period-four point.  As $\kappa_1$ is increased
even further, the two period-four orbits approach each other and
finally disappear in the last saddle-node bifurcation at
$\kappa_1\approx7.6818$ in \fref{fig:1to4breakup}(a). Hence,
we are left with only the stable and saddle principal periodic orbits;
see \fref{fig:1to4breakup}(b).

Overall, the torus loses normal hyperbolicity and then breaks up and
disappears completely. In particular, unlike for the cases studied in
\sref{sec:resonance}, the torus does not re-emerge on the other side
of the $1\!:\!4$ resonance tongue.

%%%%%%%%%%%%%%%%%%%%%%%%%%%%%%%%%%%%%%%%%%%%%%%%%%%
\section{Overall bifurcation diagram and conclusions}
\label{sec:overall}

Our study of the scalar state-dependent DDE \eqref{eq:twostatedep}
concentrated on the dynamics associated with the presence of
codimension-two Hopf-Hopf bifurcation points. We presented a normal
form procedure for state-dependent DDEs that, by eliminating the state
dependence up to order three, allowed us to determine the type of
Hopf-Hopf bifurcation from the resulting DDE with nine constant
delays. In this way, we showed that a pair of torus bifurcation curves
emerges locally from each of the three Hopf-Hopf bifurcation points in
the region of interest of the $(\kappa_1,\kappa_2)$-plane of
\eqref{eq:twostatedep}. Our normal form computations have been
confirmed by finding and continuing these torus bifurcation curves
numerically with the package DDE-BIFTOOL. What is more, numerical
continuation allowed us to follow the torus bifurcation curves beyond
the local neighborhoods of the Hopf-Hopf bifurcation points, and to
identify the structure of resonance tongues emerging from them.
We computed locked periodic orbits on the tori and
determined the boundaries of resonance tongues by continuing their
saddle-node bifurcations. The tori and the dynamics on them was
investigated and visualised by suitable projections into
three-dimensional space, as well as by their two-dimensional
Poincar{\'e} traces. In particular, we computed the unstable manifolds
of saddle-periodic orbits with a single unstable Floquet multiplier,
which allowed us to study in considerable detail how invariant tori
break up and disappear, for example near a $1\!:\!4$ resonance.

%%%%%%%%%%%%%%%%%%%%%%%%%%%%%%%%%%%%%%%%%%%%%%%%%%%
\begin{figure}[t!]
\begin{center}
\includegraphics{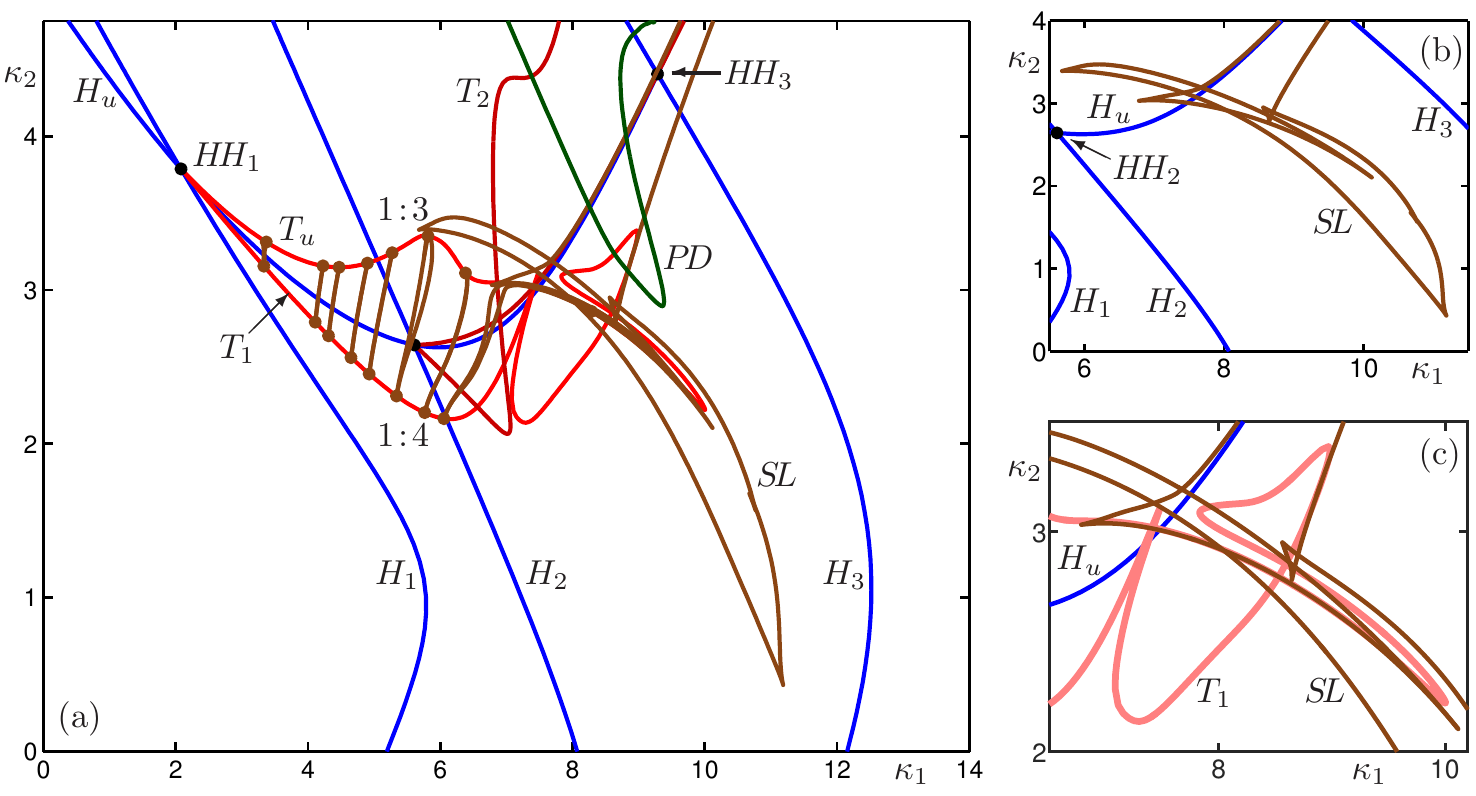}
\caption{Overall bifurcation diagram of \eqref{eq:twostatedep} in the
$(\kappa_1,\kappa_2)$-plane (a), showing curves of Hopf bifurcation
(blue) of torus bifurcation (red), of saddle-node of limit cycle
bifurcation (brown) and of period-doubling bifurcation (green). Panel
(b) is an enlargement near $\HH_2$, and panel (c) shows details of
the saddle-node of limit cycle bifurcation curve \SL that is not
connected to a resonance point on a torus.}
\label{fig:bifdiag}
\end{center}
\end{figure}
%%%%%%%%%%%%%%%%%%%%%%%%%%%%%%%%%%%%%%%%%%%%%%%%%%%

The starting point of our investigation was the one-parameter
bifurcation diagram \fref{fig:1Dbif} from \cite{Hum-Dem-Mag-Uph-12}.
Specifically, we used it to start continuations of periodic solutions
and of bifurcation curves in the $(\kappa_1,\kappa_2)$-plane, namely,
the curves of Hopf bifurcation in \fref{fig:hopfcurves}, as well as
the curves of torus bifurcation and saddle-node bifurcation that bound
certain resonance tongues in \fref{fig:biftori}(a). Returning to
\fref{fig:1Dbif}, one can identify two additional bifurcations that we
did not consider yet in our study of resonance phenomena: a
period-doubling bifurcation and an additional saddle-node bifurcation
of limit cycles.  \Fref{fig:bifdiag} shows the overall two-parameter
bifurcation diagram of \eq{eq:twostatedep} in the
$(\kappa_1,\kappa_2)$-plane with all the above bifurcation
curves. Panel (a) shows the relevant region where $0 \leq \kappa_1
\leq 14$ and $0 \leq \kappa_1 \leq 4.75$. In particular, shown are the
three pairs of torus bifurcation curves emerging from the Hopf-Hopf
bifurcation points $\HH_1$ to $\HH_3$. Notice that the two torus
bifurcation curves emerging from $\HH_3$ stay very close to the Hopf
bifurcation curve $H_u$; similarly, the torus bifurcation curve $T_u$
emerging from $\HH_2$ stays close to $H_u$, while the other curve
$T_2$ exits the top of the $(\kappa_1,\kappa_2)$-plane). Prominent in
panel (a) is the curve \PD of period-doubling bifurcation, which has a
minimum near $(\kappa_1,\kappa_2) \approx (10, 3)$. As
\fref{fig:1Dbif}(a) shows, the periodic orbit undergoing the
period-doubling bifurcation has a large amplitude.

The other new curve in \fref{fig:bifdiag} is the saddle-node of limit
cycle bifurcation curve labelled $\SL$. It enters and exits the top of
the $(\kappa_1,\kappa_2)$-plane near and in the direction of the Hopf
bifurcation curve $H_u$. As panel (b) shows, the curve $\SL$ is very
complicated and features eight cusps (two pairs of which are actually
very close to swallowtail bifurcations), resulting in quite a number
of regions with different numbers of bifurcating periodic orbits.
From \sref{sec:breakup} it is clear that some periodic orbits emerging
from saddle node bifurcations play an important role in the torus
break-up mechanism.  At the same time, the overall bifurcation diagram
in \fref{fig:bifdiag} shows with the example of $\SL$ that there are
other saddle node bifurcations that may not immediately be related to
the torus bifurcations emerging from $\HH_1$ to $\HH_3$. However,
$\SL$ comes very close to several torus bifurcation curves; see
\fref{fig:bifdiag}(c). Moreover, it follows closely the
horseshoe-shaped resonance region (discussed in \sref{sec:breakup})
that connects the $1\!:\!4$ resonance on $T_u$ with the $1\!:\!5$
resonance in $T_1$. We remark that the curve $\SL$ traverses the
$(\kappa_1,\kappa_2)$-plane several times close to the line
$\kappa_1+\kappa_2 = \gamma(a_2/a_1-2)\approx12.4$ where the singular
fold bifurcation $L_{00}$ occurs in the $\epsilon\to0$ singular limit
of \eq{eq:twostatedep}; see \cite{HBCHM:1}.  Moreover, \SL extends to
very low values of $\kappa_2$; in fact, in one-parameter bifurcation
diagrams in $\kappa_1$ for fixed $\kappa_2$, it generates the first
observed folds in the branch of periodic orbits that bifurcate from
the Hopf bifurcation $H_1$ as $\kappa_2$ is increased; see
\cite{Hum-Dem-Mag-Uph-12}.

\Fref{fig:bifdiag} can be seen as a summary and overview of the level
of complexity of the dynamics one can find in
\eqref{eq:twostatedep}. In a sense, the overall bifurcation diagram in
the $(\kappa_1,\kappa_2)$-plane of the two feedback strengths would not
be particularly unusual for a nonlinear DDE. Its surprising aspect is,
however, that all phenomena it represents are entirely due to the
state dependence. As the state-dependence parameters $c_1$ and $c_2$
of the delays are decreased to zero, the bifurcation structure in
\fref{fig:bifdiag}, including the Hopf-Hopf bifurcation points and
induced dynamics on tori, will disappear. Indeed,
\eqref{eq:twostatedep} for $c_1 = c_2 = 0$ is entirely linear and,
hence, does not have any nontrivial dynamics.  Hence, if one were to
replace the state-dependence by constant delays, none of the dynamics
we reported would be found.  Admittedly, equation
\eqref{eq:twostatedep} has been constructed as an extreme case in this
regard. Nevertheless, the study presented here should be seen as a
health warning: replacing state dependence by a constant-delay
approximation may result in the disappearance of the very dynamics one
intends to study. This may be the case even when the approximating
constant-delay DDE is actually nonlinear itself.

State-dependent DDEs have been suggested as suitable models in a
number of applications \cite{DGHP10,IST07,Jes-Cam-10,pyragas2,yb}.  We
hope that the study presented here may serve as a demonstration of
what can be achieved by a combination of analytical and numerical
tools when it comes to the bifurcation analysis of a given
state-dependent DDE. It is now possible to study models from this
class effectively in their own right, and to determine the role the
state dependence plays in the observed dynamics. In fact, normal form
calculations and numerical continuation tools are able to produce
consistent results, such as the type of codimension-two bifurcation or
the existence and organisation of resonances on tori, for which, as
yet, the respective theory has not yet been developed for
state-dependent DDEs.  We believe that case studies of specific
systems are also a useful way of guiding the further development of
theory for state-dependent DDEs.  At the same time, numerical methods
also continue to be developed further. For example, the curves shown
in \fref{fig:bifdiag} were computed with recently implemented routines
of DDE-BIFTOOL \cite{NewDDEBiftool} that allow the continuation in two
parameters of codimension-one bifurcation of periodic orbits to
determine curves of saddle-node, period-doubling and torus
bifurcations. Previously, such curves could only be constructed by
detecting the respective bifurcation in one-parameter continuations,
which is certainly not a suitable method for finding complicated
bifurcation curves such as $\SL$ in \fref{fig:bifdiag}(b). In a
nutshell, practically all advanced tools for the bifurcation analysis
of DDEs are now also available when state dependence is present.

%%%%%%%%%%%%%%%%%%%%%%%%%%%%%%%%%%%%%%%%%%%%%%%%%%%%%%%%%%%%%%%%%%%%%%%%%%%%%%%%%%%%%%%%%%%%%%%%%%%%%%%%%%%%%%%%%%%%%%

\section*{Acknowledgements}
A.R.H.~is grateful to the National Science and Engineering Research Council (NSERC),
Canada for funding through the Discovery Grant program, and thanks the University
of Auckland for its hospitality and support during two research visits. R.C.C.~thanks
the Department of Mathematics and Statistics at McGill for their hospitality during
his time as a Postdoctoral Fellow and now as an Adjunct Member of the department.
He is also grateful to NSERC and the Centre de Recherches Math\'ematiques for
funding and to the FQRNT for a PBEEE award.
We thank Jan Sieber for fruitful discussions regarding normal form calculation
within DDE-BIFTOOL, Rafael de la Llave and Xiaolong He for helpful comments
on quasiperiodic solutions in state-dependent DDEs, and two anonymous referees for their
very constructive feedback on the initial version of the manuscript.

%%%%%%%%%%%%%%%%%%%%%%%%%%%%%%%%%%%%%%%%%%%%%%%%%%%%%%%%%%%%%%%%%%%%%%%%%%%%%%%%%%%%%%%%%%%%%%%%%%%%%%%%%%%%%%%%%%%%%%

%%%********************************************************

%%%%%%%%%%%%%%%%%%%%%%%%%%%%%%%%%%%%%%%%%%%%%%%%
%\bibliographystyle{abbrv}
%\bibliographystyle{alpha}
%\bibliography{SDDEDH}

%%%%%%%%%%%%%%%%%%%%%%%%%%%%%%%%%%%%%%%%%%%%%%%%

\newpage

%%%%%%%%%%%%%%%%%%%%%%%%%%%%%%%%%%%%%%%%%%%%%%%%
\begin{appendices}

\section{Computation of the Hopf-Hopf Normal Form}
\label{app:NormalForm}

Here we describe in detail the derivation of the normal form of the
Hopf-Hopf bifurcation for the truncated constant-delay DDE
\eqref{eq:dde3} from \sref{sec:expansion}, where we follow the
derivation of Wu and Guo \cite{Wu-Guo-13}. The computational task is
to derive the restriction of the semi-flow of \eqref{eq:dde3} to the
four-dimensional center manifold up to third order, which is an ODE
from which the type of Hopf-Hopf bifurcation can be determined
\cite{Kuz-04-Book}. We elaborate these steps as follows.  In
\sref{secapp:centreman} we construct a projection to the center
manifold for the constant-delay DDE \eqref{eq:dde3}, and in
\sref{sec:dhopf} we study the flow on the center manifold near the
Hopf-Hopf bifurcation. We then compute the quadratic and cubic terms
of this flow in \sref{sec:code}, which enables us to determine the
normal form and type of Hopf-Hopf bifurcation in
\sref{secapp:NormalForm}.

%%%%%%%%%%%%%%%%%%%%%%%%%%%%%%%%%%%%%%%%%%%%%%%%%%%
\subsection{Center manifold}
\label{secapp:centreman}

To construct the center manifold for the constant delay DDE
\eq{eq:dde3} we write it as an RFDE in the form \eq{eq:rfde}, that is,
as a sum of linear and nonlinear operators as \be \label{eq:rfdelnl}
u'(t)=\cL u_t+\cF(u_t).  \ee It follows from \eq{eq:F} and
\eq{eq:lineareq} that
\begin{align} \label{eq:lop}
\cL u_t & =-\gamma u_t(0) -\kappa_1
u_t(-a_1)-\kappa_2 u_t(-a_2),\\ \cF(u_t)& =F(u_t)-\cL
u_t, \label{eq:nop}
\end{align}
while, from \sref{sec:expansion}, the nonlinear operator
is given by
\begin{align}
\notag \cF(u_t) & = \sum_{i=1}^2\kappa_i(c u_t(0))L
u_t(-a_i) + \sum_{i,j=1}^2\kappa_i\kappa_jc^2u_t(0)u_t(-a_i)L
u_t(-a_i-a_j) \\ & \qquad -\frac12(c u_t(0))^2\sum_{i=1}^2\kappa_i
L^2u_t(-a_i).  \label{fut}
\end{align}
Here, the difference operator $L$ defined in
\eq{eq:Llinearop}, has been applied to $u_t$ in the natural way, so
\begin{align}
\notag Lu_t(\theta)&=-\gamma u_t(\theta)- \kappa_1
u_t(\theta-a_1) - \kappa_2 u_t(\theta - a_2)\\ &=-\gamma u(t+\theta)-
\kappa_1 u(t+\theta-a_1) - \kappa_2 u(t+\theta- a_2). \label{LCop}
\end{align}

We start by introducing the appropriate spaces and operators that we
will need to perform the reduction to a four-dimensional center
manifold at a Hopf-Hopf bifurcation point. Throughout this section we
will follow the notation used in Wu and Guo \cite{Wu-Guo-13} and adapt
the corresponding theory to study \eqref{eq:rfdelnl} near Hopf-Hopf
bifurcations.

\sloppy{As noted in the introduction, it is standard to treat the RFDE
\eq{eq:rfde} as an infinite-dimensional dynamical system in the Banach}
space of continuous functions of an interval into $\R^d$. For the
scalar DDE \eq{eq:rfdelnl} we have $d=1$, and we equip $\R$ with the
Euclidian norm, $|\cdot|$, and, for given $\tau>0$, we define
\begin{equation}\label{banachspace-1-tau}
C = C([-\tau, 0], \R),
\end{equation}
the Banach space of continuous mappings, equipped with
the supremum norm. For $\varphi\in C$, this norm is given by
\[
\| \varphi \| = \sup_{\theta\in[-\tau,0]} |\varphi(\theta)|.
\]
In an analogous manner, we define
\begin{equation}
C^1 = C^1([-\tau, 0], \R),
\end{equation}
the space of continuous differentiable mappings with
continuous derivative, which is also a Banach space with the
corresponding supremum norm
\[
\| \varphi \| = \sup_{\theta\in[-\tau,0]}(|\varphi(\theta)| +
|\tfrac{d}{d\theta}\varphi(\theta)|), \qquad\varphi\in C^1.
\]
With $u_t\in C$ defined by \eq{eq:ut} equation \eq{eq:rfdelnl}
defines an RFDE of the form \eq{eq:rfde} provided $\cL:C\to\R$ and
$\cF:C\to\R$.  The linear operator $\cL$ is defined in \eqref{eq:lop},
and it is a continuous operator from $C$ into $\R$ whenever $\tau\geq
a_2$ (recalling that $a_2>a_1$).  However, some care needs to be taken
with the operator $\cF$.  As noted in \sref{sec:expansion}, the
truncation to third order results in constant delays, the largest of
which is $\tau=3a_2$. This shows up in \eq{fut} where $u(t-3a_2)$
appears in both the terms $L u_t(-2a_2)$ and $L^2u_t(-a_2)$.  Hence,
we require $\tau\geq 3a_2$ for $\cF:C\to\R$ and for \eq{eq:rfdelnl} to
define an RFDE. This contrasts with the state-dependent DDE
\eq{eq:twostatedep}, which in \eq{eq:F} we defined as an RFDE with
$\tau=a_2+\tfrac{a_1}{\gamma}(\kappa_1+\kappa_2)$.

In the following consider \eq{eq:rfdelnl} as an RFDE with $C$ defined
by \eq{banachspace-1-tau} and $\tau=3a_2$.  The linearized system
associated to \eqref{eq:rfdelnl} is \be \label{lin-sys} u'(t) = \cL
u_t.  \ee Since the linear operator $\cL:C \to \R$, defined in
\eqref{eq:lop}, is continuous, then, as shown in \cite{Hal-Lun-93} by
the Riesz representation theorem, there exists a function
$\eta:[-\tau,0]\to\R$ of bounded variation such that
\[
\cL \varphi = \int_{-\tau}^{0}d\eta(\theta)\varphi(\theta), \quad
\forall \varphi \in C.
\]
The function $\eta$ satisfies that $\eta(\theta)=0$ for $\eta\in
(-\tau, -a_2) \cup (-a_2, -a_1) \cup (-a_1,0)$, $\eta(0) = - \gamma$,
$\eta(-a_1) = - \kappa_1$, $\eta(-a_2) = - \kappa_2$. Therefore,
\begin{equation} \label{lin-operator}
\cL \varphi = \int_{-3
a_2}^{0}d\eta(\theta)\varphi(\theta) = -\gamma \varphi(0) - \kappa_1
\varphi(-a_1) -\kappa_2 \varphi(-a_2).
\end{equation}
Let $T(t):C \to C$ be the solution operator of the
linear system \eqref{lin-sys}.  Then, as is shown in
\cite{Hal-Lun-93}, the infinitesimal generator $\cA$ of the semi-group
$T(t)$ is defined by
$$\cA \varphi = \lim_{t\to 0^+} \frac{T(t)\varphi - \varphi}{t}$$
for $\varphi \in C$, which results in
\begin{equation} \label{eq:dTAT}
\frac{d}{dt}T(t)\varphi = \cA T(t)\varphi
\end{equation}
and
\begin{equation} \label{inf-gen}
(\cA\varphi)(\theta) = \left\{
\begin{array}{ll} \frac{d}{d\theta}\varphi, & \textrm{if
}\theta\in[-\tau,0),\\
-\gamma\varphi(0)-\kappa_1\varphi(-a_1)-\kappa_2\varphi(-a_2), &
\textrm{if }\theta=0.
\end{array}\right.
\end{equation}
Here the domain of $\cA$ is given by
\[
\textrm{dom}(\cA) = \{\varphi:\varphi\in C^1,\, \varphi'(0) = \cL
\varphi\}.
\]

Following \cite{Wu-Guo-13}, we now enlarge the phase space $C$ so that
\eqref{eq:rfdelnl} can be written as an abstract ODE in a Banach
space. Let $BC$ be the set of functions from $[-\tau, 0]$ to $\R$ that
are uniformly continuous on $[-\tau, 0)$ and may have a jump
discontinuity at 0. We also introduce the function $X_0:[-\tau,0]\to
\R$ defined by
\[
X_0(\theta) =
\left\{
\begin{array}{ll}
1, & \theta=0,\\
0, & \theta\in [-\tau,0).
\end{array}\right.
\]
Then every $\varphi\in BC$ can be expressed as $\varphi = \phi +
X_0\xi$ with $\phi \in C$ and $\xi \in \R$, and thus $BC$ can be
identified with $C\times\R$. We equip $BC$ with the norm $|\phi + X_0
\xi| = \|\phi\| + |\xi|$, which is then also a Banach space.

The spectrum of the infinitesimal generator $\cA$ consists of the
eigenvalues $\lambda \in \sigma(\cA)$ that satisfy the characteristic
equation
\begin{equation}\label{characteristic-equation}
0 = \Delta(\lambda) =
\lambda - \int_{-\tau}^0 e^{\lambda \theta} d \eta(\theta) =\lambda +
\gamma + \kappa_1 e^{-a_1 \lambda} + \kappa_2 e^{-a_2 \lambda}.
\end{equation}
For any $\lambda \in \sigma(\cA)$, the generalized
eigenspace $\mathcal{M}_\lambda(\cA)$ is finite-dimensional and, since
in our case the eigenvalues will have multiplicity $1$, we write
$\cM_{\lambda_i}(\cA) = \ker(\lambda_i I - \cA)$ and we have the
decomposition
\[
C=\ker(\lambda I - \cA)\oplus\textrm{im}(\lambda I - \cA).
\]
If we have a set of distinct eigenvalues $\Lambda =
\{\lambda_1,...,\lambda_d\} \subset \sigma(\cA)$, we will use the
notation $\cM_\Lambda(\cA)$ for the generalized eigenspace
corresponding to those eigenvalues.  Let $d=\dim \cM_\Lambda(\cA)$,
and $\varphi_1,\ldots,\varphi_d$ be a basis for $\cM_\Lambda(\cA)$,
and $\Phi_\Lambda = (\varphi_1,\ldots,\varphi_d)$. Then there exists a
$d\times d$ constant matrix $B = B_\Lambda$ such that $\cA
\Phi_\Lambda = \Phi_\Lambda B$, and
\begin{itemize}
\item[i)] the only eigenvalues of $B$ are $\Lambda =
\{\lambda_1,\ldots,\lambda_d\}$,
\item[ii)] $\Phi_\Lambda (\theta) = \Phi(0) e^{B \theta}$,
\item[iii)] $T(t) \Phi_\Lambda = \Phi_\lambda e^{B t}$, where $T(t)$
satisfies \eq{eq:dTAT}.
\end{itemize}

We denote by $C^*$ the dual of $C$, so $C^* = C([0, \tau], \R^*) =
C([0, \tau], \R),$ the space of continuous functions from $[0, \tau]$
to $\R$ with norm given for a function $y \in C^{*}$ by,
\[
\|y\| = \sup_{t \in [0, \tau]}|y(t)|.
\]
We also introduce a bilinear form associated with $\cL$, for
$\varphi\in C$ and $\psi \in C^*$, as
\begin{equation}\label{bilinear}
\langle \psi, \varphi \rangle =
\overline{\psi}(0)\varphi(0) -
\kappa_1\int_{-a_1}^{0}\overline{\psi}(s+a_1)\varphi(s)ds -
\kappa_2\int_{-a_2}^{0}\overline{\psi}(s+a_2)\varphi(s)ds.
\end{equation}
Then we can find, at least formally, an adjoint
linearized problem,
\[
y'(t) = \gamma y(t) + \kappa_1 y(t+a_1) + \kappa_2 y(t+a_2)
\]
acting on functions $y_t \in C^{*}$, with the corresponding
solution operator $T^*:C^*\to C^*$.  We denote the infinitesimal
generator of the strongly continuous semi-group $T^*$ by $\cA^*$.  For
$\psi \in C^*$, $\cA^*\psi$ is defined by
\begin{equation} \label{adj-inf-gen}
(\cA^*\psi)(\xi) = \left\{
\begin{array}{ll} -\frac{d}{d\xi}\psi(\xi), & \textrm{if } \xi \in (0
,\tau],\\ \gamma \psi(0) + \kappa_1 \psi(-a_1) +\kappa_2\psi(-a_2),
&\textrm{if }\xi=0.
\end{array}\right.
\end{equation}
The operators $\cA$ and $\cA^*$ as defined by
\eqref{inf-gen} and \eqref{adj-inf-gen} are then adjoint with respect
to the bilinear form \eqref{bilinear}; that is
\[
\langle \psi, \cA \varphi\rangle = \langle \cA^* \psi, \varphi
\rangle, \qquad\varphi\in C, \; \psi\in C^*.
\]

%%%%%%%%%%%%%%%%%%%%%%%%%%%%%%%%%%%%%%%%%%%%%%%%%%%
\subsection{Hopf-Hopf bifurcation}
\label{sec:dhopf}

We will now use the properties of $\cA$ and $\cA^*$ to construct a
basis of eigenfunctions for the center space and the adjoint of the
center space at the Hopf-Hopf bifurcation. At a Hopf-Hopf bifurcation,
the infinitesimal generator $\cA$ defined by \eqref{inf-gen} has two
pairs of simple purely imaginary eigenvalues $\pm i \omega_1$ and $\pm
i \omega_2$ that do not have a strong resonance; that is, $k \omega_1
\neq \ell \omega_2$ where $k$ and $\ell$ are positive integers with
$k+\ell \leq 5$.

Following the discussion in \sref{secapp:centreman}, we know that the
generalized center eigenspace $E^c = \mathcal M_{\{\pm i\omega_1, \pm
i \omega_2\}}$ is a four-dimensional linear space. We also have two
complex conjugate eigenvectors $q_1, q_2 \in C$ such that
\[
\cA q_j = i \omega_j q_j, \quad \textrm{for } j=1,2,
\]
namely, $q_j(\theta) = e^{i \omega_j \theta}$, since clearly
$\frac{dq_j}{d\theta} = i \omega_j q_j$, while at the Hopf-Hopf point
we have
\[
\cA q_j(0)=-\gamma q_j(0) - \kappa_1 q_j(-a_1) -\kappa_2 q_j(-a_2) =
-\gamma  - \kappa_1 e^{-i \omega_j a_1} -\kappa_2 e^{-i \omega_j a_2} =
 i \omega_j q_j(0),
\]
as required to satisfy \eq{inf-gen}. We also introduce the adjoint
eigenvectors $p_1, p_2\in C^*$, such that
\[
\cA^* p_j = - i \omega_j p_j, \quad \textrm{for } j=1,2.
\]
Then, \[p_j(s) = D_j e^{i\omega_j s}.\] We choose the constants $D_j =
1/\left( \overline{1 - \kappa_1 e^{-i a_1 \omega_j} - \kappa_2 e^{-i
a_2 \omega_j}} \right )$ so that that these eigenvectors are
normalized with respect to the bilinear form \eqref{bilinear}, that
is,
\[
\langle p_j, q_k \rangle = \delta_{j,k}, \quad\textrm{and}
\quad \langle p_j, \overline q_k \rangle = 0.
\]
Therefore, if we let $\Phi = (q_1,\overline q_1, q_2, \overline q_2)$
and $\Psi = (p_1, \overline p_1, p_2, \overline p_2)^T$.  Then
$\langle \Psi, \Phi \rangle = Id_4$. Hence, $\Phi$ is a basis for $P =
E^c$ and $\Psi$ is a basis for $E^{c*} = P^*$ in $C^*$ and we have
that $\frac{d}{d\theta}{\Phi} = \Phi B$, where
\[
B = \left( \begin{array}{cccc}
 i \omega_1 & 0 & 0 & 0\\
 0 &-i \omega_1 & 0 & 0\\
 0 & 0 & i \omega_2 & 0\\
 0 & 0 & 0 &-i \omega_2
\end{array}\right).
\]
It follows that $BC = P\oplus\ker \Pi$ with $E^s\oplus E^u \subset
\ker \Pi$, where for $\varphi = \phi + X_0 \xi \in BC$ the projection
$\Pi: BC \to P$ is defined by
\[
\Pi(\varphi) = \Pi(\phi + X_0 \xi) = \Phi \langle \Psi, \phi + X_0 \xi \rangle =
\Phi[ \langle \Psi, \phi \rangle + \overline \Psi(0) \xi]
\]
for $\phi \in C$ and $\xi \in \R$.

So the abstract ODE in $BC$ associated with \eqref{eq:rfdelnl} can be
rewritten in the form
\begin{equation}\label{ODE-in-BC1}
\frac{d}{dt}u_t = \cA u_t + X_0 \cF(u_t).
\end{equation}
For the solution $u_t$ of \eqref{ODE-in-BC1} we define
$z_j(t) = \langle p_j, u_t \rangle$ with $j = 1,2$ and
\begin{align} \notag
w(z) & = u_t -  \Bigl(z_1(t) e^{i \omega_1 \theta} + z_2(t) e^{i \omega_2 \theta} +
\overline z_1(t) e^{-i \omega_1 \theta} + \overline z_2(t) e^{-i \omega_2 \theta}\Bigr)\\
& = u_t - 2 \Re\Bigl( z_1(t) e^{i \omega_1 \theta} + z_2(t) e^{i \omega_2 \theta} \Bigr),
\label{ctr-graph}
\end{align}
where $z = (z_1, z_2) \in \C^2$. In fact, $z_j$ and $\overline z_j$
are local coordinates for the center manifold $\cM_{\rm loc}$ in the
directions of $D_j e^{i \omega_j s}$ and $\overline D_j e^{-i \omega_j
s}$, $j = 1,2$.  We notice that
\begin{align*}
\langle p_j, w(z) \rangle
& =\langle D_j e^{i \omega_j s}, w(z) \rangle
=\langle D_j e^{i \omega_j s}, u_t - 2 \Re( z_1(t) e^{i \omega_1
  \theta} + z_2(t) e^{i \omega_2 \theta} )\rangle \\
& = z_j(t) - \langle D_j e^{i \omega_j s}, 2 \Re( z_1(t) e^{i
  \omega_1 \theta} + z_2(t) e^{i \omega_2 \theta} )\rangle
= z_j(t) - z_j(t) = 0.
\end{align*}
Then for the solutions $u_t$ of \eqref{ODE-in-BC1} that belong to
$\cM_{\rm loc}$, we have that
\begin{align*}
\dot z_j(t) &= \langle D_j e^{i \omega_j s}, \dot u_t \rangle = \langle D_j e^{i \omega_j s},
\cA u_t + X_0 \cF(u_t) \rangle \\
&= \langle \cA^* D_j e^{i \omega_j s}, u_t\rangle + \langle D_j e^{i \omega_j s},
X_0 \cF(u_t) \rangle\\
&= i \omega_j z_j(t) + \langle D_j e^{i \omega_j s},
X_0 \cF(u_t) \rangle = i \omega_j z_j(t) + \overline D_j \cF(u_t)\\
&=  i \omega_j z_j(t)
+ \overline D_j \cF\left( w(z) + 2 \Re( z_1(t) e^{i \omega_1 \theta}
+ z_2(t) e^{i \omega_2 \theta} ) \right).
\end{align*}
Therefore, the flow on the center manifold satisfies
\begin{equation}\label{ctr-flow}
\dot z_j(t) = i \omega_j z_j(t) + g^j(z(t)),
\end{equation}
where
\begin{equation}\label{gjz}
g^j(z(t)) =  \overline D_j \cF\left( w(z) + 2 \Re( z_1(t) e^{i
    \omega_1 \theta}+ z_2(t) e^{i \omega_2 \theta} )\right)
\end{equation}
and $w(z)$ satisfies the ODE projected into the complement of $P$,
that is,
\begin{equation}\label{ctr-manifold}
\frac{d}{dt}w = \cA w + ( X_0 - \Phi \overline \Psi(0)) \cF\left(w(z) +
2 \Re( z_1(t) e^{i \omega_1 \theta} + z_2(t) e^{i \omega_2 \theta}) \right).
\end{equation}

%%%%%%%%%%%%%%%%%%%%%%%%%%%%%%%%%%%%%%%%%%%%%%%%%%%
\subsection{Complex ODE}
\label{sec:code}

To find the normal form we need explicit expressions of the flow
\eqref{ctr-flow} on the center manifold. We let
\begin{equation}\label{nonlinearity-exp}
g^j(z) = \sum_{\ell +s +r +k \geq 2}
\frac{1}{\ell! s! r! k!} g^j_{\ell s r k} z_1^\ell \overline z_1^r z_2^r \overline z_2^k
\end{equation}
and
\begin{equation}\label{manifold-exp}
w(z) = \sum_{\ell +s +r +k \geq 2}
\frac{1}{\ell! s! r! k!} w_{\ell s r k}
z_1^\ell \overline z_1^r z_2^r \overline z_2^k.
\end{equation}
Then we can compute the terms of order two and three of the flow
\eqref{ctr-flow} on the center manifold that we will require for the
normal form computation.

The nonlinearity $\cF:C\to\R$ defined by \eqref{fut} for the RFDE
\eqref{eq:rfdelnl} contains only quadratic and cubic terms and, hence,
for $\varphi \in C$ (with $\tau = 3 a_2$) we can expand the
nonlinearity as
 \begin{align} \notag
\cF(\varphi) & = \sum_{i=1}^2\kappa_i(c \varphi(0))L \varphi(-a_i)
+ \sum_{i,j=1}^2\kappa_i\kappa_jc^2 \varphi(0) \varphi(-a_i)L \varphi(-a_i-a_j) \\
& \qquad -\frac12(c \varphi(0))^2\sum_{i=1}^2\kappa_i L^2\varphi(-a_i)\\
& = \frac12 \cF^2(\varphi, \varphi) + \frac16\cF^3(\varphi, \varphi, \varphi),
\end{align}
where $\cF^j$ are the $j$-th order terms given by
\[
\cF^j(\nu_1,...,\nu_j) = \frac{\partial^j}{\partial t_1 \partial t_2 ... \partial t_j}
\cF\left.\left(\sum^j_{s=1}t_s \nu_s\right)\right\vert_{t_1=t_2=...=t_j=0}
\]
and $L$ is the difference operator defined by \eqref{LCop}.  We obtain
that, for $\nu_1, \nu_2, \nu_3 \in C$,
\begin{equation}  \label{eq:F2}
\cF^2(\nu_1, \nu_2) = \sum_{i=1}^2 \kappa_i c [ \nu_1(0) \nu'_2(-a_i)
+  \nu_2(0)) \nu'_1(-a_i)]
\end{equation}
and
\begin{equation} \label{eq:F3}
\cF^3(\nu_1, \nu_2, \nu_3) = -\sum_{\sigma \in S_3}
 c^2 (\nu_{\sigma(1)}(0) \nu_{\sigma(2)}(0))
[\kappa_1 \nu''_{\sigma(3)}(-a_1) + \kappa_2 \nu''_{\sigma(3)}(-a_2)],
\end{equation}
where the first sum is taken over the group $S_3$ of permutations of
three elements. Evaluating these expressions in the elements of the
basis $\Phi$, using \eqref{gjz}, we obtain the terms of the expansion.

The quadratic terms of the flow \eqref{ctr-flow} for the our specific
equation \eqref{eq:rfdelnl} are given by
\begin{align*}
g^j_{2000} &= \overline p_j(0) \cF^2(q_1(\theta), q_1(\theta))
= \overline D_j 2 c \sum_{i =1}^2 \kappa_i e^{-i \omega_1 a_1}  \left(-\gamma
-\kappa_1 e^{-i \omega_1 a_1} -\kappa_2 e^{-i \omega_1 a_2} \right)\\
&= \overline D_j 2 c i \omega_1 \sum_{i =1}^2 \kappa_i e^{-i \omega_1 a_1}
= \overline D_j 2 c i \omega_1 (-\gamma - i \omega_1)
\end{align*}
and similarly
\[
g^j_{0020}=\overline p_j(0) \cF^2(q_1(\theta), q_1(\theta))
=\overline D_j 2 c i \omega_2 \sum_{i =1}^2 \kappa_i e^{-i \omega_2 a_1}
= \overline D_j 2 c i \omega_2 (-\gamma - i \omega_2),
\]
where we have used that $i \omega_j = -\gamma - \kappa_1 e^{-i a_1
  \omega_j} -\kappa_2 e^{-i a_2 \omega_j}$.
The remaining quadratic terms are obtained similarly as
\begin{align*}
g^j_{1100} &= 2 \overline D_j  c  \omega_1^2, \quad
g^j_{0011} = 2 \overline D_j  c  \omega_2^2, \quad \ \
g^j_{1010} = \overline D_j  c (\omega_1^2 + \omega_2^2 -i
\gamma(\omega_1 +\omega_2)), \\
g^j_{0101} &= \overline D_j  c (\omega_1^2 + \omega_2^2 + i
\gamma(\omega_1 +\omega_2)), \quad
g^j_{1001} = \overline D_j  c (\omega_1^2 + \omega_2^2 - i
\gamma(\omega_1 -\omega_2)),\\
g^j_{0110} &= \overline D_j  c (\omega_1^2 + \omega_2^2 + i
\gamma(\omega_1 - \omega_2)),\quad
g^j_{0200} = -\overline D_j 2 c i \omega_1 (-\gamma -i \omega_1),\\
g^j_{0002} &= -\overline D_j 2 c i \omega_2 (-\gamma +i \omega_2).
\end{align*}

Finally, we need to determine a few terms of the expansion of the
graph of the center manifold from \eqref{ctr-manifold}, namely the
terms $w_{1100}, w_{2000}, w_{1010}, w_{1001}, w_{0002}$, and
$w_{0011}$. We will determine these by substituting the expansion
\eqref{manifold-exp} into \eqref{ctr-manifold}. From the definition of
$\cA$ in \eq{inf-gen} this results in a differential equation and a
boundary condition that each coefficient of \eqref{manifold-exp} must
satisfy.

For the coefficient $w_{2000}$ we obtain the differential equation
\begin{equation}\label{w_2000-diff}
\frac{d}{d\theta}w_{2000}(\theta)
= 2 i \omega_1 w_{2000}(\theta)
+g^1_{2000} e^{i \omega_1 \theta}
+\overline g^1_{0200} e^{-i \omega_1 \theta}
+g^2_{2000} e^{i \omega_2 \theta}
+\overline g^2_{2000} e^{-i \omega_2 \theta},
\end{equation}
together with the boundary condition
\begin{equation}\label{w_2000-bndry}
\cL w_{2000} = 2 i \omega_1 w_{2000}(0)
+ g^1_{2000} + \overline g^1_{0200}
+ g^2_{2000} + \overline g^2_{2000} - \cF^2(q_1, q_1).
\end{equation}
The ODE \eq{w_2000-diff} can be solved by using an integrating factor
to obtain
\be
\label{w_2000-gensol}
w_{2000}(\theta) = -\frac{g^1_{2000}  e^{i \omega_1 \theta}}{i \omega_1}
-\frac{\overline{g}^1_{0200}  e^{-i \omega_1 \theta}}{3 i \omega_1}
+\frac{g^2_{2000} e^{i \omega_2 \theta}}{i (\omega_2-2 \omega_1)}
-\frac{\overline{g}^2_{0200}  e^{-i \omega_2 \theta}}{i (\omega_2+2 \omega_1)}
    +E_{2000} e^{2 i \omega_1 \theta}.
\ee
To determine the constant of integration $E_{2000}$, notice that
\eq{w_2000-gensol} implies
\[
\cL w_{2000} =
-\frac{g^1_{2000}(i\omega_1) %e^{i \omega_1 \theta}
}{i \omega_1}
-\frac{\overline{g}^1_{0200}(-i\omega_1)%e^{-i \omega_1 \theta}
}{3 i \omega_1}
 +\frac{g^2_{2000} (i\omega_2) %e^{i \omega_2 \theta}
}{i (\omega_2-2 \omega_1)}
-\frac{\overline{g}^2_{0200}(-i\omega_2)%e^{-i \omega_2 \theta}
}{i (\omega_2+2 \omega_1)}
+E_{2000} (-\Delta(2 \omega_1) + 2 i\omega_1)%e^{2 i \omega_1 \theta}
\]
and
\[
2 i \omega_1 w_{2000} (0) =
-\frac{g^1_{2000}(2 i\omega_1) %e^{i \omega_1 \theta}
}{i \omega_1}
-\frac{\overline{g}^1_{0200}(2 i\omega_1)%e^{-i \omega_1 \theta}
}{3 i \omega_1}
 +\frac{g^2_{2000} (2 i\omega_1) %e^{i \omega_2 \theta}
}{i (\omega_2-2 \omega_1)}
-\frac{\overline{g}^2_{0200}(2 i\omega_1)%e^{-i \omega_2 \theta}
}{i (\omega_2+2 \omega_1)}
+E_{2000} (2i\omega_1).
\]
Substituting these expressions into \eqref{w_2000-bndry} we obtain
\[
E_{2000} = \frac{\cF^2(q_1,q_1)}{\Delta( 2 i \omega_1)}.
\]

We determine $w_{1100}$, $w_{1010}$, $w_{1001}$, $w_{0020}$, and
$w_{0011}$ similarly. The equations that they satisfy are given by
\begin{gather*}
\left.
\begin{array}{rcl}
\displaystyle
\frac{d}{d\theta} w_{1100}(\theta)
& = & g^1_{1100} e^{i \omega_1 \theta}
+\overline g^1_{1100} e^{-i \omega_1 \theta}
+g^2_{1100} e^{i \omega_2 \theta}
+\overline g^2_{100} e^{-i \omega_2 \theta},\\
\cL w_{1100} & = & g^1_{1100} + \overline g^1_{0011}
+ g^2_{1100} + \overline g^2_{0011} - \cF^2(q_2, \overline q_2),
\end{array}\right\} \\
\left.
\begin{array}{rcl}
\displaystyle
\frac{d}{d\theta}w_{1010}(\theta)
& = & i( \omega_1 + \omega_2)  w_{1010}(\theta)
+g^1_{1010} e^{i \omega_1 \theta}
+\overline g^1_{0101} e^{-i \omega_1 \theta}
+g^2_{1010} e^{i \omega_2 \theta}
+\overline g^2_{0101} e^{-i \omega_2 \theta}, \\
\cL w_{1010} & = & i( \omega_1 + \omega_2) w_{1010}(0)
+ g^1_{1010} + \overline g^1_{0101}
+ g^2_{1010} + \overline g^2_{0101} - \cF^2(q_1, q_2),
\end{array}\right\} \\
\left.
\begin{array}{rcl}
\displaystyle
\frac{d}{d\theta}w_{1001}(\theta)
& = & i( \omega_1 - \omega_2)  w_{1001}(\theta)
+g^1_{1001} e^{i \omega_1 \theta}
+\overline g^1_{0110} e^{-i \omega_1 \theta}
+g^2_{1001} e^{i \omega_2 \theta}
+\overline g^2_{0110} e^{-i \omega_2 \theta},\\
\cL w_{1001} & = &  i( \omega_1 - \omega_2) w_{1010}(0)
+ g^1_{1001} + \overline g^1_{0110}
+ g^2_{1001} + \overline g^2_{0110} - \cF^2(\overline q_1, q_2),
\end{array}\right\} \\
\left.
\begin{array}{rcl}
\displaystyle
\frac{d}{d\theta}w_{0020}(\theta)
& = & 2 i \omega_2  w_{0020}(\theta)
+g^1_{0020} e^{i \omega_1 \theta}
+\overline g^1_{0002} e^{-i \omega_1 \theta}
+g^2_{0020} e^{i \omega_2 \theta}
+\overline g^2_{0002} e^{-i \omega_2 \theta},\\
\cL w_{0020} & = &  2 i \omega_2 w_{0020}(0)
+ g^1_{0020} + \overline g^1_{0002}
+ g^2_{0020} + \overline g^2_{0002} - \cF^2(q_2, q_2),
\end{array}\right\} \\
\left.
\begin{array}{rcl}
\displaystyle
\frac{d}{d\theta}w_{0011}(\theta)
& = & g^1_{0011} e^{i \omega_1 \theta}
+\overline g^1_{0011} e^{-i \omega_1 \theta}
+g^2_{0011} e^{i \omega_2 \theta}
+\overline g^2_{0011} e^{-i \omega_2 \theta},\\
\cL w_{0011} & = & g^1_{0011} + \overline g^1_{0011}
+ g^2_{0011} + \overline g^2_{0011} - \cF^2(q_2, \overline q_2).
\end{array}\right\}
\end{gather*}

These equations are solved similarly to \eqref{w_2000-diff} and
\eqref{w_2000-bndry} to obtain expressions equivalent to those of
\cite{Wu-Guo-13} for all the quadratic coefficients of the graph of
the center manifold $w(z)$ as
\begin{gather*}
w_{2000} = -\frac{g^1_{2000}  e^{i \omega_1 \theta}}{i \omega_1}
-\frac{\overline{g}^1_{0200}e^{-i \omega_1 \theta}}{3 i \omega_1}
    +\frac{g^2_{2000} e^{i \omega_2 \theta}}{i (\omega_2-2 \omega_1)}
    -\frac{\overline{g}^2_{0200}e^{-i \omega_2 \theta}}{i (\omega_2+2
      \omega_1)}
    +E_{2000} e^{2 i \omega_1 \theta},\\
w_{1100} = \frac{g^1_{1100}  e^{i \omega_1 \theta}}{i \omega_1}
-\frac{\overline{g}^1_{1100}e^{-i \omega_1 \theta}}{i \omega_1}
    +\frac{g^2_{1100} e^{i \omega_2 \theta}}{i \omega_2}
    -\frac{\overline{g}^2_{1100} e^{-i \omega_2 \theta}}{i \omega_2}
    +E_{1100},\\
w_{1010} = -\frac{g^1_{1010}  e^{i \omega_1 \theta}}{i \omega_2}
-\frac{\overline{g}^1_{0101} e^{-i \omega_1 \theta}}{i (2
  \omega_1+\omega_2)}
    -\frac{g^2_{1010} e^{i \omega_2 \theta}}{i \omega_1}
    -\frac{\overline{g}^2_{0101}  e^{-i \omega_2 \theta}}{i
      (\omega_1+2 \omega_2)}
    +E_{1010} e^{i (\omega_1+\omega_2) \theta},\\
w_{1001} = \frac{g^1_{1001}  e^{i \omega_1 \theta}}{i \omega_2}
+\frac{\overline{g}^1_{0110}e^{-i \omega_1 \theta}}{i (\omega_2-2
  \omega_1)}
    +\frac{g^2_{1001} e^{i \omega_2 \theta}}{i (2 \omega_2-\omega_1)}
    -\frac{\overline{g}^2_{0110} e^{-i \omega_2 \theta}}{i \omega_1}
    +E_{1001} e^{i (\omega_1-\omega_2) \theta},\\
w_{0020} = \frac{g^1_{0020}  e^{i \omega_1 \theta}}{i (\omega_1-2
  \omega_2)} -\frac{\overline{g}^1_{0002} e^{-i \omega_1 \theta}}{i
  (\omega_1+2 \omega_2)}
    -\frac{g^2_{0020} e^{i \omega_2 \theta}}{i \omega_2}
    -\frac{\overline{g}^2_{0002} e^{-i \omega_2 \theta}}{3 i \omega_2}
    +E_{0020} e^{2 i \omega_2 \theta},\\
w_{0011} = \frac{g^1_{0011}  e^{i \omega_1 \theta}}{i \omega_1}
-\frac{\overline{}g^1_{0011} e^{-i \omega_1 \theta}}{i \omega_1}
    +\frac{g^2_{0011} e^{i \omega_2 \theta}}{i \omega_2}
    -\frac{\overline{g}^2_{0011}  e^{-i \omega_2 \theta}}{i \omega_2}
    +E_{0011},
\end{gather*}
where the constants of integration are given by
\begin{gather*}
E_{1100} = \frac{\cF^2(q_1,\overline q_1)}{\Delta( 0)},\qquad
E_{2000} = \frac{\cF^2(q_1,q_1)}{\Delta( 2 i \omega_1)},\qquad
E_{1010} = \frac{\cF^2(q_1,q_2)}{\Delta( i (\omega_1+\omega_2))},\\
E_{1001} = \frac{\cF^2(q_1,\overline q_2)}{\Delta( i (\omega_1-\omega_2))},\qquad
E_{0020} = \frac{\cF^2(q_2,q_2)}{\Delta( 2 i \omega_2)},\qquad
E_{0011} = \frac{\cF^2(q_2,\overline q_2)}{\Delta( 0)},
\end{gather*}
and $\Delta$ is the characteristic function defined in
\eqref{characteristic-equation}.  Finally the cubic terms are given by
the expressions
\be \label{G-H_normal_form}
\begin{split}
g^j_{2100} &= \overline D_j \cF^3 (q_1 , q_1 , \overline q_1 )
+ 2\overline D_j \cF^2 (q_1 , w_{1100} ) + \overline D_j \cF^2 (\overline q_1 , w_{2000} ),\\
g^j_{1011} &= \overline D_j \cF^3 (q_1 , q_2 ,\overline q_2 )
+ \overline D_j \cF^2 (q_1 , w_{0011} )
+\overline D_j \cF^2 (q_2 , w_{1001} )
+ \overline D_j \cF^2 (\overline q_2 , w_{1010} ),\\
g^j_{1110} &= \overline D_j \cF^3 (q_1 ,\overline q_1 , q_2 )
+ \overline D_j \cF^2 (q_1 ,\overline w_{1001} )
+\overline D_j \cF^2 (q_2 , w_{1100} )
+ \overline D_j \cF^2 (\overline q_1 , w_{1010} ),\\
g^j_{0021} &= \overline D_j \cF^3 (q_2 , q_2 ,\overline q_2 )
+ 2\overline D_j \cF^2 (q_2 , w_{0011} )
+ \overline D_j \cF^2 (\overline q_2 , w_{0020} ),
\end{split}
\ee
where, from formula \eqref{eq:F3},
\begin{gather*}
\cF^3(q_1, q_1, \overline q_1) = - c^2 \omega_1^2[3 \gamma + i \omega_1], \quad
\cF^3(q_2, q_2, \overline q_2) = - c^2 \omega_2^2[3 \gamma + i \omega_2], \\
\cF^3(q_1,  \overline q_1, q_2) = - c^2 [ \gamma ( 2\omega_1^2 + \omega_2^2)
+ i \omega_2^3], \quad
\cF^3(q_1, q_2, \overline q_2) = - c^2 [ \gamma ( \omega_1^2 + 2\omega_2^2)
+ i \omega_1^3]. \\
\end{gather*}

%%%%%%%%%%%%%%%%%%%%%%%%%%%%%%%%%%%%%%%%%%%%%%%%%%%
\subsection{Normal form}
\label{secapp:NormalForm}

To determine the bifurcation structure near a Hopf-Hopf point, we
follow the the approach of Kuznetsov \cite{Kuz-04-Book}. Kuznetsov
considers the same ODE \eqref{ctr-flow} (\cite[Equation
(8.88)]{Kuz-04-Book}) on the generalized center eigenspace $E^c$ with
$g^j(z(t))$ defined by \eqref{gjz}, but he expands $g^j(z)$ as
\begin{equation}\label{nonlinearity-exp-Kuz}
g^j(z) = \sum_{\ell +s +r +k \geq 2}
\widetilde{g}^j_{\ell s r k}
z_1^\ell \overline z_1^r z_2^r \overline z_2^k.
\end{equation}
Comparing \eqref{nonlinearity-exp-Kuz} with \eqref{nonlinearity-exp}
we see that we require
\begin{equation}\label{nonlinearity-conv}
\widetilde{g}^j_{\ell s r k}=\frac{1}{\ell! s! r! k!} g^j_{\ell s r k}.
\end{equation}
Kuznetsov is inconsistent between papers on whether or not he includes
the factorial terms in the expansion of $w(z)$ in his version of
\eq{manifold-exp}, but that is irrelevant to our exposition in this
section, because we only use $w(z)$ in the previous section to project
the center manifold onto the generalized center eigenspace.  As such,
terms from the expansion of $w(z)$ appear in $g^j_{\ell s r k}$, but
these were computed already in the previous section.  We then have

\begin{lemma}[Poincar\'e Normal Form (Lemma 8.13 in Kuznetsov
  \cite{Kuz-04-Book})]
Assume the non-resonance condition
\begin{itemize}
\item[{\rm(HH.0)}] $ k \omega_1 \neq \ell \omega_2$ for $k,\ell \in
  \N_0$ with  $k+\ell \leq 5$.
\end{itemize}
Then there exists a locally defined smooth and smoothly
parameter-dependent invertible transformation of the complex variables
that for all sufficiently small $\|\alpha\|$ (where
$\alpha=(\kappa_1-\kappa_1^*,\kappa_2-\kappa_2^*)$) reduces
\eqref{ctr-flow} to
\be \label{K8.89}
\begin{split}
\dot{w}_1 & = \lambda_1(\alpha)w_1 +
G_{2100}^1(\alpha)w_1|w_1|^2+G_{1011}^1(\alpha)w_1|w_2|^2+
G_{3200}^1(\alpha)w_1|w_1|^4 \\
& \mbox{}\qquad +G_{2111}^1(\alpha)w_1|w_1|^2|w_2|^2+
G_{1022}^1(\alpha)w_1|w_2|^4+\cO(\|(w_1,\overline{w}_1,w_2,\overline{w}_2)\|^6),\\
\dot{w}_2 & = \lambda_2(\alpha)w_2 +
G_{0021}^2(\alpha)w_2|w_2|^2+G_{1110}^2(\alpha)w_2|w_1|^2+
G_{0032}^2(\alpha)w_2|w_2|^4 \\
& \mbox{}\qquad +G_{1121}^2(\alpha)w_2|w_1|^2|w_2|^2+
G_{2210}^2(\alpha)w_2|w_1|^4+\cO(\|(w_1,\overline{w}_1,w_2,\overline{w}_2)\|^6),
\end{split}
\ee
where $w_{1,2}\in\C$ and
$\|(w_1,\overline{w}_1,w_2,\overline{w}_2)\|^2=|w_1|^2+|w_2|^2$. The
complex-valued functions ${G}^{1,2}_{\ell s r k}(\alpha)$ are smooth
and, moreover,
\begin{align} \notag
G^1_{2100}(0)&=\widetilde{g}^1_{2100}+\frac{i}{\omega_1}\widetilde{g}^1_{1100}\,\widetilde{g}^1_{2000}
+\frac{i}{\omega_2}(\widetilde{g}^1_{1010}\,\widetilde{g}^2_{1100}-\widetilde{g}^1_{1001}\,\gtb^2_{1100})
-\frac{i}{2\omega_1+\omega_2}\widetilde{g}^1_{0101}\,\gtb^2_{0200} \\
& \mbox{} \qquad
-\frac{i}{2\omega_1-\omega_2}\widetilde{g}^1_{0110}\,\widetilde{g}^2_{2000}
-\frac{i}{\omega_1}|\widetilde{g}^1_{1100}|^2-\frac{2i}{3\omega_1}|\widetilde{g}^1_{0200}|^2, \label{K8.90}\\
G^1_{1011}(0)&=\widetilde{g}^1_{1011},
+\frac{i}{\omega_2}(\widetilde{g}^1_{1010}\,\widetilde{g}^2_{0011}-\widetilde{g}^1_{1001}\,\gtb^2_{0011}) \notag\\
& \mbox{} \qquad
+\frac{i}{\omega_1}(2\widetilde{g}^1_{2000}\,\widetilde{g}^1_{0011}-\widetilde{g}^1_{1100}\,\gtb^1_{0011}
-\widetilde{g}^2_{1010}\,\widetilde{g}^1_{0011}-\widetilde{g}^1_{0011}\,\gtb^2_{0110})
-\frac{2i}{\omega_1+2\omega_2}\widetilde{g}^1_{0002}\,\gtb^2_{0101} \notag \\
& \mbox{} \qquad
-\frac{2i}{\omega_1-2\omega_2}\widetilde{g}^1_{0020}\,\widetilde{g}^2_{1001}
-\frac{i}{2\omega_1-\omega_2}|\widetilde{g}^1_{0110}|^2-\frac{i}{2\omega_1+\omega_2}|\widetilde{g}^1_{0101}|^2, \label{K8.91}\\
G^2_{1110}(0)&=\widetilde{g}^2_{1110}
+\frac{i}{\omega_1}(\widetilde{g}^1_{1100}\,\widetilde{g}^2_{1010}-\widetilde{g}^2_{0110}\,\gtb^1_{1100}) \notag\\
& \mbox{} \qquad
+\frac{i}{\omega_2}(2\widetilde{g}^2_{0020}\,\widetilde{g}^2_{1100}-\widetilde{g}^2_{0011}\,\gtb^2_{1100}
-\widetilde{g}^1_{1010}\,\widetilde{g}^2_{1100}-\widetilde{g}^2_{1100}\,\gtb^1_{1001})
-\frac{2i}{2\omega_1+\omega_2}\widetilde{g}^2_{0200}\,\gtb^1_{0101} \notag \\
& \mbox{} \qquad
+\frac{2i}{2\omega_1-\omega_2}\widetilde{g}^2_{2000}\,\widetilde{g}^1_{0110}
+\frac{i}{\omega_1-2\omega_2}|\widetilde{g}^2_{1001}|^2-\frac{i}{\omega_1+2\omega_2}|\widetilde{g}^2_{0101}|^2,  \label{K8.92}\\
G^2_{0021}(0)&=\widetilde{g}^2_{0021}+\frac{i}{\omega_2}\widetilde{g}^2_{0011}\,\widetilde{g}^2_{0020}
+\frac{i}{\omega_1}(\widetilde{g}^2_{1010}\,\widetilde{g}^1_{0011}-\widetilde{g}^2_{0110}\,\gtb^1_{0011})
-\frac{i}{2\omega_2+\omega_1}\widetilde{g}^2_{0101}\,\gtb^1_{0002} \notag \\
& \mbox{} \qquad
-\frac{i}{2\omega_2-\omega_1}\widetilde{g}^2_{1001}\,\widetilde{g}^1_{0020}
-\frac{i}{\omega_2}|\widetilde{g}^2_{0011}|^2-\frac{2i}{3\omega_2}|\widetilde{g}^2_{0002}|^2,
 \label{K8.93}
\end{align}
where all the $\widetilde{g}^j_{\ell s r k}$ are evaluated at $\alpha=0$.
\end{lemma}

Note that the last two terms in each of the expressions
\eq{K8.90}--\eq{K8.93} are purely imaginary; these terms will vanish
when we take real parts later.

We next make a near identity transformation
$$v_1=w_1+K_1w_1|w_1|^2, \qquad v_2=w_2+K_2w_2|w_2|^2,$$
and introduce a new time $\tau$ with
$$dt=(1+e_1|w_1|^2+e_2|w_2|^2)d\tau,$$
where $K_{1,2}(\alpha)$ and $e_{1,2}(\alpha)$ are chosen judiciously,
to give the following result, where $\dot{v}_{1,2}$ indicates the
derivative with respect to $\tau$.

\begin{lemma}[Lemma 8.14 in Kuznetsov \cite{Kuz-04-Book}] \label{lemK8.14}
Assume that
\begin{itemize}
\item[\rm(HH.1)] $\Re G^1_{2100}(0) \neq 0$ ;
\item[\rm(HH.2)] $\Re G^1_{1011}(0) \neq 0$ ;
\item[\rm(HH.3)] $\Re G^2_{1110}(0) \neq 0$ ;
\item[\rm(HH.4)] $\Re G^2_{0021}(0) \neq 0$ ;
\end{itemize}
then the system \eq{K8.89} is locally smoothly orbitally equivalent to
\be \label{K8.95} \left.
\begin{aligned}
\dot{v}_1 & = \lambda_1(\alpha)v_1 + P_{11}(\alpha)v_1|v_1|^2+P_{12}(\alpha)v_1|v_2|^2+
iR_{1}(\alpha)v_1|v_1|^4+ S_{1}(\alpha)v_1|v_2|^4\\
& \mbox{}\qquad +\cO(\|(v_1,\overline{v}_1,v_2,\overline{v}_2)\|^6),\\
\dot{v}_2 & = \lambda_2(\alpha)v_2 + P_{21}(\alpha)v_2|v_1|^2+P_{22}(\alpha)v_2|v_2|^2+
S_{2}^2(\alpha)v_2|v_2|^4 +iR_{2}(\alpha)v_2|v_2|^4\\
& \mbox{}\qquad +\cO(\|(v_1,\overline{v}_1,v_2,\overline{v}_2)\|^6),
\end{aligned} \right\}
\ee
where $v_{1,2}$ are new complex variables, $P_{jk}(\alpha)$ and
$S_{k}(\alpha)$ are complex-valued smooth functions, and $R_k(\alpha)$
are real-valued smooth functions.
\end{lemma}

From the proof of Lemma~\ref{lemK8.14} we obtain
\be
\begin{split} \label{K8.1045}
\Re P_{11}(0)=\Re G^1_{2100}(0), & \quad \Re P_{12}(0)=\Re G^1_{1011}(0),  \\ %\label{K8.104}
\Re P_{21}(0)=\Re G^2_{1110}(0), & \quad \Re P_{22}(0)=\Re G^2_{0021}(0).  %\label{K8.105}
\end{split}
\ee
This follows because
$P_{11}(\alpha)=\hat{G}^1_{2100}=G^1_{2100}+\lambda_1e_1+(\lambda_1+\overline{\lambda}_1)K_1$
and $\lambda_1(0)=i\omega_1$, while $e_1(0)=-\Re G^1_{3200}(0)/\Re
G^1_{2100}(0)$, so
$$P_{11}(0)=G^1_{2100}(0)-i\omega_1\frac{\Re G^1_{3200}(0)}{\Re G^1_{2100}(0)},$$
which implies that $\Re P_{11}(0)=\Re G^1_{2100}(0)$. The other
identities in \eq{K8.1045} follow similarly.

Next we rewrite the system \eq{K8.95} in polar coordinates
$(r_1,r_2,\phi_1,\phi_2)$ by letting
$$v_1=r_1e^{i\phi_1}, \qquad v_2=r_2e^{i\phi_2}.$$
Writing $v_i=x_i+y_i$ and ignoring the higher-order terms for a
moment, we have $r_i^2=x_i^2+y_i^2$ and
\begin{align*}
r_i\rdot_i &= x_i\xdot_i+y_i\ydot_i = x_i\Re\vdot_i+y_i\Im\vdot_i\\
& = x_i\bigl(\Re(\lambda_iv_i)+\Re(P_{i1}v_i)r_1^2+\Re(P_{i2}v_i)r_2^2-R_iy_ir_i^4+\Re(S_iv_i)r_{3-i}^4\bigr)\\
& \mbox{}\qquad+y_i\bigl(\Im(\lambda_iv_i)+\Im(P_{i1}v_i)r_1^2+\Im(P_{i2}v_i)r_2^2+R_ix_ir_i^4+\Im(S_iv_i)r_{3-i}^4\bigr)\\
&= x_i\Bigl(\mu_ix_i-\omega_iy_i+(\Re(P_{i1})x_i-\Im(P_{i1})y_i)r_1^2+(\Re(P_{i2})x_i-\Im(P_{i2})y_i)r_2^2-R_iy_ir_i^4\\
&\mbox{}\qquad+(\Re(S_i)x_i-\Im(S_i)y_i)r_{3-i}^4\Bigr)
+y_i\Bigl(\mu_iy_i+\omega_ix_i+(\Im(P_{i1})x_i+\Re(P_{i1})y_i)r_1^2\\
&\mbox{}\qquad+(\Im(P_{i2})x_i+\Re(P_{i2})y_i)r_2^2+R_ix_ir_i^4+(\Im(S_i)x_i+\Re(S_i)y_i)r_{3-i}^4\Bigr)\\
& = \mu_i(x_i^2+y_i^2) + \Re(P_{i1})(x_i^2+y_i^2)r_1^2 + \Re(P_{i2})(x_i^2+y_i^2)r_2^2 + \Re(S_i)(x_i^2+y_i^2)r_{3-i}^4 \\
& = \mu_i r_i^2 + \Re(P_{i1})r_i^2r_1^2 + \Re(P_{i2})r_i^2r_2^2 + \Re(S_i)r_i^2r_{3-i}^4,
\end{align*}
where $3-i=1$ when $i=2$ and  $3-i=2$ when $i=1$, and so denotes the
other index.

Then \eq{K8.95} can be written as
\be
\begin{split}
\rdot_1&=r_1\bigl(\mu_1(\alpha)+p_{11}(\alpha)r_1^2+p_{12}(\alpha)r_2^2+s_1(\alpha)r_2^4\bigr)
+\Phi_1(r_1,r_2,\phi_1,\phi_2,\alpha),\\
\rdot_2&=r_2\bigl(\mu_2(\alpha)+p_{21}(\alpha)r_1^2+p_{22}(\alpha)r_2^2+s_2(\alpha)r_1^4\bigr)
+\Phi_2(r_1,r_2,\phi_1,\phi_2,\alpha),\\
\phidot_1&=\omega_1(\alpha)+\Psi_1(r_1,r_2,\phi_1,\phi_2,\alpha),\\
\phidot_2&=\omega_2(\alpha)+\Psi_2(r_1,r_2,\phi_1,\phi_2,\alpha),
\end{split}
\ee
where
\be \label{K357}
p_{jk}=\Re P_{jk}, \quad s_j=\Re S_j, \quad j,k=1,2.
\ee

If the map
$(\kappa_1,\kappa_2)\mapsto(\mu_1(\kappa_1,\kappa_2),\mu_2(\kappa_1,\kappa_2))$
is regular at $(\kappa_1^*,\kappa_2^*)$ or, equivalently, the map
$\alpha\mapsto(\mu_1(\alpha),\mu_2(\alpha))$ is regular at $\alpha=0$,
that is,
$\det\left.\left(\tfrac{\partial(\mu_1,\mu_2)}{\partial(\alpha_1,\alpha_2)}\right)\right|_{\alpha=0}\ne0$,
then one can use $(\mu_1,\mu_2)$ to parameterize a small neighbourhood
of $(\kappa_1^*,\kappa_2^*)$ in the parameter plane and, hence, regard
the functions of $\alpha$ in the theory above as functions of
$(\mu_1,\mu_2)$, which are the real parts of the eigenvalues, which
vanish at the Hopf-Hopf bifurcation. This condition is easy to verify
since
\be \label{eq:dlambdadkj}
\left.\frac{\partial\mu_i}{\partial \kappa_j}\right|_{\alpha=0}
=\left.\Re\left(\frac{\partial\lambda}{\partial \kappa_j}\right)\right|_{\lambda=i\omega_i},\quad \textrm{and} \quad
\frac{\partial\lambda}{\partial\kappa_j}
=\frac{-e^{-a_j\lambda}}{1-a_j\kappa_je^{-a_j\lambda}-a_{3-j}\kappa_{3-j}e^{-a_{3-j}\lambda}},
\ee
where the last expression follows from differentiating \eq{chareq}.
We obtain the following theorem, adapted from Theorem~8.8 in Kuznetsov
\cite{Kuz-04-Book}.

\begin{theorem} \label{thmK8.8}
Consider the constant delay DDE \eq{eq:rfdelnl}, where the linear
operator $\cL u_t$ is defined by \eq{eq:lop} and nonlinear operator
$\cF(u_t)$ is given by \eq{fut}, with parameters $(\kappa_1,\kappa_2)$
which has eigenvalues
$$\lambda_j(\kappa_1,\kappa_2)=\mu_j(\kappa_1,\kappa_2)\pm i\omega_j(\kappa_1,\kappa_2), \quad j=1,2$$
with
$$\mu_j(\kappa_1^*,\kappa_2^*)=0, \quad \omega_j(\kappa_1^*,\kappa_2^*)=\omega_j, \quad j=1,2.$$
If the nondegeneracy conditions
\begin{itemize}
\item[\rm(HH.0)] $ k \omega_1 \neq \ell \omega_2$ for $k,\ell \in \N_0$ with  $k+\ell \leq 5$,
\item[\rm(HH.1)] $p_{11}(\kappa_1^*,\kappa_2^*)=\Re G^1_{2100}(0) \neq 0$;
\item[\rm(HH.2)] $p_{12}(\kappa_1^*,\kappa_2^*)=\Re G^1_{1011}(0) \neq 0$;
\item[\rm(HH.3)] $p_{21}(\kappa_1^*,\kappa_2^*)=\Re G^2_{1110}(0) \neq 0$;
\item[\rm(HH.4)] $p_{22}(\kappa_1^*,\kappa_2^*)=\Re G^2_{0021}(0) \neq 0$;
\end{itemize}
hold, where the ${G}^{1,2}_{\ell s r k}(0)$ are defined by \eq{K8.90}-\eq{K8.93},
and
\begin{itemize}
\item[\rm(HH.5)] the map
  $(\kappa_1,\kappa_2)\mapsto(\mu_1(\kappa_1,\kappa_2),\mu_2(\kappa_1,\kappa_2))$
  is regular at $(\kappa_1^*,\kappa_2^*)$,
\end{itemize}
then the system is locally orbitally equivalent near the origin to
\be \label{ode4d}
\left.
\begin{aligned}
\dot{r}_1 &= r_1(\mu_1 + p_{11}(\mu)r_1^2 + p_{12}(\mu)r_2^2 + s_1(\mu)r_2^4) + \cO((r_1^2 + r_2^2)^3),\\
\dot{r}_2 &= r_2(\mu_2 + p_{21}(\mu)r_1^2 + p_{22}(\mu)r_2^2 + s_2(\mu)r_1^4) + \cO((r_1^2 + r_2^2)^3),\\
\dot{\varphi}_1 &=  \omega_1(\mu) +\Psi_1(r_1,r_2,\phi_1,\phi_2,\mu),\\
\dot{\varphi}_2 &=  \omega_2(\mu) +\Psi_2(r_1,r_2,\phi_1,\phi_2,\mu),
\end{aligned}\right\}
\ee
where $\Psi_j(0,0,\phi_1,\phi_2,\mu)=0$.
\end{theorem}

We remark that Kuznetsov \cite{Kuz-04-Book} also gives a formula for
the $s_j(0)$, but we will not need this and, anyway, it requires terms
${G}^{1,2}_{\ell s r k}(0)$ that he does not state.

%%%%%%%%%%%%%%%%%%%%%%%%%%%%%%%%%%%%%%%%%%%%%%%%%%%
\subsection{Determining the normal form bifurcation diagram}
\label{sec:DHopfUnfold}

To determine the dynamics and bifurcation near a Hopf-Hopf point we
will apply Theorem~\ref{thmK8.8} to the constant delay DDE
\eq{eq:dde3}, which was written as an RFDE of the form
\eq{eq:rfdelnl}.  We do not need to consider the angle equations for
$\varphi_{1,2}$ from \eq{ode4d} because $\omega_{1,2}(0)>0$ and
$\Psi_{1,2}(0,0,\phi_1,\phi_2,\mu)=0$, so close to a Hopf-Hopf
bifurcation these equations just describe rotations.  Nearly nobody
also computes the functions $p_{ij}(\mu)$ appearing in
Theorem~\ref{thmK8.8}. To determine the qualitative bifurcation
diagram it is sufficient to consider the truncated amplitude system
\be \label{K8.110}
\begin{split}
\dot{r}_1 &= r_1(\mu_1 + p_{11}r_1^2 + p_{12}r_2^2 + s_1r_2^4),\\
\dot{r}_2 &= r_2(\mu_2 + p_{21}r_1^2 + p_{22}r_2^2 + s_2r_1^4).
\end{split}
\ee
Here $p_{ij}$ and $s_i$ are formally functions of $\mu$, but it is
sufficient to calculate $p_{ij}(0)$ and $s_i(0)$ to determine the
bifurcation diagram.  We only need to consider positive amplitudes
and, following Kuznetzov \cite{Kuz-04-Book}, we let
$\rho_j=r_j^2\geq0$ and rewrite the amplitude equations \eq{K8.110} as
\be \label{K8.111}
\begin{split}
\dot{\rho}_1 &= 2\rho_1(\mu_1 + p_{11}\rho_1 + p_{12}\rho_2 + s_1\rho_2^2),\\
\dot{\rho}_2 &= 2\rho_2(\mu_2 + p_{21}\rho_1 + p_{22}\rho_2 + s_2\rho_1^2).
\end{split}
\ee
Notice that an equlibrium of these equations with $\rho_1=\rho_2=0$
corresponds to the trivial steady state of \eq{eq:dde3}. An
equilibrium of the amplitude equations with exactly one of $\rho_i$
non-zero corresponds to a periodic orbit of \eq{eq:dde3} (because of
rotation from the angle equations), while an equilibrium of the
amplitude equations with both $\rho_i$ non-zero corresponds to a
two-dimensional torus for \eq{eq:dde3}. A periodic orbit of
\eq{K8.111} corresponds to a three-dimensional torus for \eq{ode4d}
and \eq{eq:dde3}.

There are several possible cases, but we focus on the case where
$p_{11}<0$ and $p_{22}<0$, since it arises at $\HH_j$ for $j=1$, $2$
and $3$. We make the change of coordinates
\be \label{K8.1105}
\xi_1=-p_{11}\rho_1, \qquad \xi_2=-p_{22}\rho_2, \quad \tau=2t
\ee
in  \eq{K8.111},  yielding
\begin{equation} \label{K8.112}
\begin{split}
\xi_1' &= \xi_1(\mu_1 -\xi_1-\vartheta\xi_2 + \Theta\xi_2^2),\\
\xi_2' &= \xi_2(\mu_2 -\xi_2-\delta\xi_1 + \Delta\xi_1^2),
\end{split}
\end{equation}
where
$$\vartheta=\frac{p_{12}}{p_{22}},
\quad \delta=\frac{p_{21}}{p_{11}}, \quad
\Theta=\frac{s_{1}}{p_{22}^2}, \quad \Delta=\frac{s_{2}}{p_{11}^2}.$$
Recalling that $\rho_i\geq0$, with $p_{11}<0$ and $p_{22}<0$ the minus
signs are incorporated into the change of coordinates \eq{K8.1105} so
that only solutions of \eq{K8.112} with $\xi_i\geq0$ for each $i$
correspond to solutions of \eq{K8.110}.

Kuznetsov \cite{Kuz-04-Book} only analyses the case $\vartheta\geq\delta$
and suggests to make a change of coordinates otherwise, but actually
it is easy to deal directly with all the cases where
$\vartheta\neq\delta$.

Equation \eq{K8.112} has a steady state at $(\xi_1,\xi_2)=(0,0)$ for
all values of the parameters, corresponding to the steady state of
\eq{ode4d} and \eq{eq:dde3}. For $\mu_1>0$ there is another steady
state of \eq{K8.112} with $(\xi_1,\xi_2)=(\mu_1,0)$. This corresponds
to a periodic orbit for \eq{ode4d} that bifurcates from the steady
state along the Hopf bifurcation curve
$$H_1=\{(\mu_1,\mu_2):\mu_1=0\}.$$
Similarly, for $\mu_2>0$ there is a third steady state of \eq{K8.112}
with $(\xi_1,\xi_2)=(0,\mu_2)$ corresponding to another periodic orbit
for \eq{ode4d} that bifurcates from the steady state along the Hopf
bifurcation curve
$$H_2=\{(\mu_1,\mu_2):\mu_2=0\}.$$
Finally, let us look for the torus and torus bifurcations.  We seek a
steady state of \eq{K8.112} not on the coordinate axes, so we require
\be \label{impfthmeqss} 0=\mu_1 -\xi_1-\vartheta\xi_2 +
\Theta\xi_2^2=\mu_2 -\xi_2-\delta\xi_1 + \Delta\xi_1^2.  \ee Applying
the implicit function theorem, we can find a function
$(\xi_1,\xi_2)=g(\mu_1,\mu_2)$ such that $(\mu_1,\mu_2,\xi_1,\xi_2)$
satisfy \eq{impfthmeqss} provided the appropriate Jacobian matrix is
nonzero, for which we require $\vartheta\delta-1\ne0$ or, equivalently,
\begin{itemize}
\item[\rm(HH.6)] \quad $\det\left(\begin{array}{cc}
p_{11}(0) & p_{12}(0) \\
p_{21}(0) & p_{22}(0)
\end{array}\right)\ne0$.
\end{itemize}
Then the implicit function theorem gives a steady-state solution of
\eq{K8.112} with
\be \label{hh1tor}
\xi_1=\frac{\vartheta\mu_2-\mu_1}{\vartheta\delta-1}+\cO(\mu_1^2+\mu_2^2), \qquad
\xi_2=\frac{\delta\mu_1-\mu_2}{\vartheta\delta-1}+\cO(\mu_1^2+\mu_2^2).
\ee
(This can also be seen by letting $\Theta=\Delta=0$ in \eq{K8.112} and
solving directly for $\xi_{1,2}$.). Recall that we need $\xi_{1,2}>0$
for the solution \eq{hh1tor} to correspond to a torus of \eq{ode4d}
and \eq{eq:dde3}. If $\delta\vartheta-1<0$ we then require
\be \label{eq:T12cond}
\delta\mu_1<\mu_2, \qquad \vartheta\mu_2<\mu_1
\ee
to satisfy this condition close to the bifurcation point; or with the
inequalities reversed, if $\delta\vartheta-1>0$.  This defines the torus
bifurcation curves $T_1$ and $T_2$ which both start at
$(\mu_1,\mu_2)=(0,0)$. To leading order, these satisfy one strict
inequality in \eq{eq:T12cond} with equality in the other expression;
and the torus exists in the cone for which both equalities hold.

If $\vartheta>0>\delta$ we obtain
\begin{gather} \label{eq:T1}
T_1=\{(\mu_1,\mu_2):\mu_2=\delta\mu_1+\cO(\mu_1^2), \; \mu_1>0 \},\\
T_2=\{(\mu_1,\mu_2):\mu_1=\vartheta\mu_2+\cO(\mu_2^2), \; \mu_2>0 \}, \label{eq:T2}
\end{gather}
with the torus existing between them with $\mu_1>0$.  Notice that, as
$(\mu_1,\mu_2)\to T_1$, we have $(\xi_1,\xi_2)\to(\mu_1,0)$, which is
the fixed point corresponding to the periodic orbit created in the
$H_1$ Hopf bifurcation. Similarly, as $(\mu_1,\mu_2)\to T_2$, we have
$(\xi_1,\xi_2)\to(0,\mu_2)$. Kuznetsov identifies this as Case III of
five cases depending on the signs of $\vartheta$, $\delta$ and
$\delta\vartheta-1$, which result in topologically different bifurcation
diagrams.

Once the normal form is calculated it is actually straightforward to
transform back to the original parameters $(\kappa_1,\kappa_2)$.  The
linear part of the mapping
$(\kappa_1,\kappa_2)\mapsto(\mu_1(\kappa_1,\kappa_2),\mu_2(\kappa_1,\kappa_2))$
is defined by
\[
\left(\!\begin{array}{r}
\mu_1\\
\mu_2
\end{array}\!\right)
=
\left(\!\begin{array}{r}
\Re\lambda_1\\
\Re\lambda_2
\end{array}\!\right)
=
\left(\!\begin{array}{rr}
\left.\Re\left(\frac{\partial\lambda}{\partial\kappa_1}\right)\right|_{\lambda=i\omega_1}
& \left.\Re\left(\frac{\partial\lambda}{\partial\kappa_2}\right)\right|_{\lambda=i\omega_1} \\
\left.\Re\left(\frac{\partial\lambda}{\partial\kappa_1}\right)\right|_{\lambda=i\omega_2}
& \left.\Re\left(\frac{\partial\lambda}{\partial\kappa_2}\right)\right|_{\lambda=i\omega_2}
\end{array}\!\right)
\left(\!\begin{array}{r}
\kappa_1-\kappa_1^*\\
\kappa_2-\kappa_2^*
\end{array}\!\right)
=
J\left(\!\begin{array}{r}
\kappa_1-\kappa_1^*\\
\kappa_2-\kappa_2^*
\end{array}\!\right),
\]
where the entries in the Jacobian matrix $J$ are calculated from
\eq{eq:dlambdadkj}. By \textrm{(HH.5)} the Jacobian is invertible and,
hence, to leading order
\be \label{invtransf}
\left(\begin{array}{r}
\kappa_1\\
\kappa_2
\end{array} \right)
=
\left(\begin{array}{r}
\kappa_1^*\\
\kappa_2^*
\end{array} \right)
+J^{-1}\left(\begin{array}{r}
\mu_1\\
\mu_2
\end{array} \right).
\ee

The calculation of the normal form coefficients for the
state-dependent DDE \eqref{eq:twostatedep} is implemented in the
accompanying Matlab \cite{Matlab} code \NormalForm.  This code first
uses symbolic differentiation to compute the constant delay expansion
of the DDE described \sref{sec:expansion}. The exact locations of the
Hopf-Hopf points are computed, as described in \sref{sec:centreman},
with the auxiliary routine \texttt{findHH}. Finally, the coefficients
of the Hopf-Hopf normal form are computed as described in
Appendix~\ref{app:NormalForm}.  At any of the points $\HH_1$, $\HH_2$ and
$\HH_3$, the code \NormalForm identifies the Hopf-Hopf bifurcation and
computes its respective normal form in about $3.7$ seconds on a Lenovo
Thinkpad X230.

The results of these calculations are summarized in
Tables~\ref{tab:HH1} and~\ref{tab:HHj}.

Although applied only to the state-dependent DDE
\eqref{eq:twostatedep} here, the code \NormalForm is nevertheless
general purpose. To compute Hopf-Hopf normal forms for a different
state-dependent DDE it would be necessary only to:
\begin{enumerate}[1.]
\item change the definition of the nonlinearity and the characteristic function,
\item supply approximate Hopf-Hopf points, and
\item compute the basis for the adjoint problem (which amounts to
  computing an integral that depends on the linear operator of the problem).
\end{enumerate}

\end{appendices}

\end{document}